\documentclass[11pt]{article}
 
\usepackage[margin=1in]{geometry} 
\usepackage{amsmath,amsthm,amssymb}
\usepackage{mathrsfs}
\usepackage[T1]{fontenc}
\usepackage[english]{babel}
\usepackage{graphicx}
\usepackage{color} 

\usepackage{hyperref}
\usepackage{afterpage}

\usepackage{color}

\usepackage{atbegshi}
\AtBeginDocument{\AtBeginShipoutNext{\AtBeginShipoutDiscard}}
\usepackage{pdfpages} 
\usepackage{longtable} 
\usepackage{amsmath,amsthm,amsfonts,amssymb,relsize,bm} 
\usepackage{xparse,xcolor} 
\usepackage{graphicx,float} 
\usepackage{etoolbox}
\usepackage[shortlabels]{enumitem} 
\usepackage{tikz-cd} 
\usepackage{cite} 
\usepackage[normalem]{ulem}
\usepackage{comment}
\usepackage{lipsum,mathptmx,etoolbox}
\usepackage{tocloft}

\numberwithin{equation}{section}

\theoremstyle{plain}
\newtheorem{theorem}{Theorem}[section]
\newtheorem*{theorem*}{Theorem}
\newtheorem{lemma}{Lemma}[section]
\newtheorem{corollary}{Corollary}[section]
\newtheorem{proposition}{Proposition}[section]
\newtheorem{claim}{Claim}[section]

\theoremstyle{definition}
\newtheorem{definition}{Definition}[section]
\newtheorem*{definition*}{Definition}
\newtheorem{example}{Example}[section]
\newtheorem{remark}{Remark}[section]

\begin{document}
 
 \title{On the Epstein zeta function and the zeros of a class of Dirichlet series}

\author{{Pedro Ribeiro,\quad    Semyon  Yakubovich} } 
\thanks{
{\textit{ Keywords}} :  {Dirichlet series;  Epstein zeta function; Modified Bessel functions; Selberg-Chowla Formula; Hardy's Theorem}

{\textit{2020 Mathematics Subject Classification} }: {11E45, 11M41, 44A15, 33C10}

Department of Mathematics, Faculty of Sciences of University of Porto, Rua do Campo Alegre,  687; 4169-007 Porto (Portugal). 

\,\,\,\,\,E-mail: up201403460@edu.fc.up.pt}
\date{}
 
\maketitle

\begin{abstract} By generalizing the classical Selberg-Chowla formula, we establish the analytic continuation and functional equation for a large class of Epstein zeta functions. This continuation is studied in order to provide new classes of theorems regarding the distribution of zeros of Dirichlet series in their critical lines and to produce a new method for the study of these problems. Due to the symmetries provided by the representation via the Selberg-Chowla formula, some generalizations of well-known formulas in  analytic number theory are also deduced as examples.
\end{abstract}

\tableofcontents

\newpage

\begin{center}\section{Introduction}\end{center}

Let $Q(x,y)=ax^{2}+bxy+cy^{2}$ be a real and positive definite quadratic
form. The classical Epstein zeta function is defined as the Dirichlet series \cite{epstein_I, epstein_II} 
\begin{equation}
Z_{2}(s,\,Q)=\sum_{m,n\neq0}\frac{1}{Q(m,n)^{s}},\,\,\,\,\,\text{Re}(s)>1,\label{Classical Epstein}
\end{equation}
where the notation given in the subscript, $m,\,n\neq0$, here means
that only the term $m=n=0$ is omitted from the infinite series. 

In 1949, S. Chowla and A. Selberg \cite{Selberg_Chowla I} announced the following formula for (\ref{Classical Epstein}),
valid in the entire complex plane, 
\begin{align}
a^{s}\Gamma(s)\,Z_{2}(s,\,Q) & =2\Gamma(s)\zeta(2s)+2k^{1-2s}\pi^{1/2}\Gamma\left(s-\frac{1}{2}\right)\zeta(2s-1)\nonumber \\
 & \,\,\,\,\,\,+\,8k^{\frac{1}{2}-s}\pi^{s}\sum_{n=1}^{\infty}n^{s-\frac{1}{2}}\sigma_{1-2s}(n)\,\cos\left(n\pi b/a\right)K_{s-\frac{1}{2}}\left(2\pi k\,n\right),\label{Selberg Chowla Formula}
\end{align}
where $d:=b^{2}-4ac$ is the discriminant of the quadratic form, $k^{2}:=|d|/4a^{2}$
and $\sigma_{\nu}(n)=\sum_{d|n}d^{\nu}$ is the generalized divisor
function of index $\nu$. Also, $\zeta(s)$ denotes the classical
Riemann zeta function and $K_{\nu}(z)$ is the modified Bessel function. 

According with Berndt's revision \cite{dirichletserisVI}, after its first announcement, (\ref{Selberg Chowla Formula})
was proved by Rankin \cite{Rankin}, Bateman
and Grosswald \cite{Bateman_Epstein}, Chowla and
Selberg \cite{selberg_chowla} (in a paper which
appeared 18 years after the statement of (\ref{Selberg Chowla Formula}))
and Motohashi \cite{motohashi}, although the latter
author had the main goal of proving directly the first Kronecker limit
formula. As it should be expected, all these proofs used the theory of Fourier
series and invoke directly the Poisson summation formula. 

Berndt himself \cite{dirichletserisVI} gave a very interesting proof of (\ref{Selberg Chowla Formula}),
which was based on a previous theorem due to him regarding the Bessel
expansion of the so called generalized Dirichlet series \cite{dirichlet and hecke}. In fact,
Berndt's proof of (\ref{Selberg Chowla Formula}) was considered for
a more general Epstein zeta function of the form 
\[
Z_{2}(s,\,Q;\,g,\,h)=\sum_{m,n\neq0}\frac{e^{2\pi i(h_{1}m+h_{2}n)}}{Q(m+g_{1},\,n+g_{2})^{s}},\,\,\,\,\,\,\text{Re}(s)>1,
\]
where $g_{1}$, $g_{2}$, $h_{1}$ and $h_{2}$ are real. Although
only written for the particular case (\ref{Selberg Chowla Formula}),
a simplified version of Berndt's argument was later given by Kuzumaki \cite{Kuzumaki}. It is remarkable
that Kuzumaki's proof, as well as Berndt's, depends explicitly on
the functional equation for the Hurwitz zeta function $\zeta(s,\,a)$.
Although this functional equation can be derived immediately from
the reflection formula for the $\theta-$function \cite{fine}
\begin{equation}
\frac{1}{2}+\sum_{n=1}^{\infty}e^{-\alpha n^{2}}\cos(\beta n)=\frac{1}{2}\sqrt{\frac{\pi}{\alpha}}\,\sum_{n=-\infty}^{\infty}e^{-\frac{1}{\alpha}\left(\pi n+\frac{\beta}{2}\right)^{2}},\label{Theta Hurwitz}
\end{equation}
(which is obviously connected with the Poisson summation formula, thus with all original proofs of (\ref{Selberg Chowla Formula})) it
can nevertheless be established in strikingly different ways \footnote{
see, for instance, a very recent paper by A.\ Dixit and R.\ Kumar \cite{dixit_kumar},
which established the functional equation for $\zeta(s,\,\alpha)$ via its Hermite
integral representation. See also the discussion in Remark \ref{hurwitz remark}.}. 

The generalization of (\ref{Selberg Chowla Formula}) for Epstein
zeta functions attached to quadratic forms with $n$ variables was
given by Terras \cite{terras_epstein}, whose proof employed a multidimensional
version of the Poisson summation formula. Suzuki \cite{suzuki} provided
analogues of (\ref{Selberg Chowla Formula}) for other Dirichlet series
which apriori have nothing to do with (\ref{Classical Epstein}).
These Dirichlet series include the Riemann $\zeta-$function, the
Dirichlet $L-$function and $L-$functions attached to holomorphic cusp forms.

The idea behind a Bessel expansion of a Dirichlet series of the form
(\ref{Classical Epstein}) seems to have been motivated by a formula
discovered by Watson [\cite{watson_reciprocal}, p. 99, eq. (4)], which was later generalized
by H. Kober \cite{kober_epstein} (see Example \ref{Watsonformula as Example}), and perhaps the first explicit
form of (\ref{Selberg Chowla Formula}) appearing in the literature
is due to Taylor \cite{taylor_epstein}. Although Taylor himself attributed
(\ref{Selberg Chowla Formula}) to Kober, it is in Taylor's paper
where the continuation of the Epstein zeta function via (\ref{Selberg Chowla Formula}) is made explicit
for the first time. Taylor's proof employed an elegant Mellin-Barnes representation [\cite{taylor_epstein}, p. 181]
of the left-hand side of (\ref{Selberg Chowla Formula}), involving the Gauss hypergeometric function\footnote{Although fairly unknown, Taylor's approach has found interesting applications
in some aspects of spectral theory. Since Maass cusp forms have Fourier
expansions analogous to (\ref{Selberg Chowla Formula}), a representation involving the hypergeometric function similar to Taylor's holds for these forms,
allowing to construct $\psi_{j}(z)$ from the $L-$function attached
to it, $L_{j}(s)$. This representation is one of the main results
in {[}\cite{maass_forms_zeros}, p. 120, eq. (4.3){]}, where it is used
to prove an analogue of Hardy-Littlewood's Theorem to $L-$functions
attached to Maass cusp forms. }.
Although Taylor did not prove the functional equation for $Z_{2}(s,\,Q)$
directly from (\ref{Selberg Chowla Formula}), as Selberg and Chowla
did, his derivation of both results used the same ideas and this seems
to be the first indication of the connection between (\ref{Selberg Chowla Formula})
and the symmetric properties of the continuation of $Z_{2}(s,\,Q)$. 

\medskip{}

The Selberg-Chowla formula has found many applications in Number Theory.
By using it, Selberg and Chowla \cite{selberg_chowla} proved the existence
of a real zero for $Z_{2}(s,\,Q)$ in the open interval $(\frac{1}{2},\,1)$,
when the quadratic form is $Q(x,\,y)=x^{2}+cy^{2}$ and $c$ is large
enough. This result was later improved by Bateman and Grosswald [\cite{Bateman_Epstein},
p. 367, Thm 3.], who have shown that, for a general positive definite and real quadratic
form $Q$, $Z_{2}(s,\,Q)$ has a real zero between $\frac{1}{2}$
and $1$ if $k:=\sqrt{|d|}/2a>7.0556$. A generalization of this argument to Epstein zeta functions attached to ternary quadratic forms was furnished by Terras \cite{terras_ternary}, who employed prior results due to her concerning a multidimensional version of (\ref{Selberg Chowla Formula}) (see [\cite{terras_epstein}, p. 480, eq. (2.3)] and (\ref{Selberg Chowla diagonal Multidimensional}) below).  

With the aid of (\ref{Selberg Chowla Formula}), Stark \cite{stark_zeros_epstein}
and Fujii \cite{fujii_I} gave very detailed descriptions about the distribution
of the complex zeros of $Z_{2}(s,\,Q)$ in bounded regions. Building on the work of many previous authors, H. Ki \cite{ki_all_but} proved a remarkable result regarding the zeros of each finite truncation of the right-hand side of (\ref{Selberg Chowla Formula}), showing that all but finitely many of them lie on the critical line $\text{Re}(s)=1/2$. The idea of studying the zeros of the Epstein zeta function via equivalent
representations to (\ref{Selberg Chowla Formula}) seems to have its origin in Deuring's work \cite{deuring_1935}. 

In a very interesting paper by the same author \cite{deuring}, in which most of the
approaches later used to prove (\ref{Selberg Chowla Formula}) were already given\footnote{For example, a closer look at Deuring's argument shows also that the
fundamental ideas employed by Motohashi in his proof of the Kronecker limit formula
(see [\cite{motohashi}, p. 615] and the evaluation of (5) by using the
Cahen-Mellin representation (\ref{Cahen Mellin integral for applications}))
were already been used by Deuring [\cite{deuring}, p. 587, eq. (5), (6)].}, Deuring connected the Fourier expansion for the simple Epstein $\zeta-$function
$Z_{2}(s,\,\mathbb{I}_{2})$ (where $b=0$, $a=c=1$) with the existence
of infinitely many zeros of $\zeta(s)$ on the critical line $\text{Re}(s)=\frac{1}{2}$.
Using this relation between two drastically different Dirichlet series, he was able to deduce Hardy's Theorem for $\zeta(s)$.
Although Deuring's proof of Hardy's Theorem requires more motivation
than the standard ones given by Hardy \cite{hardy_note}, Landau \cite{landau_hardy} or Fekete \cite{fekete_zeros}, it is nevertheless an interesting
proof, since the way it seeks the contradiction is implied from a subconvex 
estimate for the Epstein zeta function on the critical line, due to E.\ C.\ Titchmarsh \cite{epstein_titchmarsh}, 
\begin{equation}
Z_{2}\left(\frac{1}{2}+it,\,\mathbb{I}_{2}\right)=O\left(|t|^{\frac{1}{3}+\epsilon}\right),\label{Connection Deuring}
\end{equation}
which in its turn is analogous to the classical Van der Corput \footnote{As indicated in (\ref{condition critical line for r=00005Cleq1}) in Theorem \ref{deuring to hold}, we do not need the full strength of Titchmarsh's estimate (\ref{Connection Deuring}) when the Dirichlet series $\phi(s)$ is the Riemann zeta function.} estimate
for $\zeta\left(\frac{1}{2}+it\right)$ \cite{ivic}.

In the most part of the classical proofs of Hardy's Theorem, one arrives
at a contradiction by using the reflection formula for the Jacobi
$\theta-$function, which ensures a suitable decay for $\theta\left(e^{i\left(\frac{\pi}{4}-\epsilon\right)}\right)$,
as $\epsilon\rightarrow0^{+}$ \cite{landau_hardy, hardy_note}. Standard variations of these arguments \cite{chadrasekharan_narasimhan_ideal classes, Titchmarsh_Potter, selberg class hardy}
employ the evaluation of certain exponential integrals and sums, by contrasting
the behavior of the integrals $\int_{T}^{2T}\left|\xi\left(\frac{1}{2}+it\right)\right|\,dt$
and $\left|\int_{T}^{2T}\xi\left(\frac{1}{2}+it\right)\,dt\right|$ as $T\rightarrow \infty$. 

Although much more difficult to ensure than the asymptotic properties
of Jacobi's $\theta-$function, a condition like (\ref{Connection Deuring})
is deep since it connects two fundamental properties of two distinct Dirichlet series. 

From the correspondence between the zeros of $\zeta(s)$ and the order
of magnitude of $Z_{2}(s,\,\mathbb{I}_{2})$, motivated by the Selberg-Chowla
formula, one gains an additional degree of freedom, which is the separate study of the properties of the latter Dirichlet series.

\medskip{}

Our goal in this paper is to study this additional degree of freedom in a sufficiently general framework. One of the first steps towards a generalization of Deuring's proof is to extend the scope of the Selberg-Chowla
formula (\ref{Selberg Chowla Formula}). We start by making a distinction between two cases: 
\begin{enumerate}
\item \textbf{The diagonal Epstein zeta function}: for $\text{Re}(s)$ sufficiently
large, one of the goals of this paper is to study an infinite series
of the form
\[
\sum_{m,n\neq0}^{\infty}\frac{a_{1}(m)\,a_{2}(n)}{(\lambda_{m}+\lambda_{n}^{\prime})^{s}},
\]
where $a_{1}(m)$, $a_{2}(n)$, $\lambda_{m}$
and $\lambda_{n}^{\prime}$ will be specified later. This double series \footnote{Throughout this paper we shall use the term "double series" or "multiple series" to denote Dirichlet series attached to more than one arithmetical function. This slightly contrasts with the standard terminology, in which these terms are applied to Dirichlet series with more than one complex variable.} was
also studied in Berndt's paper \cite{dirichletserisVI} and a formula
of Selberg-Chowla type for it has been indicated there. Under additional
assumptions, in section \ref{section 2} we give two formulas analogous to (\ref{Selberg Chowla Formula})
and use them to arrive at a functional equation for this double infinite series.  
\item \textbf{The non-diagonal Epstein zeta function}: analogously to the previous
item, for $\text{Re}(s)$ sufficiently large, we study a double infinite
series of the form 
\[
\sum_{m,n\neq0}\frac{a_{1}(m)\,a_{2}(n)}{Q(\lambda_{m},\lambda_{n}^{\prime})^{s}},
\]
where $Q(x,y)$ denotes, just as in the classical case (\ref{Classical Epstein}), a real and positive definite Quadratic Form. We study some conditions ensuring the analytic continuation and functional
equation of such double series. 
\end{enumerate}

The study of the double series outlined above is made in Sections
\ref{section 2} and \ref{section 3} respectively. By proving analogues of the Selberg-Chowla formula (\ref{Selberg Chowla Formula}), we establish functional equations for these series. In section \ref{section 4} we use the obtained Selberg-Chowla formulas to generalize Deuring's argument \cite{deuring} to a Class of Dirichlet series. In section \ref{section 5} we furnish particular examples of diagonal and non-diagonal Epstein zeta functions and a description of their analytic continuation is given. 


\medskip{}

Most of our principal results are given in section \ref{section 4} and generalize the Theorems presented in \cite{berndt_zeros_(i), berndt_zeros_(ii), Hecke_middle_line}, where analogues of Hardy's Theorem are proved for large classes of Dirichlet series. The corollaries derived from our main Theorem \ref{deuring to hold} contain also some of the key observations
used in most of the well-known proofs of Hardy's Theorem. In fact, by simplifying Deuring's argument for $\zeta(s)$, we shall see how his proof of Hardy's Theorem, as well as our general and "inductive" approach, extends to all the other proofs available in the literature, which usually involve the asymptotic behavior of the Jacobi $\theta-$function or classical estimates of some exponential integrals \cite{titchmarsh_zetafunction, hardy_note, landau_hardy, fekete_zeros, Titchmarsh_Potter, selberg class hardy, chadrasekharan_narasimhan_ideal classes}. 

Although written in the general setting of generalized Dirichlet series,
the method used by us to prove the existence of infinitely many zeros
of $\phi(s)$ on the line $\text{Re}(s)=\frac{r}{2}$ seems to possess
certain advantages. The first of these is its uniform applicability
to a wide class of Dirichlet series, allowing to generalize certain results due to Berndt and Hecke \cite{berndt_zeros_(ii), Hecke_middle_line}
(see Corollaries \ref{corollary on theta} and \ref{corollary the one with residue}). A simple application of it also allowed us to reprove a known
theorem due to Berndt \cite{berndt_zeros_(i)} (see corollary \ref{narrow abcissa corollary}) via ``elementary'' means, as well as
to prove that a certain combination of bounded shifts of $\Gamma(s)\,\phi(s)$
has infinitely many zeros on the line $\text{Re}(s)=\frac{r}{2}$, extending a result by A. Dixit, N. Robles, A. Roy and A. Zaharescu \cite{combinations_dixit}.  

The second advantage is the fact that it provides a certain degree of freedom
in applications, since in order to check that Hardy's theorem is valid for
a particular Dirichlet series, it suffices to verify (independently)
that it holds for at least one (out of the infinitely many!) Epstein
zeta functions attached to it. As proved by Suzuki \cite{suzuki}, it is possible
to obtain analogues of the Selberg-Chowla formula for $\zeta(s)$
(or $L(s,\,\chi)$ and $L(s,\,a)$), so that we can find Dirichlet
series to which $\zeta(s)$ (or $L(s,\,\chi)$ and $L(s,\,a)$) plays
the same role as the Epstein zeta function (\ref{Classical Epstein}) does with $\zeta(s)$.
Thus, the results described by Theorem \ref{deuring to hold} and Corollary \ref{quantitative Classic}
can be employed, for example, to study the zeros of the Dirichlet appearing in Suzuki's paper. \bigskip{}

As a rather simple example of our method, we will see that the final conclusion of Hardy's Theorem for $\zeta(s)$ is implied by the celebrated Jacobi's 4-square Theorem, which constitutes a curious connection between a purely analytic theorem involving deep properties of $\zeta(s)$ and a beautiful arithmetical property of $r_{4}(n)$ (see Examples \ref{hardy and jacobi} and \ref{epstein example}). 

\begin{center}\subsection{Notation and Definitions}\label{notation and def to refer} \end{center} 

To state the two generalizations of the Selberg-Chowla formula given
in this paper, we need the following definition: 

\begin{definition} \label{definition 1.1.}
Let $\left(\lambda_{n}\right)_{n\in\mathbb{N}}$ and $\left(\mu_{n}\right)_{n\in\mathbb{N}}$
be two sequences of positive numbers strictly increasing to $\infty$
and $\left(a(n)\right)_{n\in\mathbb{N}}$ and $\left(b(n)\right)_{n\in\mathbb{N}}$
two sequences of complex numbers not identically zero. Consider the
functions $\phi(s)$ and $\psi(s)$ representable as Dirichlet series 
\begin{equation}
\phi(s)=\sum_{n=1}^{\infty}\frac{a(n)}{\lambda_{n}^{s}}\,\,\,\,\,\,\text{and }\,\,\,\,\,\psi(s)=\sum_{n=1}^{\infty}\frac{b(n)}{\mu_{n}^{s}}\label{representable as Dirichlet series in first definition ever}
\end{equation}
with finite abcissas of absolute convergence $\sigma_{a}$
and $\sigma_{b}$ respectively. Let $\Delta(s)$ denote
one of the following three gamma factors: $\Gamma(s)$, $\Gamma\left(\frac{s}{2}\right)$
and $\Gamma\left(\frac{s+1}{2}\right)$ and $r$ be an arbitrary positive
real number in the first case and $1$ in the other two. We say that
$\phi$ and $\psi$ satisfy the functional equation 
\begin{equation}
\Delta(s)\,\phi(s)=\Delta(r-s)\,\psi(r-s),\label{functional equation multi}
\end{equation}
if there exists a meromorphic function $\chi(s)$ with the following
properties:
\begin{enumerate}
\item $\chi(s)=\Delta(s)\,\phi(s)$ for $\text{Re}(s)>\sigma_{a}$
and $\chi(s)=\Delta(r-s)\,\psi(r-s)$ for $\text{Re}(s)<r-\sigma_{b}$;
\item $\lim_{|\text{Im}(s)|\rightarrow\infty}\chi(s)=0$ uniformly in every
interval $-\infty<\sigma_{1}\leq\text{Re}(s)\leq\sigma_{2}<\infty$. 
\item The singularities $\chi(s)$ are at most poles and are confined to
some compact set. 
\end{enumerate}
Moreover, if $\Delta(s)=\Gamma\left(\frac{s+\delta}{2}\right)$, $\delta\in\{0,1\}$,
we take a convention which extends $a(n)$ to the negative
integers as follows: $a(-n)=(-1)^{\delta}a(n)$.
\end{definition}

We say that the pair of functions $\left(\phi,\,\psi\right)$ representable
as Dirichlet series (\ref{representable as Dirichlet series in first definition ever})
satisfy \textbf{Hecke's functional equation} if they satisfy the conditions
of the previous definition for $\Delta(s)=\Gamma(s)$. The particular
case of (\ref{functional equation multi}) reads
\begin{equation}
\Gamma(s)\,\phi(s)=\Gamma(r-s)\,\psi(r-s),\,\,\,\,\,\,r>0.\label{Hecke Dirichlet series Functional}
\end{equation}

Similarly, we say that the pair of functions $\left(\phi,\,\psi\right)$ representable
as (\ref{representable as Dirichlet series in first definition ever})
with finite abcissas of absolute convergence $\sigma_{a}$
and $\sigma_{b}$ satisfy \textbf{Bochner's functional equation} if
they satisfy the functional equation 
\begin{equation}
\Gamma\left(\frac{s+\delta}{2}\right)\phi(s)=\Gamma\left(\frac{1+\delta-s}{2}\right)\,\psi\left(1-s\right),\label{This is the first Bochner ever}
\end{equation}
in the sense of Definition \ref{definition 1.1.}. Moreover, if the pair $(\phi,\,\psi)$ satisfies
(\ref{This is the first Bochner ever}) with $\delta=0$, then we
will say that $\phi(s)$ and $\psi(s)$ are even Dirichlet series
belonging to the Bochner class. Otherwise, if (\ref{This is the first Bochner ever}) is satisfied for $\delta=1$, we will say that $\phi(s)$ and $\psi(s)$ are odd Dirichlet series belonging to the Bochner class. 


\bigskip{}

Although under slight variations, the class of Dirichlet series covered by Definition
\ref{definition 1.1.} was considered by a large number of authors. The Dirichlet series
satisfying the functional equation (\ref{This is the first Bochner ever})
were studied by S. Bochner and K. Chandrasekharan in \cite{bochner_chrandrasekharan},
where considerations of general analogues of Hamburger's theorem were given.
K. Chandrasekharan and R. Narasimhan \cite{arithmetical identities} proved the equivalence
between several identities in Number theory and Hecke's functional
equation (\ref{Hecke Dirichlet series Functional}).  Included as a subclass in the context of generalized Dirichlet series are the classical Dirichlet series with signature $(\lambda,\, r,\, \gamma)$, considered
in Hecke's works \cite{dirichlet and hecke}.
\\

Before proceeding further, we shall denote $\sum_{m,n\neq0}$ as the infinite sum over all integers
$m,\,n$ not simultaneously zero. If we assume that $m$ and $n$
are positive, we shall write $\sum_{m,n\neq0}^{\infty}$ instead and we will always use the convention $\lambda_{0}=\mu_{0}=0$ and $\lambda_{-n}=-\lambda_{n}$, $\mu_{-n}=-\mu_{n}$. Also, we will always write $\sigma_{a}:=\max\left\{ \sigma_{a_{1}},\,\sigma_{a_{2}}\right\}$ and, analogously, $\sigma_{b}:=\max\left\{ \sigma_{b_{1}},\,\sigma_{b_{2}}\right\}$.   

We will now define a double Dirichlet series which exhibits a similar
behavior as the classical Epstein zeta function (\ref{Classical Epstein}). We will often call
it by the name ``diagonal Epstein zeta function''. 

\begin{definition} \label{definition double diagonal}
Let $\left(\lambda_{n},\,\lambda_{n}^{\prime}\right)_{n\in\mathbb{N}}$
be a pair of sequences of positive numbers strictly increasing to
$\infty$ and $\left(a_{1}(n),\,a_{2}(n)\right)_{n\in\mathbb{N}}$
a pair of sequences of complex numbers not identically zero. Also,
let $\phi_{1}(s)$ and $\phi_{2}(s)$ be two functions representable
as Dirichlet series 
\begin{equation}
\phi_{1}(s)=\sum_{n=1}^{\infty}\frac{a_{1}(n)}{\lambda_{n}^{s}},\,\,\,\,\,\,\phi_{2}(s)=\sum_{n=1}^{\infty}\frac{a_{2}(n)}{\lambda_{n}^{\prime s}},\label{Dirichlet Series in Definition}
\end{equation}
with finite abcissas of absolute convergence $\sigma_{a_{1}}$
and $\sigma_{a_{2}}$ respectively. For the pair $(\phi_{1},\,\phi_{2})$,
we define the generalized diagonal Epstein zeta function as the double
Dirichlet series 
\begin{equation}
\mathcal{Z}_{2}\left(s,\,a_{1},\,a_{2},\,\lambda,\,\lambda^{\prime}\right)=\sum_{m,n\neq0}^{\infty}\frac{a_{1}(m)\,a_{2}(n)}{\left(\lambda_{m}+\lambda_{n}^{\prime}\right)^{s}},\,\,\,\,\,\text{Re}(s)>2\,\max\left\{ \sigma_{a_{1}},\,\sigma_{a_{2}}\right\} .\label{Epstein as double series}
\end{equation}
\end{definition}
\bigskip{}

It is easily seen that the double series (\ref{Epstein as double series})
is absolutely convergent if $\text{Re}(s)>2\,\max\left\{ \sigma_{a_{1}},\,\sigma_{a_{2}}\right\}=2\sigma_{a}$.
Examples of (\ref{Epstein as double series}) include the classical
Dirichlet associated with the sum of two squares 
\begin{equation}
\frac{1}{4}\,\zeta_{2}(s)=\frac{1}{4}\,\sum_{n=1}^{\infty}\frac{r_{2}(n)}{n^{s}}=\sum_{m,n\neq0}^{\infty}\frac{1}{(m^{2}+n^{2})^{s}},\,\,\,\,\text{Re}(s)>1\label{sum of two squres funnnnction dirichhhhlet sssseeeeeiiiiires}
\end{equation}
which can be constructed if we take $\phi_{1}(s)=\phi_{2}(s)=\zeta(2s)$. 

Just as (\ref{sum of two squres funnnnction dirichhhhlet sssseeeeeiiiiires}),
we may represent a general diagonal Epstein $\mathcal{Z}_{2}\left(s;\,a_{1},\,a_{2};\,\lambda,\,\lambda^{\prime}\right)$
via the single Dirichlet series 
\begin{equation}
\mathcal{Z}_{2}\left(s,\,a_{1},\,a_{2},\,\lambda,\,\lambda^{\prime}\right)=\sum_{m,n\neq0}^{\infty}\frac{a_{1}(m)\,a_{2}(n)}{\left(\lambda_{m}+\lambda_{n}^{\prime}\right)^{s}}=\sum_{n=1}^{\infty}\frac{\mathfrak{U}_{2}(n)}{\Lambda_{n}^{s}},\,\,\,\,\,\,\text{Re}(s)>2\,\sigma_{a}\label{representation Epstein with arithmetica}
\end{equation}
where $\mathfrak{U}_{2}(n)$ generalizes the arithmetical function
$r_{2}(n)$ in the following way: 
\begin{equation}
\mathfrak{U}_{2}(n)=\sum_{j_{1},j_{2}\,:\lambda_{j_{1}}+\lambda_{j_{2}}^{\prime}=\Lambda_{n}}a_{1}(j_{1})\,a_{2}(j_{2}),\label{definition general 2 square}
\end{equation}
and the sequence $\Lambda_{n}$ is taken as $\Lambda_{n}:=\left(\lambda_{j_{1}}+\lambda_{j_{2}}^{\prime}\right)_{j_{1},j_{2}\in\mathbb{N}_{0}}$
and then rearranged in increasing order. 

Since (\ref{sum of two squres funnnnction dirichhhhlet sssseeeeeiiiiires})
constitutes a particular case of the Epstein zeta function $Z_{2}(s,\,Q)$
when $Q$ is a diagonal quadratic form, we shall call to the double
series (\ref{Epstein as double series}) by the name ``diagonal Epstein
zeta function''. In section \ref{section 2} we prove that, if $\phi_{1}$ and $\phi_{2}$
satisfy definition \ref{definition 1.1.} with $\Delta(s)=\Gamma(s)$, then it is possible
to write two distinct Selberg-Chowla formulas for (\ref{Epstein as double series}).
We will also introduce a subclass of Dirichlet, named class $\mathcal{A}$,
for which (\ref{Epstein as double series}) will also obey to Hecke's
functional equation (see definition \ref{class A definition}). 
\bigskip{}

In section \ref{section 3} we will study the analytic continuation of the following
double Dirichlet series: 
\begin{definition} \label{definition non-diagonal epstein}
Let $Q$ be a real, binary and positive definite quadratic form given
by $Q(x,\,y)=ax^{2}+bxy+cy^{2}$ and $(\phi_{1},\,\phi_{2})$ a pair
of Dirichlet series similar to (\ref{Dirichlet Series in Definition}) with abcissas
of absolute convergence $\sigma_{a_{1}}$ and $\sigma_{a_{2}}$.
Assume also that their arithmetical functions $a_{i}(n)$
can be extended to the negative integers as $a_{i}(n)=(-1)^{\delta}a_{i}(n)$,
$\delta\in\{0,\,1\}$, and take the convention $\lambda_{0}=\lambda^{\prime}_{0}=0$ and $\lambda_{-n}=-\lambda_{n}$, $\lambda^{\prime}_{-n}=-\lambda^{\prime}_{n}$.

We define the non-diagonal Epstein zeta function associated with $\phi_{1}$
and $\phi_{2}$ as the double Dirichlet series 
\begin{equation}
\mathcal{Z}_{2}\left(s;\,Q;\,a_{1},\,a_{2};\,\lambda,\,\lambda^{\prime}\right)=\sum_{m,n\neq0}\frac{a_{1}(m)\,a_{2}(n)}{\left(a\lambda_{m}^{2}+b\,\lambda_{m}\lambda_{n}^{\prime}+c\lambda_{n}^{\prime2}\right)^{s}},\,\,\,\,\text{Re}(s)>\sigma_{a}:=\max\{\sigma_{a_{1}},\,\sigma_{a_{2}}\}\label{Epstein zeta function non diagonal-1}
\end{equation}
where the infinite series is taken over the pair of integers which
are not both zero. 

\end{definition}

It is easily seen that the double Dirichlet series on the right-hand
side converges absolutely if $\text{Re}(s)>\max\left\{ \sigma_{a_{1}},\,\sigma_{a_{2}}\right\} $. We will always adopt the convention that $a_{1}(n)$ and
$a_{2}(n)$ have the same parity since, otherwise, the
previous double series would be identically zero. 

In section \ref{section 3} we prove analogues of the Selberg-Chowla formula for (\ref{Epstein zeta function non diagonal-1}) when $\phi_{1}$ and $\phi_{2}$ satisfy definition \ref{definition 1.1.} with $\Delta(s)=\Gamma\left(\frac{s+\delta}{2}\right)$. 

\medskip{}



Before giving the preliminary results, we remark that throughout this paper we shall use the following notation and conventions:
\begin{itemize}
\item Let $Q$ denote the matrix representation of the Quadratic form $Q(x,\,y)$
given in definition \ref{definition non-diagonal epstein}. If $d$ is the discriminant of $Q$, we will
define the inverse quadratic form as $Q^{-1}(\mathbf{x})=-\frac{d}{4}\,\mathbf{x}^{T}\,Q^{-1}\,\mathbf{x}$.
Note that $Q^{-1}(x,\,y)=cx^{2}-bxy+ay^{2}$. 
\item In order to make our general formulas resemble the original formulation
(\ref{Selberg Chowla Formula}), we will use the symbol $\sigma_{z}(\nu_{j};\,b_{1},\,a_{2};\,\mu,\,\lambda^{\prime})$
to denote a generalized version of the divisor function (see Definition
\ref{definition generalized divisor 2.1.}). At some points (see Example \ref{Watsonformula as Example}), we will simplify the notation and reduce it
to $\sigma_{z}(\nu_{j};\,b_{1},\,a_{2})$. 
\item Whenever we use the term ``critical line'' for a given Dirichlet
series $\phi(s)$ satisfying Definition \ref{definition 1.1.} we will be referring to $\text{Re}(s)=\frac{r}{2}$
if $\phi(s)$ satisfies Hecke's functional equation (\ref{Hecke Dirichlet series Functional})
and to $\text{Re}(s)=\frac{1}{2}$ if $\phi(s)$ satisfies Bochner's
functional equation (\ref{This is the first Bochner ever}). The property
``$\phi(s)$ has infinitely many zeros at its critical line'' will
sometimes be reformulated as ``$\phi(s)$ satisfies Hardy's Theorem''. 
\item We will write $H_{r_{1}}\left(s;\,b_{1},\,a_{2};\,\mu,\,\lambda^{\prime}\right)$
and $H_{r_{2}}(s;\,b_{2},\,a_{1};\,\mu^{\prime},\,\lambda)$
to denote the entire complex functions (\ref{H1}) and (\ref{H2})
appearing in the Selberg-Chowla representations for the diagonal Epstein
zeta function. For the entire complex functions appearing in (\ref{First Selberg Chowla even non diagonal})
and (\ref{First Selberg Chowla odd}), we shall use alternatively
$H_{1}\left(s;\,Q;\,b_{1},\,a_{2};\,\mu,\,\lambda^{\prime}\right)$
and $H_{2}\left(s;\,Q;\,b_{2},\,a_{1},\,\mu^{\prime},\,\lambda\right)$.
In analogy with the generalized divisor function, we will sometimes
simplify the notation by writing these entire functions as $H_{1}\left(s;\,Q;\,b_{1},\,a_{2}\right)$
or $H_{2}\left(s;\,Q;\,b_{2},\,a_{1}\right)$. 
\item If, in Definitions \ref{definition double diagonal} and \ref{definition non-diagonal epstein}, the Dirichlet series are such that $\phi_{1}=\phi_{2}=\phi$, we will
write $\mathcal{Z}_{2}(s;\,a;\,\lambda)$ and $\mathcal{Z}_{2}(s;\,Q;\,a;\,\lambda)$
to respectively denote (\ref{Epstein as double series}) and (\ref{Epstein zeta function non diagonal-1}). 
\item  We shall use $R_{\phi}(s)$ and $R_{\psi}(s)$ to denote respectively 
$\Gamma(s)\,\phi(s)$ and $\Gamma(s)\,\psi(s)$. This notation will be employed mainly during the proof of Theorem \ref{deuring to hold}. Under some conditions given there, $R_{\phi}(s)$ and $R_{\psi}(s)$ will represent real functions on the critical line $\text{Re}(s)=r/2$. 
\item Throughout the proof of Theorem \ref{deuring to hold}, $c$ will represent the least positive integer such that $a(c)\neq0$, where $a(n)$ is the arithmetical function attached to a Dirichlet series satisfying definition \ref{definition 1.1.}. This integer also appears in the statement of Corollary \ref{corollary the one with residue}.
\end{itemize}

\begin{center}\subsection{Preliminary results}\end{center}

In several occasions throughout this paper,
we shall need to estimate the asymptotic order of certain integrals
involving a combination of Gamma functions and Dirichlet series, $\phi(s)$,
satisfying Definition \ref{definition 1.1.}. To justify most of the steps, we will often
invoke the following version of Stirling's formula 
\begin{equation}
\Gamma(\sigma+it)=(2\pi)^{\frac{1}{2}}\,t^{\sigma+it-\frac{1}{2}}\,e^{-\frac{\pi t}{2}-it+\frac{i\pi}{2}(\sigma-\frac{1}{2})}\left(1+1/12(\sigma+it)+O\left(1/t^{2}\right)\right),\label{Stirling exact form on Introduction}
\end{equation}
as $t\rightarrow\infty$, uniformly for $-\infty<\sigma_{1}\leq\sigma\leq\sigma_{2}<\infty$.
A similar formula can be written for $t<0$ as $t$ tends to $-\infty$
by using the fact that $\Gamma(\overline{s})=\overline{\Gamma(s)}$. 

\medskip{}

To estimate the order of $\phi(s)$ at the line $\text{Re}(s)=\sigma$
we shall need a version of the classical Phragm\'en-Lindel\"of theorem
given in {[}\cite{titchmarsh_theory_of_functions}, p. 180, 5.65{]}. Since, for
$\sigma>\sigma_{b}$, $\psi(\sigma+it)=O(1)$ as $|t|\rightarrow\infty$,
it follows from the functional equation (\ref{functional equation multi})
and Stirling's formula (\ref{Stirling exact form on Introduction})
that, for $\sigma<r-\sigma_{b}$, $\phi(s)$ satisfies
\begin{align}
\phi(\sigma+it) & =O\left(\frac{\Delta(r-s)}{\Delta(s)}\,\psi(r-s)\right)=O\left(\frac{\Delta(r-s)}{\Delta(s)}\right)\nonumber \\
 & =\begin{cases}
O\left(|t|^{r-2\sigma}\right) & \text{if }\Delta(s)=\Gamma(s)\\
O\left(|t|^{\frac{1-2\sigma}{2}}\right) & \text{if }\Delta(s)=\Gamma\left(\frac{s+\delta}{2}\right)
\end{cases},\label{order phragmen introduction}
\end{align}
as $|t|$ tends to $\infty$. Since $\phi(s)=O(1)$ for $\sigma=\text{Re}(s)>\sigma_{a}$,
it follows from property 2. in Definition \ref{definition 1.1.} that: 
\begin{enumerate}
\item If $\phi(s)$ satisfies Hecke's functional equation (\ref{Hecke Dirichlet series Functional}),
then for any $\delta>0$ and $r-\sigma_{a}-\delta\leq\sigma:=\text{Re}(s)\leq\sigma_{a}+\delta$,
\begin{equation}
\phi(\sigma+it)=O\left(|t|^{\sigma_{a}+\delta-\sigma}\right),\,\,\,\,|t|\rightarrow\infty.\label{Lindelof Phragmen for Hecke}
\end{equation}
\item If $\phi(s)$ satisfies Bochner's functional equation (\ref{This is the first Bochner ever}),
then for any $\epsilon>0$ and $1-\sigma_{a}-\epsilon\leq\sigma:=\text{Re}(s)\leq\sigma_{a}+\epsilon$,
\begin{equation}
\phi(s)=O\left(|t|^{\frac{\sigma_{a}+\epsilon-\sigma}{2}}\right),\,\,\,\,\,\,|t|\rightarrow\infty.\label{Lindelof Pgragmen for Bochner}
\end{equation}
\end{enumerate}

Since Hardy's short note \cite{hardy_note} and the subsequent quantitative
proofs given by Landau, Fekete, Hardy and Littlewood \cite{ landau_handbuch, hardy_littlewood_contributions, fekete_zeros}, as well as later
adaptations to other Dirichlet series by Kober \cite{kober_zeros} and Hecke \cite{Hecke_middle_line}, the study
of the zeros of $\zeta(s)$ and other general Dirichlet series is
often dependent on the asymptotic behavior of the associated Jacobi
$\theta-$function as a function in the upper half-plane. From now on, if $\phi(s)$ satisfies Hecke's functional
equation (\ref{Hecke Dirichlet series Functional}) in the sense of Definition \ref{definition 1.1.} and it is representable by
the first Dirichlet series in (\ref{representable as Dirichlet series in first definition ever}),
we shall use $\Theta(z;\,a;\,\lambda)$ to denote the generalized
$\theta-$function 
\begin{equation}
\Theta(z;\,a;\,\lambda):=\sum_{n=1}^{\infty}a(n)\,e^{-\lambda_{n}z},\,\,\,\,\,\,\text{Re}(z)>0.\label{definition generalized Theta function nnnn.}
\end{equation}

Following Bochner \cite{bochner_modular_relations}, for $\text{Re}(z)>0$, let $P(z)$ denote the residual function \cite{dirichletserisIII}
\begin{equation}
P(z)=\frac{1}{2\pi i}\,\int_{C}\chi(s)\,z^{-s}ds,\label{residual function definition Bochner modular}
\end{equation}
where $C$ denotes a curve, or curves, encircling the singularities
of $\chi (s)$ given in Definition \ref{definition 1.1.}. It was proved by Bochner \cite{bochner_modular_relations}
that Hecke's functional equation for $\phi(s)$ and $\psi(s)$ (\ref{Hecke Dirichlet series Functional})
and the modular relation 
\begin{equation}
\sum_{n=1}^{\infty}a(n)\,e^{-\lambda_{n}z}=z^{-r}\sum_{n=1}^{\infty}b(n)\,e^{-\mu_{n}/z}+P(z).\label{Bochner Modular relation at intro}
\end{equation}
are equivalent. It is interesting that the observation of this equivalence
had its genesis in B. Riemann's revolutionary memoir, where
one of the implications was proved for the first time. There, a second
proof of the functional equation for $\pi^{-s}\zeta(2s)$ (which satisfies
Hecke's functional equation with $r=\frac{1}{2}$) was given by observing
the reflection formula for the classical Jacobi $\theta-$function,
which is a particular case of (\ref{Bochner Modular relation at intro}). 

The converse is obtained upon the use of the Cahen-Mellin integral {[}\cite{hardy_littlewood_contributions},
p. 120, eq. (1.II){]}
\begin{equation}
e^{-z}=\frac{1}{2\pi i}\,\int_{c-i\infty}^{c+i\infty}\Gamma(s)\,z^{-s}ds,\label{Cahen Mellin integral for applications}
\end{equation}
valid for $c>0$, $\text{Re}(z)>0$ and $z^{-s}$ having its principal
value. 
\\

We will also need a slightly more general version of (\ref{Bochner Modular relation at intro})
when $\phi(s)$ satisfies one of Bochner's functional equations
(\ref{This is the first Bochner ever}). This will be done in Lemma \ref{lemma 3.1} of the paper. In order to establish it, we will require the representation of the confluent hypergeometric function $_{1}F_{1}$ as the following Mellin transform {[}\cite{ryzhik},
p. 503, eq. 3.952.8{]}, valid for $\text{Re}(\alpha)>0$, $\text{Re}(s)>0$
and $\beta\in\mathbb{C}$, 
\begin{equation}
\int_{0}^{\infty}x^{s-1}\,e^{-\alpha x^{2}}\cos(\beta x)\,dx=\frac{\alpha^{-\frac{s}{2}}e^{-\frac{\beta^{2}}{4\alpha}}}{2}\,\Gamma\left(\frac{s}{2}\right)\,_{1}F_{1}\left(\frac{1-s}{2};\,\frac{1}{2};\,\frac{\beta^{2}}{4\alpha}\right).\label{invoking a relation on Erdeliy table not ryzhik}
\end{equation}

Since the real function $e^{-\alpha x^{2}}\cos(\beta x)$ is smooth,
we are allowed to integrate the left-hand side of (\ref{invoking a relation on Erdeliy table not ryzhik})
by parts an arbitrary number of times. Doing so, we see that its asymptotic
estimate must be of the form $O(|s|^{-N})$ as $|s|\rightarrow\infty$
for any $N\in\mathbb{N}$. A simple argument by continuation shows that an estimate like this holds for every $s\in \mathbb{C}$. This gives a suitable decaying behavior for
the right-hand side of (\ref{invoking a relation on Erdeliy table not ryzhik})
which will be useful in the proof of Lemma \ref{lemma 3.1}. Of course, stronger estimates for the right-hand side of (\ref{invoking a relation on Erdeliy table not ryzhik}) could be given by invoking asymptotic expansions for the Whittaker function \cite{ryzhik}, but since we will not need these expansions we shall not write them. 

Since we will apply Mellin's inversion formula to the right-hand side
of (\ref{invoking a relation on Erdeliy table not ryzhik}) (see eq. (\ref{Mellin inverse 1F1}) below), the confluent hypergeometric function
will have to satisfy a reflection formula compatible with Bochner's
functional equation (\ref{This is the first Bochner ever}). As we shall see, the right symmetries are expected due to Kummer's transformation formula for $_{1}F_{1}$ {[}\cite{roy}, p. 191, eq. (4.1.11){]}
\begin{equation}
_{1}F_{1}\left(a;\,c;\,x\right)=e^{x}\,_{1}F_{1}\left(c-a;\,c;\,-x\right).\label{Kummer confluent transformation}
\end{equation}

In the sequel, we will employ several times well-known integral representations
for the modified Bessel function $K_{\nu}(z)$. For a matter of clarity
in our exposition, we will invoke them solely when they are needed.

\begin{center}\section{The Selberg-Chowla formula and the analytic continuation of the Diagonal Epstein zeta function} \label{section 2} \end{center}


Before stating the properties of the special class $\mathcal{A}$, for which functional equations will be given, 
we will consider the following general Theorem, which establishes
the analytic continuation of the Diagonal Epstein zeta function (\ref{Epstein as double series}). As usual, the 
following result provides a formula valid in a region contained in the plane of absolute
convergence of the Double series (\ref{Epstein as double series}). However, as one should expect, they may be taken to other regions of $\mathbb{C}$ by analytic continuation.
\bigskip{}

\begin{theorem}[A general Selberg-Chowla formula {[}\cite{dirichletserisVI}, p. 166 {]}] \label{theorem 2.1}
Let $\phi_{1}$ and $\phi_{2}$ be the pair of Dirichlet series (\ref{Dirichlet Series in Definition}) satisfying the above definition \ref{definition 1.1.} and obeying to Hecke's functional equation (\ref{Hecke Dirichlet series Functional}). Also, let $s$ be a complex number such that $\text{Re}(s)>\mu>\max\left\{ 0,\,2\sigma_{a},\,2\sigma_{b},\,\frac{r^{\prime}}{2}\right\},$
where $\sigma_{a}=\max\{\sigma_{a_{1}},\,\sigma_{a_{2}}\}$,
$\sigma_{b}=\max\{\sigma_{b_{1}},\,\sigma_{b_{2}}\}$
and $r^{\prime}=\max\left\{ r_{1},\,r_{2}\right\} $. Let $H_{r_{1}}\left(s;\,b_{1},\,a_{2};\,\mu,\,\lambda^{\prime}\right)$
denote the double series 
\begin{equation}
H_{r_{1}}(s;\,b_{1},\,a_{2};\,\mu,\,\lambda^{\prime})=\sum_{m,n=1}^{\infty}b_{1}(m)\,a_{2}(n)\,\left(\frac{\mu_{m}}{\lambda_{n}^{\prime}}\right)^{\frac{s-r_{1}}{2}}\,K_{r_{1}-s}\left(2\sqrt{\mu_{m}\lambda_{n}^{\prime}}\right).\label{H1}
\end{equation}
Then, for $\text{Re}(s)>\mu$, the following identity for the generalized
Epstein zeta function holds 
\begin{align}
\Gamma(s)\,\mathcal{Z}_{2}\left(s;\,a_{1},\,a_{2};\,\lambda,\,\lambda^{\prime}\right)&=a_{1}(0)\,\Gamma(s)\,\phi_{2}(s)\,+\,a_{2}(0)\,\Gamma(s)\,\phi_{1}(s)\,+\,\sum_{n=1}^{\infty}\frac{a_{2}(n)}{\lambda_{n}^{\prime s}}\,R_{1}\left(s,\,\frac{1}{\lambda_{n}^{\prime}}\right)\,\label{First Selberg Chowla-1}\nonumber\\
&+\,2\,H_{r_{1}}\left(s;\,b_{1},\,a_{2};\,\mu,\,\lambda^{\prime}\right),
\end{align}
where $b_{1}(m)$ is the arithmetical function associated
to the Dirichlet series $\psi_{1}(s)$ and $R_{1}(s,\,x)$ denotes
the sum of the residues of the meromorphic function $R_{1}(z)=\Gamma(z)\,\phi_{1}(z)\,\Gamma(s-z)\,x^{-z}$ 
at the poles of $\Gamma(z)\,\phi_{1}(z)$.
\medskip{}
Equivalently, if $H_{r_{2}}\left(s;\,b_{2},\,a_{1};\,\mu^{\prime},\,\lambda\right)$
denotes the double infinite series, 
\begin{equation}
H_{r_{2}}(s;\,b_{2},\,a_{1};\,\mu^{\prime},\,\lambda)=\sum_{m,n=1}^{\infty}b_{2}(m)\,a_{1}(n)\,\left(\frac{\mu_{m}^{\prime}}{\lambda_{n}}\right)^{\frac{s-r_{2}}{2}}\,K_{r_{2}-s}\left(2\,\sqrt{\mu_{m}^{\prime}\lambda_{n}}\right),\label{H2}
\end{equation}
then $\mathcal{Z}_{2}\left(s;\,a_{1},\,a_{2};\,\lambda,\,\lambda^{\prime}\right)$
can be also described by the formula
\begin{align}
\Gamma(s)\,\mathcal{Z}_{2}\left(s;\,a_{1},\,a_{2};\,\lambda,\,\lambda^{\prime}\right)&=a_{1}(0)\,\Gamma(s)\,\phi_{2}(s)\,+\,a_{2}(0)\,\Gamma(s)\,\phi_{1}(s)\,+\,\sum_{m=1}^{\infty}\frac{a_{1}(m)}{\lambda_{m}^{s}}\,R_{2}\left(s,\,\frac{1}{\lambda_{m}}\right)\,\label{Second Selberg CHOWLA} \nonumber\\
&+\,2\,H_{r_{2}}\left(s;\,b_{2},\,a_{1};\,\mu^{\prime},\,\lambda\right),
\end{align}
whenever $\text{Re}(s)>\mu$. Here, in analogy with $R_{1}(s,\,x)$,
$R_{2}(s,\,x)$ denotes the sum of the residues of the meromorphic
function $R_{2}(z)=\Gamma(z)\,\phi_{2}(z)\,\Gamma(s-z)\,x^{-z}$ at the poles of $\Gamma(z)\,\phi_{2}(z)$.
\end{theorem}

\begin{proof}
The idea of the proof follows [\cite{dirichlet and hecke}, p. 311] and its sketched
version given in \cite{dirichletserisVI}. Let $\mu>\max\left\{ 0,\,2\sigma_{a},\,2\sigma_{b},\,\frac{r^{\prime}}{2}\right\} $
as in the statement and consider a fixed $s\in\mathbb{C}$ such that
$\text{Re}(s)>\mu$. Under this condition, the Beta transform
integral [\cite{ERDELIY}, p. 349, 7.3.15]
\begin{equation}
\frac{1}{2\pi i}\,\int_{\mu-i\infty}^{\mu+i\infty}\Gamma(z)\,\Gamma(s-z)\,x^{-z}dz=\frac{\Gamma(s)}{(1+x)^{s}},\,\,\,\,\,\,x>0\label{Mellin Representation Beta}
\end{equation}
holds. Since our Epstein series (\ref{Epstein as double series}) is defined as a double infinite series
which converges absolutely for $\text{Re}(s)>\mu\geq2\,\sigma_{a}$,
we may write it as 
\begin{equation}
\mathcal{Z}_{2}\left(s;\,a_{1},\,a_{2};\,\lambda,\,\lambda^{\prime}\right)=a_{1}(0)\phi_{2}(s)\,+\,a_{2}(0)\phi_{1}(s)\,+\,\sum_{n=1}^{\infty}a_{2}(n)\,\ell_{n}\left(s,\,a_{1}\right),\label{Series representation Z2}
\end{equation}
where $\ell_{n}(s,\,a_{1})$ denotes the generalized Dirichlet
series \cite{dirichlet and hecke} given by 
\[
\ell_{n}\left(s,\,a_{1}\right)=\sum_{m=1}^{\infty}\frac{a_{1}(m)}{\left(\lambda_{m}+\lambda_{n}^{\prime}\right)^{s}},\,\,\,\,\,\,\text{Re}(s)>\sigma_{a_{1}}.
\]

By the Mellin representation (\ref{Mellin Representation Beta}),
we see that we can write $\ell_{n}\left(s,\,a_{1}\right)$
as the contour integral 
\begin{equation}
\lambda_{n}^{\prime s}\,\Gamma(s)\,\ell_{n}\left(s,\,a_{1}\right)=\sum_{m=1}^{\infty}\frac{a_{1}(m)\,\Gamma(s)}{\left(1+\frac{\lambda_{m}}{\lambda_{n}^{\prime}}\right)^{s}}=\sum_{m=1}^{\infty}\frac{a_{1}(m)}{2\pi i}\,\int_{\mu-i\infty}^{\mu+i\infty}\Gamma(z)\,\Gamma(s-z)\,\left(\frac{\lambda_{m}}{\lambda_{n}^{\prime}}\right)^{-z}dz,\label{1st equality}
\end{equation}
valid whenever $\text{Re}(s)>\sigma_{a_{1}}$, which is
the case as $\text{Re}(s)>\mu>\max\left\{ 0,\,2\sigma_{a}\right\} \geq\max\left\{ 0,\,2\sigma_{a_{1}}\right\} $.
From the definition of $\mu$, we know that $\sum a_{1}(m)\,\lambda_{m}^{-\mu}$
converges absolutely. Moreover, from a corollary of Stirling's formula (\ref{Stirling exact form on Introduction}),
\begin{equation}
|\Gamma(\mu+it)\,\Gamma(s-\mu-it)|=O\left(|t|^{\text{Re}(s)-1}e^{-\pi|t|}\right),\,\,\,\,\,\,|t|\rightarrow\infty,\label{Stirling's formula in the first proof}
\end{equation}
we are allowed to interchange the orders of integration and summation
in (\ref{1st equality}). This gives the integral representation for
the generalized Dirichlet series $\ell_{n}(s,\,a_{1})$,
\[
\lambda_{n}^{\prime s}\Gamma(s)\,\ell_{n}(s,\,a_{1})=\frac{1}{2\pi i}\,\int_{\mu-i\infty}^{\mu+i\infty}\phi_{1}(z)\,\Gamma\left(z\right)\Gamma(s-z)\,\left(\frac{1}{\lambda_{n}^{\prime}}\right)^{-z}\,dz,\,\,\,\,\mu>\max\left\{ 0,\,\sigma_{a},\,\sigma_{b},\,\frac{r_{1}}{2}\right\} .
\]

Let us now move the line of integration to $z=r_{1}-\mu+it$ (which
is on the left of $\text{Re}(z)=\mu$ because $\mu>\frac{r^{\prime}}{2}\geq\frac{r_{1}}{2}$
by hypothesis) and integrate along a positively oriented rectangular contour $\mathcal{R}_{1}$
containing the vertices $\mu\pm iT$ and $r_{1}-\mu\pm iT$, where
$T>\sup_{\rho\in\mathscr{S}}\,|\text{Im}(\rho)|$, with $\mathscr{S}$
denoting the set of singularities\footnote{Note that this supremum taken over $\mathscr{S}$ is finite from the
assumption that the singularities of $\chi_{1}(z)$ are contained
in a compact set and the observation that all the singularities of
$\Gamma(s-z)$ have their imaginary part equal to $\text{Im}(s)$.} of $\chi_{1}(z)\,\Gamma(s-z)$. By the Residue Theorem, we have the
equality 
\begin{align}
\frac{1}{2\pi i}\,\int_{\mu-iT}^{\mu+iT}\phi_{1}(z)\,\Gamma\left(z\right)\Gamma(s-z)\,\lambda_{n}^{\prime z}\,dz & =\frac{1}{2\pi i}\,\int_{r_{1}-\mu-iT}^{r_{1}-\mu+iT}\phi_{1}(z)\,\Gamma\left(z\right)\Gamma(s-z)\,\lambda_{n}^{\prime z}\,dz\,\nonumber \\
+\frac{1}{2\pi i}\int_{\mu\pm iT}^{r_{1}-\mu\pm iT}\phi_{1}(z)\,\Gamma\left(z\right)\Gamma(s-z)\,\lambda_{n}^{\prime z}\,dz & +\sum_{\rho\in\mathcal{R}_{1}}\,\,\text{Res}_{z=\rho}\left(\phi_{1}(z)\,\Gamma\left(z\right)\Gamma(s-z)\,\left(1/\lambda_{n}^{\prime}\right)^{-z}\right),\label{T goes to infinity?}
\end{align}
where the second integral represents the integrals over the horizontal
lines. We now see that 
\begin{equation}
\int_{\mu\pm iT}^{r_{1}-\mu\pm iT}\phi_{1}(z)\,\Gamma\left(z\right)\Gamma(s-z)\,{\lambda_{n}^{\prime}}^{z}\,dz\rightarrow0\,\,\,\,\text{as }\,T\rightarrow\infty,\label{vanishing}
\end{equation}
by condition 2 on Definition \ref{definition 1.1.} \footnote{Despite this imposition, we should
note that (\ref{vanishing}) could easily come from an immediate application
of Stirling's formula (\ref{Stirling exact form on Introduction}) together with the Phragm\'en-Lindel\"of estimate (\ref{Lindelof Phragmen for Hecke})
and a weaker condition of the form $\phi(s)=O\left(\exp\,|s|^{K}\right)$, for
some $K>0$, as $|s|\rightarrow\infty$ \cite{dirichlet and hecke, berndt_zeros_(i), berndt_zeros_(ii)}.
However, by the adoption of Definition \ref{definition 1.1.}, this justification becomes
even easier.}. Letting $T\rightarrow\infty$ in (\ref{T goes to infinity?}),
we obtain 
\begin{align}
\lambda_{n}^{\prime s}\Gamma(s)\,\ell_{n}(s,\,a_{1}) & =\frac{1}{2\pi i}\,\int_{r_{1}-\mu-i\infty}^{r_{1}-\mu+i\infty}\phi_{1}(z)\,\Gamma\left(z\right)\Gamma(s-z)\,\left(\frac{1}{\lambda_{n}^{\prime}}\right)^{-z}\,dz\nonumber \\
 & +\sum_{\rho\in\mathcal{R}_{1}}\,\,\text{Res}_{z=\rho}\left(\phi_{1}(z)\,\Gamma\left(z\right)\Gamma(s-z)\,\left(1/\lambda_{n}^{\prime}\right)^{-z}\right).\label{Poles equation}
\end{align}

We now claim that the last term of the previous equality is precisely
$R_{1}(s,\,1/\lambda_{n}^{\prime})$, whose definition is given at
the statement of this result. It is easy to check by the functional
equation for $\phi_{1}(z)$, (\ref{Hecke Dirichlet series Functional}), and the definition of $\mu$ that all
poles of $\chi_{1}(z)=\Gamma(z)\,\phi_{1}(z)$ are on the interior
of $\mathcal{R}_{1}$. Moreover, since $\text{Re}(s)>\mu$ by hypothesis,
we clearly have that $\text{Re}(s-z)>0$ for any $z\in\mathcal{R}_{1}$,
so that $\Gamma(s-z)$ has no poles inside $\mathcal{R}_{1}$ and
the only poles taken into account in the sum in (\ref{Poles equation})
are precisely the ones coming from the function $\Gamma(z)\,\phi_{1}(z)$.
Thus, the previous equality reduces to 
\[
\lambda_{n}^{\prime s}\,\Gamma(s)\,\ell_{n}(s,\,a_{1})=\frac{1}{2\pi i}\,\int_{r_{1}-\mu-i\infty}^{r_{1}-\mu+i\infty}\phi_{1}(z)\,\Gamma\left(z\right)\Gamma(s-z)\,\left(\frac{1}{\lambda_{n}^{\prime}}\right)^{-z}\,dz+R_{1}\left(s,\,\frac{1}{\lambda_{n}^{\prime}}\right).
\]

Let us now invoke the functional equation for the Dirichlet series
$\phi_{1}(z)$ (\ref{Hecke Dirichlet series Functional}) and use it in the contour integral given above. Doing
this and taking the change of variables $z\leftrightarrow r_{1}-z$,
we arrive at 
\begin{align}
\lambda_{n}^{\prime}{}^{s}\Gamma(s)\,\ell_{n}(s,\,a_{1}) & =\lambda_{n}^{\prime r_{1}}\,\frac{1}{2\pi i}\,\int_{\mu-i\infty}^{\mu+i\infty}\psi_{1}(z)\,\Gamma\left(z\right)\Gamma(s-r_{1}+z)\,\lambda_{n}^{\prime-z}\,dz+R_{1}\left(s,\,\frac{1}{\lambda_{n}^{\prime}}\right)\nonumber \\
=\lambda_{n}^{\prime r_{1}}\,\sum_{m=1}^{\infty}\frac{b_{1}(m)}{2\pi i} & \,\int_{\mu-i\infty}^{\mu+i\infty}\Gamma\left(z\right)\Gamma(s-r_{1}+z)\,\left(\mu_{m}\lambda_{n}^{\prime}\right)^{-z}\,dz+R_{1}\left(s,\,\frac{1}{\lambda_{n}^{\prime}}\right),\label{NIIIINE}
\end{align}
where in the last step we have also used the absolute convergence
of the Dirichlet series $\psi_{1}(z)$ for $\text{Re}(z)=\mu>\sigma_{b}\geq\text{\ensuremath{\sigma}}_{b_{1}}$
and Stirling's formula for the product of $\Gamma-$functions. 

Recalling the Mellin representation for the Modified Bessel function [\cite{ERDELIY}, p. 349, 7.3 (17)] 
\begin{equation}
K_{\nu}(x)=\left(\frac{x}{2}\right)^{\nu}\frac{1}{4\pi i}\,\int_{\lambda-i\infty}^{\lambda+i\infty}\Gamma(z-\nu)\,\Gamma(z)\,\left(\frac{x^{2}}{4}\right)^{-z}dz,\,\,\,\,\lambda>\max\{0,\,\,\text{Re}(\nu)\},\label{MacDonald Representation}
\end{equation}
we see that, since $\mu>\max\left\{ 0,\,r_{1}-\text{Re}(s)\right\} $,
we may represent the contour integral on the right-hand side of (\ref{NIIIINE}) as
\begin{equation}
\frac{1}{2\pi i}\int_{\mu-i\infty}^{\mu+i\infty}\Gamma\left(z\right)\Gamma(s-r_{1}+z)\,\left(\mu_{m}\lambda_{n}^{\prime}\right)^{-z}\,dz=2\,\left(\mu_{m}\lambda_{n}^{\prime}\right)^{-\frac{r_{1}-s}{2}}\,K_{r_{1}-s}\left(2\,\sqrt{\mu_{m}\lambda_{n}^{\prime}}\right).\label{Almost Final Selberg Chowla}
\end{equation}

Combining (\ref{NIIIINE}) and (\ref{Almost Final Selberg Chowla}),
we arrive at the following representation of the generalized Dirichlet
series $\ell_{n}(s,\,a_{1})$ (c.f. \cite{dirichlet and hecke}) 
\[
\lambda_{n}^{\prime s}\,\Gamma(s)\,\ell_{n}(s,\,a_{1})=2\,\lambda_{n}^{\prime r_{1}}\,\sum_{m=1}^{\infty}b_{1}(m)\,\left(\mu_{m}\lambda_{n}^{\prime}\right)^{-\frac{r_{1}-s}{2}}\,K_{r_{1}-s}\left(2\,\sqrt{\mu_{m}\lambda_{n}^{\prime}}\right)+R_{1}\left(s,\,\frac{1}{\lambda_{n}^{\prime}}\right).
\]

We can easily see that the double series on the right side of (\ref{Series representation Z2})
converges absolutely if $\text{Re}(s)>\sigma_{a_{1}}+\sigma_{a_{2}}$
and so it must converge absolutely for $\text{Re}(s)>\mu$. Since $\phi_{1}(s)$ is analytic in the region $\text{Re}(s)>\sigma_{a_{1}}$,
it is clear from the expression for $R_{1}(s,\,x)$ that the series
$\sum a_{2}(n)\,\lambda_{n}^{\prime-s}\,R_{1}(s,\,\lambda_{n}^{\prime-1})$
converges absolutely whenever $\text{Re}(s)>\sigma_{a_{2}}+\max\left\{ 0,\,\sigma_{a_{1}}\right\} $.
 Hence, assuming that $\text{Re}(s)>\mu\geq\sigma_{a_{2}}+\max\left\{ 0,\,\sigma_{a}\right\} $
and summing $a_{2}(n)\,\ell_{n}(s,\,a_{1})$
with respect to $n$ as in (\ref{Series representation Z2}), we arrive
at the representation (\ref{First Selberg Chowla-1}).  To finish, we only
need to argue that, under the assumption $\text{Re}(s)>\mu$, the
double series which we have obtained by summing with respect to $n$,
(\ref{H1}), converges absolutely. This is immediate due to the asymptotic
estimate for the modified Bessel function 
\begin{equation}
K_{\nu}(x)=O\left(x^{-\frac{1}{2}}\,e^{-x}\right),\,\,\,\,\,\,\,x\rightarrow\infty\label{Classical Estimate Bessel}
\end{equation}
which implies that, for some positive constant $C$, 
\[
\left|H_{r_{1}}(s;\,b_{1},\,a_{2};\,\mu,\,\lambda^{\prime})\right|\leq C\,\sum_{m,n=1}^{\infty}b_{1}(m)\,a_{2}(n)\,\left(\frac{\mu_{m}}{\lambda_{n}^{\prime}}\right)^{\frac{s-r_{1}}{2}}\left(\mu_{m}\lambda_{n}^{\prime}\right)^{-\frac{1}{4}}\,e^{-2\sqrt{\mu_{m}\lambda_{n}^{\prime}}}\,
\]

For any $\alpha>0$, we know that $e^{-2\sqrt{\mu_{m}\lambda_{n}^{\prime}}}\leq(\mu_{m}\lambda_{n}^{\prime})^{-\alpha}$
for $m,\,n>N$, where $N$ is sufficiently large\footnote{note that $(\nu_{j})_{j\in\mathbb{N}}:=\left(\mu_{m}\lambda_{n}^{\prime}\right)_{m,n\in\mathbb{N}}$
may be arranged into an increasing sequence and we can impose the
inequality $e^{-2\sqrt{\nu_{j}}}\leq\nu_{j}^{-\alpha}$ for $j\geq N_{0}$.}. Choose $\alpha>\max\left\{ \sigma_{a}-\frac{1}{4}+\frac{r_{1}-s}{2},\,\sigma_{b}-\frac{1}{4}+\frac{s-r_{1}}{2}\right\} $:
from this choice and the absolute convergence of $\phi_{2}$ and $\psi_{1}$
for $\text{Re}(s)>\max\left\{ \sigma_{a},\,\sigma_{b}\right\} $,
we immediately see that the series defining $H_{r_{1}}$ converges
absolutely for $\text{Re}(s)>\mu$.

In a completely analogous way, the second formula (\ref{Second Selberg CHOWLA})
may be obtained, with the same choice of $\mu$. In order to ensure
that the identities (\ref{First Selberg Chowla-1}) and (\ref{Second Selberg CHOWLA})
represent the same diagonal Epstein zeta function, we need to impose the possibility
of reversing the order of summation in the double series (\ref{Epstein as double series}).
This is the case when we take $\text{Re}(s)>\max\left\{ 0,\,2\sigma_{a}\right\} $,
so our choice of $\mu$ allows to reverse the order. This concludes
the proof.
\end{proof}
\bigskip{}

Since the number
of poles of $\phi_{1}$ and $\phi_{2}$ is finite (as they are contained
in a compact set, by condition 3. in definition \ref{definition 1.1.}), we can easily
study the continuation of the residual series 
\begin{equation}
\varphi_{1}(s):=\sum_{m=1}^{\infty}\frac{a_{1}(m)}{\lambda_{m}^{s}}\,R_{2}\left(s,\,\frac{1}{\lambda_{m}}\right),\,\,\,\,\varphi_{2}(s):=\sum_{n=1}^{\infty}\frac{a_{2}(n)}{\lambda_{n}^{\prime s}}\,R_{1}\left(s,\,\frac{1}{\lambda_{n}^{\prime}}\right).\label{Residual Functions-1}
\end{equation}



Thus, if we warrant that any one of the Dirichlet series $(\phi_{1},\,\phi_{2},\,\psi_{1},\,\psi_{2})$ has analytic continuation to the entire complex plane, the only remaining step towards the extension
of the Selberg-Chowla formula to any $s\in\mathbb{C}$ is the analytic
continuation of the functions $H_{r_{j}}(s,\,a)$. As in
the case of the classical Selberg-Chowla formula \cite{Bateman_Epstein}, we
will show that $H_{r_{1}}$ and $H_{r_{2}}$ represent entire functions
of $s$. Although the proof by Bateman and Grosswald \cite{Bateman_Epstein} works mutatis mutandis in our case, we add, for completeness, an alternative proof of this fact, which employs the Mellin representation 
(\ref{MacDonald Representation}) of the modified Bessel function.
A similar proof, adapted in a different context, is given in [\cite{dirichletserisVI}, p. 181]. 

\bigskip{}

\begin{proposition}[The analytic continuation of $H$] \label{proposition 2.1}
Let $H_{r_{1}}$ and $H_{r_{2}}$ be the double series given in (\ref{H1})
and (\ref{H2}). Then they represent entire functions in the variable $s\in \mathbb{C}$ and obey to
the functional equation
\begin{equation}
H_{r_{1}}\left(s;\,a_{1},\,b_{2};\,\lambda,\,\mu^{\prime}\right)=H_{r_{2}}\left(r_{1}+r_{2}-s;\,b_{2},\,a_{1};\,\mu^{\prime},\,\lambda\right).\label{Reflection Formula-2}
\end{equation}
\end{proposition}

\begin{proof}
As in the previous Theorem, let $\mu>\max\left\{ 0,\,2\sigma_{a},\,2\sigma_{b},\,\frac{r^{\prime}}{2}\right\} $
and take $\eta>\mu$. For any $\epsilon>0$, consider a rectangle
$\mathcal{R}_{\epsilon,\eta}(T)$, $T>0$, whose vertices are $r^{\prime}-\mu+\sigma_{a}+\epsilon\pm iT$
and $\eta\pm iT$. Assuming that $s\in\mathcal{R}_{\epsilon,\eta}(T)$
and returning to the representation (\ref{NIIIINE}), we can write
the function $H_{r_{1}}$ as 
\begin{equation}
H_{r_{1}}\left(s;\,a_{1};\,b_{2};\,\lambda,\,\mu^{\prime}\right):=\sum_{m,n=1}^{\infty}b_{1}(m)\,a_{2}(n)\,\lambda_{n}^{\prime r_{1}-s}\,\frac{1}{2\pi i}\,\int_{\mu-i\infty}^{\mu+i\infty}\Gamma\left(z\right)\Gamma(s-r_{1}+z)\,\left(\mu_{m}\lambda_{n}^{\prime}\right)^{-z}\,dz.\label{Returning to the classics}
\end{equation}

It is readily seen from (\ref{Returning to the classics}) that we
can argue the uniform convergence of the previous double series in
$\mathcal{R}_{\epsilon,\eta}(T)$ by estimating the contour integral
appearing on the right-hand side of (\ref{Returning to the classics}).
To do this, we will set $s:=\sigma+it\in\mathcal{R}_{\epsilon,\eta}(T)$
and split the integral concerning the vertical line $(\mu-i\infty,\,\mu+i\infty)$ into three parts:
$(\mu-i\infty,\,\mu-i(T+1))$, $(\mu-i(T+1),\,\mu+i(T+1))$ and $(\mu+i(T+1),\,\mu+i\infty)$. The bound for the first of these integrals is easy to obtain and is simply given by 
\begin{equation}
\left|\,\int_{\mu-i(T+1)}^{\mu+i(T+1)}\Gamma\left(z\right)\Gamma(s-r_{1}+z)\,\left(\mu_{m}\lambda_{n}^{\prime}\right)^{-z}\,dz\,\right|\leq C\,\left(\mu_{m}\lambda_{n}^{\prime}\right)^{-\mu}\Gamma(\sigma-r_{1}+\mu)\,\Gamma(\mu)\leq C^{\prime}\,\left(\mu_{m}\lambda_{n}^{\prime}\right)^{-\mu}\label{Bound (1)}
\end{equation}
where $C^{\prime}$ only depends on the maximum value that $\Gamma(x)$
attains in the interval $\sigma_{a}+r^{\prime}-r_{1}+\epsilon<x<\eta+\mu-r_{1}$ and thus (\ref{Bound (1)}) does not depend of $s\in\mathcal{R}_{\epsilon,\eta}(T)$.
From Stirling's formula, it is also immediate to arrive at the second
bound 
\begin{align}
\left|\int_{\mu+i(T+1)}^{\mu+i\infty}\Gamma\left(z\right)\Gamma(s-r_{1}+z)\,\left(\mu_{m}\lambda_{n}^{\prime}\right)^{-z}\,dz\right| & \leq C\,\left(\mu_{m}\lambda_{n}^{\prime}\right)^{-\mu}\,\int_{T+1}^{\infty}y^{\mu-\frac{1}{2}}\,e^{-\pi y}\,(y+t)^{\sigma-r_{1}+\mu-\frac{1}{2}}\,dy\nonumber \\
\leq C\,\left(\mu_{m}\lambda_{n}^{\prime}\right)^{-\mu}\,\int_{T+1}^{\infty}y^{\mu-\frac{1}{2}}\,e^{-\pi y}\, & (y+t)^{\eta-r_{1}+\mu-\frac{1}{2}}\,dy\leq C^{\prime}\,(\mu_{m}\lambda_{n}^{\prime})^{-\mu},\label{Bound (2)}
\end{align}
where, once more, $C^{\prime}$ does not depend also on $s\in\mathcal{R}_{\epsilon,\eta}(T)$.
Due to the symmetry $\Gamma(\overline{z})=\overline{\Gamma(z)}$,
it is simple to see that a similar estimate to (\ref{Bound (2)}) holds when we integrate from $\mu-i\infty$ to $\mu-i(T+1)$. By (\ref{Returning to the classics}),
we see that 
\[
\left|H_{r_{1}}\left(s;\,a_{1};\,b_{2};\,\lambda,\,\mu^{\prime}\right)\right|\leq C\,\sum_{m=1}^{\infty}\frac{|b_{1}(m)|}{\mu_{m}^{\mu}}\,\cdot\,\sum_{n=1}^{\infty}\frac{|a_{2}(n)|}{\lambda_{n}^{\prime\mu+\sigma-r_{1}}},
\]
where $\sigma>r^{\prime}-\mu+\sigma_{a}+\epsilon$ (since
$s\in\mathcal{R}_{\epsilon,\eta}(T)$) and $C$ does not depend on $s$.
Since the right-hand side of the previous inequality is the product
of two absolutely convergent series, an application of the Weierstrass M-test leads to the conclusion that the infinite series appearing in (\ref{Returning to the classics}) converges
absolutely and uniformly on $\mathcal{R}_{\epsilon,\eta}(T)$. Since
$\mu$ can be taken sufficiently large, we know that every bounded
subset of $\mathbb{C}$ is contained in some rectangle of the form
$\mathcal{R}_{\epsilon,\eta}(T)$ and this shows that the double series (\ref{H1-1}) converges absolutely and uniformly in every bounded subset
of the complex plane. From the fact that the double sequence of functions
$f_{m,n}(s)=2\,\left(\mu_{m}/\lambda_{n}^{\prime}\right)^{\frac{s-r_{1}}{2}}\,K_{r_{1}-s}\left(2\sqrt{\mu_{m}\lambda_{n}^{\prime}}\right)$
represents, for each pair of integers $(m,\,n)\in\mathbb{N}^{2}$,
an analytic function of $s$, we thus have that $H_{r_{1}}(s;\,b_{1},\,a_{2};\,\mu,\,\lambda^{\prime})$
is an entire function. 


Analogously, one may see that $H_{r_{2}}$ converges uniformly on
any rectangle $\mathcal{R}_{\epsilon,\eta}(T)$. The same can be concluded
after an application of the reflection formula (\ref{Reflection Formula-2}).
By analytic continuation, this shows that both representations (\ref{First Selberg Chowla-1})
and (\ref{Second Selberg CHOWLA}) hold for the Epstein zeta function (\ref{Epstein as double series}). Finally, it is clear that
(\ref{Reflection Formula-2}) comes immediately from the property for
the Modified Bessel function $K_{\nu}(z)=K_{-\nu}(z)$. 
\end{proof}
\medskip{}

\begin{remark}
It should be also pointed out that, even if one of the Dirichlet series
$\phi_{1}$ or $\phi_{2}$ does not possess a functional equation
of Hecke type (\ref{Hecke Dirichlet series Functional}), one of the formulas (\ref{First Selberg Chowla-1})
or (\ref{Second Selberg CHOWLA}) is still valid, since the only tool
required to prove these is the Functional equation of only one of
the Dirichlet series. Although the appropriate Selberg-Chowla formula
provides the analytic continuation of $\mathcal{Z}_{2}$ in this case,
as it will be clear, we cannot obtain a suitable functional equation
for it. 


\end{remark}

\medskip{}


\begin{definition}[A Generalized Divisor function] \label{definition generalized divisor 2.1.}
Let $\{\nu_{j}\}_{j\in\mathbb{N}}$ represent the "product sequence" $\nu_{j}=\left\{ \mu_{m}\cdot\lambda_{n}^{\prime}\right\} _{m,n\in\mathbb{N}}$ and then arrange this product in increasing order. We define the generalized divisor function as 
\begin{equation}
\sigma_{z}\left(\nu_{j};\,b_{1},\,a_{2};\,\mu,\,\lambda^{\prime}\right)=\sum_{(m,\,n)\in \mathbb{N}^{2}\,:\,\mu_{m} \lambda_{n}^{\prime}=\nu_{j}}b_{1}(m)\,a_{2}(n)\,\mu_{m}^{z}.\label{On a generalized Divisor Function}
\end{equation}
\end{definition}

\medskip{}

With the notation introduced in Definition \ref{definition generalized divisor 2.1.}, we have that $H_{r_{1}}$ can be described by
the series 
\begin{equation}
H_{r_{1}}(s;\,b_{1},\,a_{2};\,\mu,\,\lambda^{\prime})=\,\sum_{j=1}^{\infty}\sigma_{s-r_{1}}\left(\nu_{j};\,b_{1},\,a_{2};\,\mu,\,\lambda^{\prime}\right)\,\nu_{j}^{\frac{r_{1}-s}{2}}\,K_{r_{1}-s}\left(2\sqrt{\nu_{j}}\right), \label{Representation with divisotr}
\end{equation}
which strongly resembles the entire part
appearing in the classical Selberg-Chowla formula (\ref{Selberg Chowla Formula}). Analogously, if we write $\{\nu_{j}^{\prime}\}_{j\in\mathbb{N}}=\{\mu_{m}^{\prime}\lambda_{n}\}_{m,n\in\mathbb{N}}$,
it is also evident that we can write $H_{r_{2}}$ as 
\[
H_{r_{2}}(s;\,b_{2},\,a_{1};\,\mu^{\prime},\,\lambda)=\,\sum_{j=1}^{\infty}\sigma_{s-r_{2}}\left(\nu_{j}^{\prime};\,b_{2},\,a_{1};\,\mu^{\prime},\,\lambda\right)\,\nu_{j}^{\prime\frac{r_{2}-s}{2}}\,K_{r_{2}-s}\left(2\,\sqrt{\nu_{j}^{\prime}}\right).
\]

\medskip{}

Taking $z=0$ in (\ref{On a generalized Divisor Function}), we also get
a generalization of the standard divisor function $d(n)=\sum_{d|n}1$, 
\begin{equation}
d\left(\nu_{j};\,b_{1},\,a_{2};\,\mu,\,\lambda^{\prime}\right)=\sum_{(m,\,n)\in \mathbb{N}^{2}\,:\,\mu_{m}\lambda_{n}^{\prime}=\nu_{j}}b_{1}(m)\,a_{2}(n).\label{generalized Divisor function}
\end{equation}

\medskip{}


Having proved that $H_{r_{i}}$, for each $i=1,\,2$, is entire, we
are ready to establish the analytic continuation of the Epstein series
$\mathcal{Z}_{2}\left(s;\,a_{1},\,a_{2};\,\lambda,\,\lambda^{\prime}\right)$.
To do this, we need to understand the continuation of the residual
functions given in the proof of Theorem \ref{theorem 2.1}, $\varphi_{1}(s)$ and $\varphi_{2}(s)$, appearing in (\ref{Residual Functions-1}). Although we can proceed the study of the analytic continuation of $\mathcal{Z}_{2}(s;\,\cdot)$ in a general form (see Remark \ref{remark with general residual}), we state all the next corollaries in the following class of Dirichlet series: 

\begin{definition} \label{class A definition}
Let $\phi(s)$ be a Dirichlet series satisfying Definition \ref{definition 1.1.} with $\Delta(s)=\Gamma(s)$. We say
that $\phi$ belongs to the class $\mathcal{A}$ if additionally: 
\begin{enumerate}
\item $\phi$ and $\psi$ have analytic continuations to the entire complex plane and are analytic on $\mathbb{C}$ except
for possible simple poles located at $s=r$ with residues $\rho$
and $\rho^{\star}$ respectively.
\item $a(n)$ can be defined at $0$ and has the value $-\phi(0)$. $b(n)$ can be analogously extended to $n=0$ as $-\psi(0)$. 
\end{enumerate}
\end{definition}
\begin{remark} \label{easy computations class A}
Note that, by the functional equation (\ref{Hecke Dirichlet series Functional})
and condition 1 in the previous definition, $\phi_{i}(0)=-\rho_{i}^{\star}\Gamma(r_{i})$,
while $\psi_{i}(0)=-\rho_{i}\,\Gamma(r_{i})$. Under the symmetry
imposed by the functional equation and the first condition, it is
clear that $\phi\in\mathcal{A}$ if and only if $\psi\in\mathcal{A}$.
The purpose of introducing this class is to mimic as much as possible
the class of Dirichlet series with a given signature, which is a clear
subclass of $\mathcal{A}$. Besides, this class reproduces in a general
form what happens in the classical cases of diagonal Epstein zeta
functions. As a corollary of the Selberg-Chowla formulas given in Theorem \ref{theorem 2.1}, in the next result 
we prove that, if $\phi_{1},\,\phi_{2}\in\mathcal{A}$, then $\mathcal{Z}_{2}\in\mathcal{A}$.
Hence, this class is closed under composition of Dirichlet series attached to diagonal Epstein zeta functions. A similar class was also considered in [\cite{hecke and identities Berndt}, p. 221]. 
\end{remark}

\medskip{}

From the functional equation for $\phi$ and under the hypotheses given for the class $\mathcal{A}$, it is simple to see that the
residual function $R_{1}(s,\,x)$ appearing in (\ref{First Selberg Chowla-1}) is given by 
\begin{equation}
R_{1}\left(s,\,x\right)=\phi_{1}(0)\,\Gamma(s)+\rho_{1}\Gamma\left(r_{1}\right)\,\Gamma\left(s-r_{1}\right)\,x^{-r_{1}},\label{R1}
\end{equation}
so that, for $\text{Re}(s)>\mu>\max\left\{ 0,\,2\sigma_{a},\,2\sigma_{b},\,\frac{r^{\prime}}{2}\right\} $,
we can easily see that 
\begin{equation}
\varphi_{2}(s)=\sum_{n=1}^{\infty}\frac{a_{2}(n)}{\lambda_{n}^{\prime s}}\,R_{1}\left(s,\,1/\lambda_{n}^{\prime}\right)=\phi_{1}(0)\,\Gamma(s)\,\phi_{2}(s)+\rho_{1}\Gamma(r_{1})\,\Gamma(s-r_{1})\,\phi_{2}\left(s-r_{1}\right).\label{Residual series again}
\end{equation}

Likewise, one may find $R_{2}\left(s,\,x\right)$ and $\varphi_{1}(s)$.
Invoking the fact that $a_{i}(0)=-\phi_{i}(0)$, we see that formulas
(\ref{First Selberg Chowla-1}) and (\ref{Second Selberg CHOWLA}),
when considered over the class $\mathcal{A}$, are
\begin{align}
\Gamma(s)\,\mathcal{Z}_{2}\left(s;\,a_{1},\,a_{2};\,\lambda,\,\lambda^{\prime}\right)&=-\phi_{2}(0)\,\Gamma(s)\,\phi_{1}(s)+\rho_{1}\Gamma(r_{1})\,\Gamma(s-r_{1})\,\phi_{2}\left(s-r_{1}\right) \nonumber \\
&+2\,H_{r_{1}}\left(s;\,b_{1},\,a_{2};\,\mu,\,\lambda^{\prime}\right)\label{Again First-1}
\end{align}
and 
\begin{align}
\Gamma(s)\,\mathcal{Z}_{2}\left(s;\,a_{1},\,a_{2};\,\lambda,\,\lambda^{\prime}\right)&=-\phi_{1}(0)\,\Gamma(s)\,\phi_{2}(s)+\rho_{2}\Gamma(r_{2})\,\Gamma(s-r_{2})\,\phi_{1}\left(s-r_{2}\right)\nonumber\\
&+\,2\,H_{r_{2}}(s;\,b_{2},\,a_{1};\,\mu^{\prime},\,\lambda).\label{Again Second-1}
\end{align}

Using both representations (\ref{Again First-1}) and (\ref{Again Second-1})
as well as the type of continuation which they yield, we now establish
the following Corollary. 

\begin{corollary}\label{The Analytic Continuation}
Let $\phi_{1}$ and $\phi_{2}$ satisfy the conditions of the class
$\mathcal{A}$ and having residues $\rho_{1}$ and $\rho_{2}$ at $s=r_{1}$ and $s=r_{2}$ respectively. Then the Selberg-Chowla formula (\ref{First Selberg Chowla-1})
provides the analytic continuation of the diagonal Epstein $\zeta-$function (\ref{Epstein as double series})
as a meromorphic function having a simple pole at $s=r_{1}+r_{2}$
with residue $\frac{\Gamma(r_{1})\Gamma(r_{2})}{\Gamma(r_{1}+r_{2})}\,\rho_{1}\rho_{2}$. Moreover, the continuation of $\mathcal{Z}_{2}$ obeys to the functional
equation 
\begin{equation}
\Gamma\left(s\right)\mathcal{Z}_{2}\left(s;\,a_{1},\,a_{2};\,\lambda,\,\lambda^{\prime}\right)=\Gamma(r_{1}+r_{2}-s)\,\mathcal{Z}_{2}\left(r_{1}+r_{2}-s;\,b_{1},\,b_{2};\,\mu,\,\mu^{\prime}\right).\label{Functional equation our Equation-1}
\end{equation}
\end{corollary}

\begin{proof}
First, we shall use the first representation for $\mathcal{Z}_{2}$
(\ref{First Selberg Chowla-1}), which, under restriction to the class
$\mathcal{A}$, reduces to (\ref{Again First-1}). We have seen in
the previous proposition that $H_{r_{1}}$ is an entire function,
so the meromorphic part of the Epstein $\zeta-$function $\mathcal{Z}_{2}\left(s;\,a_{1},\,a_{2};\,\lambda,\,\lambda^{\prime}\right)$
comes from the first two terms in (\ref{Again First-1}), i.e., from
the meromorphic function 
\begin{equation}
G_{r_{1}}\left(s;\,a_{1},\,a_{2};\,\lambda,\,\lambda^{\prime}\right)=-\phi_{2}(0)\,\phi_{1}(s)+\rho_{1}\Gamma(r_{1})\,\frac{\Gamma(s-r_{1})}{\Gamma(s)}\,\phi_{2}\left(s-r_{1}\right).
\end{equation}
Clearly, from a standard verification, $G_{r_{1}}$ has removable
singularities at $s=r_{1}-k$, $k\in\mathbb{N}_{0}$. Since, by the
definition of the class $\mathcal{A}$, $\phi_{1}(s)$ and $\Gamma\left(s-r_{1}\right)$
are analytic in a neighbourhood of $s=r_{1}+r_{2}$, $G_{r_{1}}$
has at most one pole at $s=r_{1}+r_{2}$ coming from the factor $\phi_{2}(s-r_{1})$.
The residue is easily seen to be $\frac{\Gamma(r_{1})\Gamma(r_{2})}{\Gamma(r_{1}+r_{2})}\,\rho_{1}\rho_{2}.$

\medskip{}

To prove the functional equation (\ref{Functional equation our Equation-1}),
let us use both representations (\ref{Again First-1}, \ref{Again Second-1})
for $\mathcal{Z}_{2}$ and replace $s$ by $r_{1}+r_{2}-s$ on (\ref{Again First-1}).
From the reflection formula for $H_{r_{1}}$ (\ref{Reflection Formula-2}), we see that 
\begin{equation}
H_{r_{1}}\left(r_{1}+r_{2}-s;\,b_{1},\,a_{2};\,\mu,\,\lambda^{\prime}\right)=H_{r_{2}}\left(s;\,a_{2},\,b_{1};\,\lambda^{\prime},\,\mu\right). \label{reflection on proof}    
\end{equation}

Thus, to describe $\mathcal{Z}_{2}(r_{1}+r_{2}-s;\,\cdot)$, we need
to see how $G_{r_{1}}$ behaves under the reflection $s\leftrightarrow r_{1}+r_{2}-s$.
Since $\phi_{i}(0)=-\rho_{i}^{\star}\Gamma(r_{i})$ and $\psi_{i}(0)=-\rho_{i}\,\Gamma(r_{i})$
(by condition 2. on the class $\mathcal{A}$ and the functional equation
for $\phi_{i}(s)$), we have that 
\begin{align}
\Gamma\left(r_{1}+r_{2}-s\right)\,G_{r_{1}}\left(r_{1}+r_{2}-s;\,a_{1},\,a_{2};\,\lambda,\,\lambda^{\prime}\right) & = -\phi_{2}(0)\,\Gamma\left(s-r_{2}\right)\,\psi_{1}\left(s-r_{2}\right)+\rho_{1}\Gamma(r_{1})\,\Gamma(s)\,\psi_{2}(s)\nonumber \\
&=\rho_{2}^{\star}\,\Gamma(r_{2})\,\Gamma\left(s-r_{2}\right)\,\psi_{1}\left(s-r_{2}\right)-\psi_{1}(0)\,\Gamma(s)\,\psi_{2}(s),\label{when Im sixty 3-2}
\end{align}
where in the penultimate equality we have used the functional equation
for $\phi_{i}$ directly. Now, by equations (\ref{Again First-1}), (\ref{reflection on proof}) and (\ref{when Im sixty 3-2}), we see that 
\begin{align}
\Gamma(r_{1}+r_{2}-s)\,\mathcal{Z}_{2}\left(r_{1}+r_{2}-s;\,a_{1},\,a_{2};\,\lambda,\,\lambda^{\prime}\right) & =-\psi_{1}(0)\,\Gamma(s)\,\psi_{2}(s)+\rho_{2}^{\star}\,\Gamma(r_{2})\,\Gamma\left(s-r_{2}\right)\,\psi_{1}\left(s-r_{2}\right)+\nonumber\\
 & +2\,H_{r_{2}}\left(s;\,a_{2},\,b_{1};\,\lambda^{\prime},\,\mu\right).\label{ALMOST OVER THIS-1} 
\end{align}

Using the second representation of the Selberg-Chowla formula
(\ref{Again Second-1}), and replacing there the arithmetical functions $a_{1}(n)$ by $b_{1}(n)$
and $a_{2}(n)$ by $b_{2}(n)$ or, by other words,
using the fact that $\psi_{1},\,\psi_{2}\in\mathcal{A}$ and applying the
work done in Theorem \ref{theorem 2.1} for the diagonal Epstein
$\zeta-$function,  
\begin{equation}
\mathcal{Z}_{2}\left(s;\,b_{1},\,b_{2};\,\mu,\,\mu^{\prime}\right)=\sum_{m,n\neq0}^{\infty}\frac{b_{1}(m)\,b_{2}(n)}{\left(\mu_{m}+\mu_{n}^{\prime}\right)^{s}},\,\,\,\,\,\,\text{Re}(s)>2\,\sigma_{b},
\end{equation}
we find immediately that the right-hand side of (\ref{ALMOST OVER THIS-1})
is precisely the same as the one given by (\ref{Again Second-1})
under the above mentioned substitutions. This establishes the functional
equation (\ref{Functional equation our Equation-1}) and completes
the proof.
\end{proof}


\begin{remark} \label{remark with general residual}
Although the totality of the examples given in the fifth section of
this paper concerns the class $\mathcal{A}$, we should remark that
the previous corollary, as well as the results given in section \ref{section 4}, hold for a
(much more) general class of Dirichlet series satisfying Definition \ref{definition 1.1.}
and admitting analytic continuations into the entire complex plane. However, in order to have a complete description of the analytic continuations of the Epstein zeta functions (\ref{Epstein as double series}) and (\ref{Epstein zeta function non diagonal-1}), we have decided to omit the details of this general case. In this remark we briefly see that the properties of a function satisfying definition \ref{definition 1.1.} are preserved by the Epstein zeta function. Note that, if $C_{1}$ denotes a curve,
or curves, chosen so that it encircles all the singularities of $\Gamma(z)\,\phi_{1}(z)$,
it is then clear that $R_{1}(s,\,x)$ (given in the statement of Theorem
\ref{theorem 2.1}) can be written as 
\[
R_{1}(s,\,x)=\frac{1}{2\pi i}\,\int_{C_{1}}\Gamma(s-w)\,\Gamma(w)\,\phi_{1}(w)\,x^{-w}\,dw,\,\,\,\,\text{Re}(s)>\mu, 
\]
and so the residual series $\varphi_{2}(s)$ appearing in (\ref{Residual Functions-1}) may be expressed as 
\begin{equation}
\varphi_{2}(s)=\frac{1}{2\pi i}\,\int_{C_{1}}\Gamma(s-w)\,\phi_{2}(s-w)\,\Gamma(w)\,\phi_{1}(w)\,dw,\,\,\,\,\text{Re}(s)>2\mu, \label{integral varphi 2}
\end{equation}
which can be extended to other values of $s$ by analytic continuation (with an analogous expression holding for $\varphi_{1}(s)$). Note that the right-hand side of
(\ref{integral varphi 2}) is analytic when $\text{Re}(s)$
is taken sufficiently large. Supposing that each $\phi_{i}(s)$ has an
analytic continuation to the entire complex plane, (\ref{integral varphi 2}) (resp. $\varphi_{1}(s)$)
represents a complex meromorphic function, having its singularities
also confined to a compact set. Thus, the first assertion proved in
Corollary \ref{The Analytic Continuation} holds outside the class $\mathcal{A}$. 

To deduce that we still have the functional equation (\ref{Functional equation our Equation-1})
in this general case, note that if $C_{1}$ encircles the singularities
of $\Gamma(w)\,\phi_{1}(w)$, then $C_{1}^{\prime}:=r_{1}-C_{1}=\left\{ r_{1}-z,\,\,z\in C_{1}\right\} $
encircles all the singularities of $\Gamma(w)\,\psi_{1}(w)$, by virtue
of definition \ref{definition 1.1.}. Note now that $\varphi_{2}(s)$ satisfies a functional equation of the form 
\begin{equation}
\varphi_{2}(r_{1}+r_{2}-s)=\frac{1}{2\pi i}\int_{C_{1}}\Gamma(s-w)\phi_{2}(s-w)\Gamma(w)\,\phi_{1}(w)dw =\frac{1}{2\pi i}\int_{C_{1}^{\prime}}\Gamma(s-w)\,\psi_{2}(s-w)\Gamma(w)\psi_{1}(w)dw,\label{right side general vase}
\end{equation}
where the right-hand side of the previous equality represents the residual series appearing in the first Selberg-Chowla
formula for $\mathcal{Z}_{2}(s;\,b_{1},\,b_{2};\,\mu,\,\mu^{\prime})$. 
\end{remark}
\begin{remark} \label{remark Theta-function}
Alternative proofs of Theorem \ref{theorem 2.1} and Corollary \ref{The Analytic Continuation} can be given
by using Bochner's modular relation (\ref{Bochner Modular relation at intro}) to one of the Dirichlet series
$\phi_{1}(s)$ or $\phi_{2}(s)$. Note that the left-hand side of
(\ref{First Selberg Chowla-1})
can be described, for $\text{Re}(s)>\mu>\max\left\{ 0,\,2\sigma_{a},\,2\sigma_{b},\,\frac{r^{\prime}}{2}\right\} $,
as the Mellin transform 
\begin{align}
\Gamma(s)\,\mathcal{Z}_{2}\left(s;\,a_{1},\,a_{2};\,\lambda,\,\lambda^{\prime}\right) & =\int_{0}^{\infty}x^{s-1}\,\sum_{m,n\neq0}^{\infty}a_{1}(m)\,a_{2}(n)\,e^{-(\lambda_{m}+\lambda_{n}^{\prime})\,x}\,dx\nonumber \\
=a_{1}(0)\,\int_{0}^{\infty}x^{s-1}\Theta\left(x;\,a_{2};\,\lambda^{\prime}\right)dx & +a_{2}(0)\,\int_{0}^{\infty}x^{s-1}\Theta\left(x;\,a_{1};\,\lambda\right)dx+\int_{0}^{\infty}x^{s-1}\Theta\left(x;\,a_{1};\,\lambda\right)\,\Theta\left(x;\,a_{2};\,\lambda^{\prime}\right)\,dx.\label{writing as theta functiiiionas}
\end{align}

It is then evident that, if we invoke the reflection formula for the
Theta function $\Theta(x;\,a_{1};\,\lambda)$ (\ref{Bochner Modular relation at intro}) on the third
integral above and then use the definition of the residual function
$P_{1}(x)$ (\ref{residual function definition Bochner modular}), as well as arguing by absolute convergence, the first two summands
and the residual terms coming on the third integral give the first
three terms appearing on the right-hand side of (\ref{First Selberg Chowla-1}).
Lastly, the remaining integral can be evaluated as follows 
\begin{align*}
\int_{0}^{\infty}x^{s-r_{1}-1}\Theta\left(\frac{1}{x};\,b_{1};\,\mu\right)\,\Theta\left(x;\,a_{2};\,\lambda^{\prime}\right)\,dx & =\sum_{m,n=1}^{\infty}b_{1}(m)\,a_{2}(n)\,\int_{0}^{\infty}x^{s-r_{1}-1}\exp\left(-\lambda_{n}^{\prime}x-\mu_{m}/x\right)\,dx\\
=2\,\sum_{m,n=1}^{\infty}b_{1}(m)\,a_{2}(n) & \left(\frac{\mu_{m}}{\lambda_{n}^{\prime}}\right)^{\frac{s-r_{1}}{2}}\,K_{r_{1}-s}\left(2\sqrt{\mu_{m}\lambda_{n}^{\prime}}\right):=2\,H_{r_{1}}(s;\,b_{1},\,a_{2};\,\mu,\,\lambda^{\prime}),
\end{align*}
where we have used the representation of the Modified Bessel function
$K_{\nu}(x)$ as the Mellin convolution of exponentials {[\cite{ryzhik},
eq. 3.471.9, p. 368]}
\begin{equation}
\int_{0}^{\infty}x^{s-1}e^{-\beta x}\,e^{-\gamma/x}\,dx=2\,\left(\frac{\gamma}{\beta}\right)^{\frac{s}{2}}\,K_{s}\left(2\sqrt{\beta\gamma}\right),\,\,\,\text{Re}(\beta),\,\text{Re}(\gamma)>0.\label{convolution exponentials equals K nu}
\end{equation}

The proof of the functional equation (\ref{Functional equation our Equation-1})
for $\phi_{1},\phi_{2}\in\mathcal{A}$ can be given by invoking the
reflection formula for both theta functions $\Theta(x;\,a_{1};\,\lambda)$
and $\Theta(x;\,a_{2};\,\lambda^{\prime})$. If we denote
the theta function of the Dirichlet series $\mathcal{Z}_{2}(s;\,\cdot)$
by $\Theta_{2}(x;\,a_{1},\,a_{2};\,\lambda,\,\lambda^{\prime})$
we can derive a particular case of the modular relation (\ref{Bochner Modular relation at intro}),
\begin{align}
\Theta_{2}(x;\,a_{1},\,a_{2};\,\lambda,\,\lambda^{\prime}) & =x^{-r_{1}-r_{2}}\Theta_{2}\left(x^{-1};\,b_{1},\,b_{2};\,\mu,\,\mu^{\prime}\right)+\Gamma(r_{1})\Gamma(r_{2})\,\rho_{1}\rho_{2}\,x^{-r_{1}-r_{2}}\nonumber \\
 & -\Gamma(r_{1})\Gamma(r_{2})\,\rho_{1}^{\star}\rho_{2}^{\star},\,\,\,\,\,\,\,\,\,\,\,x>0,\label{reflection formula Bochner Epstein two-dimensional}
\end{align}
which is equivalent to (\ref{Functional equation our Equation-1}) (see also [\cite{dirichletserisIII}, p. 342, Thm. 8.1] for an analogous argument with a single series). In Theorem \ref{selberg-chowla non diagonal theorem} we shall prove the Selberg-Chowla for the non-diagonal Epstein zeta function (\ref{Epstein zeta function non diagonal-1}) using this observation. Note also that Suzuki's proof \cite{suzuki} of the analogues of the Selberg-Chowla formula for $\zeta(s)$ and other Dirichlet series such as $L(s,\,\chi)$ and $L(s,\,a)$ (see example \ref{example cusp forms}) follows the lines given in this remark. 
\end{remark}

\bigskip{}
We now introduce a class of multidimensional Epstein zeta functions
and prove functional equations for them by mimicking the argument
given in Corollary \ref{The Analytic Continuation}. If $\left\{ \phi_{i}\right\} _{i=1,...,k}$
denotes a k-tuple of Dirichlet series satisfying Hecke's functional
equation with abcissas of absolute convergence $\sigma_{a_{i}}$
and Hecke's parameter $r_{i}$, representable by the series
\[
\phi_{i}(s)=\sum_{n=1}^{\infty}\frac{a_{i}(n)}{\lambda_{n,i}^{s}},\,\,\,\,\,\,\,\text{Re}(s)>\sigma_{a_{i}},
\]
then we can define the multidimensional Epstein zeta function as 
\begin{equation}
\mathcal{Z}_{k}\left(s;\,a_{1},...,\,a_{k};\,\lambda_{1},\,...,\,\lambda_{k}\right):=\sum_{n_{1},...,n_{k}\neq0}^{\infty}\frac{a_{1}(n_{1})\cdot...\cdot a_{k}(n_{k})}{\left(\lambda_{n_{1},1}+...+\lambda_{n_{k},k}\right)^{s}},\,\,\,\,\,\text{Re}(s)>k\sigma_{a},\label{Epstein Diagonal many dimensions}
\end{equation}
where $\sigma_{a}=\max_{1\leq i\le k}\,\sigma_{a_{i}}$.
As in representation (\ref{representation Epstein with arithmetica}),
we can write $\mathcal{Z}_{k}$ as the Dirichlet series 
\begin{equation}
\mathcal{Z}_{k}\left(s;\,a_{1},...,\,a_{k};\,\lambda_{1},\,...,\,\lambda_{k}\right)=\sum_{n=1}^{\infty}\frac{\mathfrak{U}_{k}(n)}{\Lambda_{n}^{s}},\,\,\,\,\,\text{Re}(s)>k\sigma_{a}\label{Single Dirichlet series representation}
\end{equation}
where, in analogy with $\mathfrak{U}_{2}(n)$ given in (\ref{definition general 2 square}),
\[
\mathfrak{U}_{k}(n)=\sum_{\lambda_{j_{1},1}+...+\lambda_{j_{k},k}=\Lambda_{n}}a_{1}(j_{1})\cdot...\cdot a_{k}(j_{k}).
\]

Suppose also that $\left\{ \psi_{i}\right\} _{i=1,...,k}$ denotes
the k-tuple of Dirichlet series conjugate to each $\phi_{i}$ via
Hecke's functional equation (\ref{Hecke Dirichlet series Functional}) and being representable by the series 
\[
\psi_{i}(s)=\sum_{n=1}^{\infty}\frac{b_{i}(n)}{\mu_{n,i}^{s}},\,\,\,\,\,\,\,\text{Re}(s)>\sigma_{b_{i}},
\]
then we can write $\mathcal{Z}_{k}(s;\,b_{1},...,\,b_{k};\,\mu_{1},...,\,\mu_{k})$
also as a single Dirichlet series of the form
\begin{equation}
\mathcal{Z}_{k}\left(s;\,b_{1},...,\,b_{k};\,\mu_{1},...,\,\mu_{k}\right)=\sum_{n=1}^{\infty}\frac{\mathfrak{V}_{k}(n)}{\Omega_{n,k}^{s}},\,\,\,\,\,\text{Re}(s)>k\sigma_{b},\label{Single Dirichlet series represntation with b's}
\end{equation}
where $\mathfrak{V}_{k}(n)$ plays the same role as $\mathfrak{U}_{k}(n)$
and $\sigma_{b}=\max_{1\leq i\le k}\,\sigma_{b_{i}}$.
In analogy with (\ref{Functional equation our Equation-1}) one should
expect a functional equation connecting (\ref{Single Dirichlet series representation})
and (\ref{Single Dirichlet series represntation with b's}). In fact, this is the case and the next corollary establishes it. 



\begin{corollary}[The Analytic Continuation of the Multidimensional diagonal Epstein zeta function]\label{analytic continuation multi}
Assume that $\left\{ \phi_{i}\right\} _{i=1,...,k}$ are Dirichlet
series satisfying Hecke's functional equation with parameter $r_{i}>0$
and belonging to the class $\mathcal{A}$. Then
(\ref{Epstein Diagonal many dimensions}) has an analytic continuation
as a Dirichlet series belonging to the class $\mathcal{A}$. This
continuation has at most a simple pole located at $s=\sum_{j=1}^{k}r_{j}$
and with residue 
\begin{equation}
\text{Res}_{s=\sum_{j=1}^{k}r_{j}}\,\mathcal{Z}_{k}\left(s;\,a_{1},...,\,a_{k};\,\lambda_{1},\,...,\,\lambda_{k}\right)=\frac{\prod_{j=1}^{k}\Gamma(r_{j})\,\rho_{j}}{\Gamma\left(\sum_{j=1}^{k}r_{j}\right)}.\label{Residue Multidimensional Epstein}
\end{equation}
Moreover, $\mathcal{Z}_{k}$ obeys to the functional equation 
\begin{equation}
\Gamma\left(s\right)\mathcal{Z}_{k}\left(s;\,a_{1},...,\,a_{k};\,\lambda_{1},\,...,\,\lambda_{k}\right)=\Gamma(r_{1}+...+r_{k}-s)\,\mathcal{Z}_{k}\left(r_{1}+...+r_{k}-s;\,b_{1},...,\,b_{k};\,\mu_{1},...,\mu_{k}\right).\label{Functional equation our Equation-1-1}
\end{equation}
\end{corollary}
\begin{proof}
We have seen in Corollary \ref{The Analytic Continuation} that
(\ref{Residue Multidimensional Epstein}) and (\ref{Functional equation our Equation-1-1})
hold for $k=2$. By induction, assume now that $\mathcal{Z}_{k-1}\left(s;\,a_{1},...,\,a_{k-1};\,\lambda_{1},\,...,\,\lambda_{k-1}\right)$
belongs to the class $\mathcal{A}$ and satisfies Hecke's functional
equation with parameter $r=\sum_{j=1}^{k-1}r_{j}$ and has a residue
at $s=\sum_{j=1}^{k-1}r_{j}$ given by
\[
\text{Res}_{s=\sum_{j=1}^{k-1}r_{j}}\,\mathcal{Z}_{k-1}=\frac{\prod_{j=1}^{k-1}\Gamma(r_{j})\,\rho_{j}}{\Gamma\left(\sum_{j=1}^{k-1}r_{j}\right)}.
\]

By the first Selberg-Chowla formula (\ref{Again First-1}) applied
to $\mathcal{Z}_{k-1}$ and $\phi_{k}(s)$ (both satisying Hecke's functional equation (\ref{Hecke Dirichlet series Functional}) by the induction hypothesis) and using the representation (\ref{Single Dirichlet series represntation with b's}), we see that $\mathcal{Z}_{k}$ can
be written as 
\begin{align}
-\phi_{k}(0)\,\mathfrak{\mathcal{Z}}_{k-1}(s;\,a_{1},...,\,a_{k-1};\,\lambda_{1},\,...,\,\lambda_{k-1})+\prod_{j=1}^{k-1}\Gamma(r_{j})\,\rho_{j}\,\frac{\Gamma(s-r_{1}-...-r_{k-1})}{\Gamma(s)}\,\phi_{k}\left(s-r_{1}-...-r_{k-1}\right)\nonumber \\
+\frac{2}{\Gamma(s)}\,\sum_{m,n=1}^{\infty}\mathfrak{V}_{k-1}(m)\,a_{k}(n)\,\left(\frac{\Omega_{m,k-1}}{\lambda_{n}}\right)^{\frac{s-r_{1}-...-r_{k-1}}{2}}\,K_{r_{1}+...+r_{k-1}-s}\left(2\,\sqrt{\Omega_{m,k-1}\lambda_{n}}\right).\label{Selberg Chowla for first Zl}
\end{align}

It is clear from a simple adaptation of Proposition \ref{proposition 2.1} that the last
term in (\ref{Selberg Chowla for first Zl}) represents an entire
function. Hence, the possible poles of (\ref{Epstein Diagonal many dimensions})
come from the sum of the first two terms. It is clear once more that
$\mathcal{Z}_{k}$ has removable singularities at $s=r_{1}+...+r_{k-1}-n$,
$n\in\mathbb{N}_{0}$. This comes from the functional equation for
$\phi_{k}$ and the induction hypothesis on $\mathcal{Z}_{k-1}$.
Thus, the only possible pole of $\mathcal{Z}_{k}$ comes from the factor $\phi_{k}\left(s-r_{1}-...-r_{k-1}\right)$ on the second term of the previous expression
and it is located at $s=r_{1}+...+r_{k}$. The computation of its residue gives  (\ref{Residue Multidimensional Epstein}). 

For proving the functional equation, it suffices to compare the representation
(\ref{Selberg Chowla for first Zl}) with the second representation
coming from (\ref{Again Second-1}), 
\begin{align*}
-\mathfrak{\mathcal{Z}}_{k-1}(0;\,a_{1},...,\,a_{k-1};\,\lambda_{1},\,...,\,\lambda_{k-1})\,\phi_{k}(s)+\rho_{k}\Gamma(r_{k})\,\frac{\Gamma(s-r_{k})}{\Gamma(s)}\, & \times\mathcal{Z}_{k-1}\left(s-r_{k};\,a_{1},...,\,a_{k-1};\,\lambda_{1},\,...,\,\lambda_{k-1}\right)\\
+\frac{2}{\Gamma(s)}\,\sum_{m,n=1}^{\infty}b_{k}(m)\,\mathfrak{U}_{k-1}(n)\,\left(\frac{\mu_{m}}{\Lambda_{n,k-1}}\right)^{\frac{s-r_{k}}{2}}\,K_{r_{k}-s}\left(2\sqrt{\mu_{m}\Lambda_{n,k-1}}\right), 
\end{align*}
and use the induction hypothesis. By adapting the argument
given in the proof of Corollary \ref{The Analytic Continuation} and replacing the roles of $\phi_{1}$ and $\phi_{2}$ there by $\phi_{k}$ and $\mathcal{Z}_{k-1}$,
we arrive immediately at (\ref{Functional equation our Equation-1-1}). 
\end{proof}

\begin{remark}[The dyadic Epstein zeta function] \label{dyadic epstein as remark}
Let $k\geq 0$ and consider the case where $\phi_{1}=\phi_{2}=...=\phi_{2^{k}}:=\phi\in\mathcal{A}$,
with $\phi(s)$ satisfying Hecke's functional equation with parameter $r$ and abcissa of absolute convergence $\sigma_{a}$. One
of the $2^{k+1}-1$ ways of constructing the Epstein zeta function
$\mathcal{Z}_{2^{k+1}}$ is by taking the symmetric construction
\begin{equation}
\mathcal{Z}_{2^{k}}\left(s;\,a;\,\lambda\right)=\sum_{m=1}^{\infty}\frac{\mathfrak{U}_{2^{k}}(m)}{\Lambda_{m}^{s}},\,\,\,\,\,\,\mathcal{Z}_{2^{k+1}}\left(s;\,a;\,\lambda\right)=\sum_{m,\,n\neq0}^{\infty}\frac{\mathfrak{U}_{2^{k}}(m)\,\mathfrak{U}_{2^{k}}(n)}{\left(\Lambda_{m}+\Lambda_{n}\right)^{s}},\label{dyadic Epstein 2^k}
\end{equation}
with each series converging absolutely for $\text{Re}(s)>2^{k}\sigma_{a},\,\,2^{k+1}\sigma_{a}$ respectively (note the convention $\mathcal{Z}_{2^{0}}(s;\,a;\,\lambda):=\phi(s)$). 
Each $\mathcal{Z}_{2^{k}}(s;\,a;\,\lambda)$ plays
the role of $\phi_{1}$ and $\phi_{2}$ in Theorem \ref{theorem 2.1} and Corollary
\ref{The Analytic Continuation}. By an application of Corollary \ref{analytic continuation multi}, it is clear that we can
write a Selberg-Chowla formula in the form
\begin{align}
\Gamma(s)\,\mathcal{Z}_{2^{k+1}}\left(s;\,a;\,\lambda\right) & =-\mathcal{Z}_{2^{k}}\left(0;\,a;\,\lambda\right)\,\Gamma(s)\,\mathcal{Z}_{2^{k}}\left(s;\,a;\,\lambda\right)+\Gamma^{2^{k}}(r)\,\rho^{2^{k}}\,\Gamma(s-2^{k}r)\,\mathcal{Z}_{2^{k}}\left(s-2^{k}r;\,a;\,\lambda\right)\,\nonumber \\
 & +2\,H_{2^{k}r}(s;\,a;\,\lambda),\label{Selberg Chowla for diagonal 2k}
\end{align}
where $H_{2^{k}r}(s;\,a;\,\lambda)$ denotes the entire function
\begin{equation}
H_{2^{k}r}(s;\,a;\,\lambda)=\sum_{m,n=1}^{\infty}\mathfrak{V}_{2^{k}}(m)\,\mathfrak{U}_{2^{k}}(n)\,\left(\frac{\Omega_{m}}{\Lambda_{n}}\right)^{\frac{s}{2}-2^{k-1}r}\,K_{2^{k}r-s}\left(2\sqrt{\Omega_{m}\Lambda_{n}}\right),\label{H1-1}
\end{equation}
with the arithmetical functions $\mathfrak{U}_{2^{k}}(n)$ and $\mathfrak{V}_{2^{k}}(n)$
appearing respectively in the Dirichlet series representations (\ref{Single Dirichlet series representation})
and (\ref{Single Dirichlet series represntation with b's}). 
From Corollary \ref{analytic continuation multi}, we know that the multidimensional Epstein
zeta function $\mathcal{Z}_{2^{k}}\left(s;\,a;\,\lambda\right)$
satisfies Hecke's functional equation (\ref{Functional equation our Equation-1-1})
with parameter $2^{k}r$ and so its critical line is $\text{Re}(s)=2^{k-1}r$. From the Phragm\'en-Lindel\"of principle (\ref{Lindelof Phragmen for Hecke}), we have that, for any $\delta>0$ and $2^{k}r-2^{k}\sigma_{a}-\delta\leq\sigma\leq2^{k}\sigma_{a}+\delta$,
\begin{equation}
\mathcal{Z}_{2^{k}}\left(\sigma+it;\,a;\,\lambda\right)=O\left(|t|^{2^{k}\sigma_{a}-\sigma+\delta}\right),\,\,\,\,\,\,|t|\rightarrow\infty.\label{Estimate general}
\end{equation}

Equations (\ref{Selberg Chowla for diagonal 2k}) and (\ref{Estimate general}) will
be used in the penultimate section of the paper, where we shall prove
a class of Hardy Theorems for Dirichlet series with narrow critical
strips (see Lemma \ref{lemma integral representation} and Theorem \ref{deuring to hold}).



\end{remark}

\begin{center}\section{The analytic continuation of a Non-Diagonal Epstein Zeta function} \label{section 3} \end{center}
Although the previous section furnished an analogue of the Epstein
$\zeta-$function, the double Dirichlet series there presented is
only useful when our study is reduced to diagonal quadratic forms.

In order to have a proper environment to develop Selberg-Chowla formulas
for Epstein zeta functions attached to non-diagonal quadratic forms, we need to replicate in a certain way the
conditions obeyed by the Riemann $\zeta-$function, which is behind
the analytic continuation of the classical Epstein zeta function $Z_{2}(s,\,Q)$. Doing so requires to introduce the following subclass $\mathcal{B}$ of Bochner
Dirichlet series, which has a similar role as the class $\mathcal{A}$ had in the previous section. 
\bigskip{}


\begin{definition}[A Subclass of Bochner Dirichlet series] \label{definition 3.1}
Let $\phi(s)$ be a Dirichlet series satisfying Definition \ref{definition 1.1.} with
$\Delta(s)=\Gamma\left(\frac{s+\delta}{2}\right)$, $\delta\in\{0,1\}$,
$r=1$ (Bochner class). We say that $\phi$ belongs to the class $\mathcal{B}$
if 
\begin{itemize}
\item For $\delta=0$, $\phi$ and $\psi$ have analytic continuations to the entire complex plane and are analytic everywhere in $\mathbb{C}$
except for possible simple poles located at $s=1$ with residues $\rho$
and $\rho^{\star}$ respectively. For $\delta=1$, $\phi$ and $\psi$
can be analytically continued as entire Dirichlet series. 
\item $a(n)$ can be defined at $0$ as $a(0):=-2\phi(-\delta)$.
Moreover, $a(n)$ can be defined for negative values of $n$ as $a(-n)=(-1)^{\delta}\,a(n)$
and the sequences $\left(\lambda_{n}\right)_{n\in\mathbb{N}}$ and
$\left(\mu_{n}\right)_{n\in\mathbb{N}}$ can be extended to $n\in\mathbb{Z}$
with the values $\lambda_{0}=\mu_{0}=0$ and $\lambda_{-n}=-\lambda_{n}$,
$\mu_{-n}=-\mu_{n}$. A similar extension holds for $b(n)$ with $b(0):=-2\psi(-\delta)$ and $b(-n)$ analogously taken. 
\end{itemize}
\end{definition}
\begin{remark} \label{easy computations class B}
Note that definition \ref{definition 3.1} mimics in some way the properties that the
Dirichlet $L-$functions and the Riemann $\zeta-$function have. By
the functional equation for Bochner Dirichlet series (\ref{This is the first Bochner ever}),
it is effortless to see that $\phi(0)=-\frac{\rho^{\star}}{2}\sqrt{\pi}$,
while $\psi(0)=-\frac{\rho}{2}\sqrt{\pi}$. Moreover, if $\phi(s)$
is entire, then it comes from the functional equation and the second
item that $a(0)=0$. We remark that, in analogy with Theorem \ref{theorem 2.1}, it is possible to prove
a general Selberg-Chowla formula for $\mathcal{Z}_{2}(s;\,Q;\,a_{1},\,a_{2};\,\lambda,\,\lambda^{\prime})$
with the general setting of Definition \ref{definition 1.1.}. However, in order to simplify
the foregoing argument, we shall prove a Selberg-Chowla formula for
(\ref{Epstein zeta function non diagonal-1}) assuming apriori the conditions
of the previous definition. 

\end{remark}


To prove a Selberg-Chowla formula for the non-diagonal Epstein zeta
function $\mathcal{Z}_{2}(s;\,Q;\,a_{1},\,a_{2};\,\lambda,\,\lambda^{\prime})$,
we use a generalization of Bochner's modular relation (\ref{Bochner Modular relation at intro}) for the class
$\mathcal{B}$. In the well-known case $\phi_{1}(s)=\phi_{2}(s)=\pi^{-\frac{s}{2}}\zeta(s)$
the formula given below reduces to Jacobi's reflection formula. We
remark that, when the non-diagonal Epstein zeta function reduces to (\ref{Classical Epstein}), our method of proof of the generalized
Selberg-Chowla formula for $\mathcal{Z}_{2}(s;\,Q;\,a_{1},\,a_{2};\,\lambda,\,\lambda^{\prime})$
reduces to the one presented on the second section of Selberg and
Chowla's paper \cite{selberg_chowla} \footnote{It is also possible to generalize to this scope the proof given by
Taylor \cite{taylor_epstein}, mentioned at the Introduction. We remark
that Taylor's proof is incomplete, since it only provides the Selberg-Chowla
formula for the subset of positive definite binary quadratic forms
satisfying $b^{2}<2ac$ (see page 182 on Taylor's paper and the
assumption on the variable $Y$). However, as remarked in the historical
introduction, the purpose of Taylor's analysis was only to obtain
Kober's result \cite{kober_epstein}, which was itself a particular case
of the now called Selberg-Chowla formula. Nevertheless, a simple argument
of continuation provided by the contiguous relations for the hypergeometric
function $_{2}F_{1}$ {[}\cite{ryzhik}, p. 1009, eq. (9.132.2){]} allows
to extend Taylor's proof to all positive definite binary quadratic
forms. }. 

\begin{lemma}[A generalization of the Theta reflection formula for the Class $\mathcal{B}$] \label{lemma 3.1}
Let $\phi(s)$ be a Dirichlet series belonging to the class $\mathcal{B}$.
If $a(n)$ is an even arithmetical function (i.e., $\delta=0$),
then, for any $x>0$, $\text{Re}(\alpha)>0$ and $\beta \in \mathbb{C}$, the following identity
holds
\begin{equation}
\sum_{n=1}^{\infty}a(n)\,e^{-\alpha\lambda_{n}^{2}x^{2}}\,\cos\left(\beta\lambda_{n}x\right)=\phi(0)+\sqrt{\frac{\pi}{\alpha}}\frac{\rho}{2\,x}\,e^{-\frac{\beta^{2}}{4\alpha}}+\frac{1}{2\sqrt{\alpha}\,x}\,\sum_{n\neq0}b(n)\,e^{-\frac{1}{\alpha}\left(\frac{\mu_{n}}{x}+\frac{\beta}{2}\right)^{2}}.\label{even final form general Jacobi}
\end{equation}

If, by other hand, $a(n)$
is an odd arithmetical function, we have the identity 
\begin{equation}
\sum_{n=1}^{\infty}a(n)\,e^{-\alpha\lambda_{n}^{2}x^{2}}\,\sin\left(\beta\lambda_{n}x\right)=\frac{1}{2\,\sqrt{\alpha}x}\,\sum_{n\neq0}b(n)\,e^{-\frac{1}{\alpha}\left(\frac{\mu_{n}}{x}+\frac{\beta}{2}\right)^{2}}.\label{Reflection formula theta generalized odd case}
\end{equation}
\end{lemma}
\begin{proof}
As in the proof of Theorem \ref{theorem 2.1}, we start by choosing a suitable parameter
$\mu$ which allows absolute convergence in every side of the equality
to be proven. In our case, assume $\mu>\max\left\{ \sigma_{a},\,\sigma_{b},\,1\right\} $
and use the observation that, for $\text{Re}(\alpha)>0$, $\sigma=\text{Re}(z)>0$
and $\beta\in\mathbb{C}$, the following Mellin inverse representation
holds {[\cite{ryzhik}, p. 503, eq. 3.952.8]}
\begin{equation}
e^{-\alpha x^{2}}\cos(\beta x)=\frac{e^{-\frac{\beta^{2}}{4\alpha}}}{4\pi i}\,\int_{\sigma-i\infty}^{\sigma+i\infty}\alpha^{-\frac{z}{2}}\Gamma\left(\frac{z}{2}\right)\,_{1}F_{1}\left(\frac{1-z}{2};\,\frac{1}{2};\frac{\beta^{2}}{4\alpha}\right)\,x^{-z}dz,\label{Mellin inverse 1F1}
\end{equation}
which can be proved by invoking (\ref{invoking a relation on Erdeliy table not ryzhik})
together with Kummer's transformation (\ref{Kummer confluent transformation}).
Although the key of this comparison is Mellin's inversion formula,
(\ref{Mellin inverse 1F1}) could be alternatively proved by invoking
the absolutely convergent power series for $_{1}F_{1}$ and then interchanging
the orders of summation and integration. 

From the fact that the integrand in (\ref{Mellin inverse 1F1}) is the Mellin transform of a smooth function and decays faster than any polynomial as $|s|\rightarrow \infty$, we can write the left-hand
side of (\ref{even final form general Jacobi}) as the contour integral 
\begin{equation}
\sum_{n=1}^{\infty}a(n)\,e^{-\alpha\lambda_{n}^{2}x^{2}}\,\cos\left(\beta\lambda_{n}x\right)=\frac{e^{-\frac{\beta^{2}}{4\alpha}}}{4\pi i}\,\int_{\mu-i\infty}^{\mu+i\infty}\Gamma\left(\frac{z}{2}\right)\phi(z)\,_{1}F_{1}\left(\frac{1-z}{2};\,\frac{1}{2};\frac{\beta^{2}}{4\alpha}\right)\,(\sqrt{\alpha}x)^{-z}\,dz.\label{Representation on the way: the first}
\end{equation}

Proceeding once more as in Theorem \ref{theorem 2.1}, let us now move the line of integration
on (\ref{Representation on the way: the first}) to $\text{Re}(z)=1-\mu<0$
and integrate along a positively oriented rectangular contour $\mathcal{R}$
containing the vertices $\mu\pm iT$ and $1-\mu\pm iT$, $T>0$. Once
again, we are allowed to invoke the Phragm\'en-Lindel\"of principle for $\phi(z)$ (\ref{Lindelof Pgragmen for Bochner}), together with the decay properties of the integrand on the right-hand side of (\ref{Mellin inverse 1F1}) in order to conclude that the integrals
along the horizontal segments $\left[1-\mu-iT,\,\mu-iT\right]$ and
$\left[\mu+iT,\,1-\mu+iT\right]$ must vanish when $T\rightarrow\infty$. 

By the residue Theorem, we obtain the equality
\begin{align}
\sum_{n=1}^{\infty}a(n)\,e^{-\alpha\lambda_{n}^{2}\,x^{2}}\cos(\beta\lambda_{n}x) & =\frac{e^{-\frac{\beta^{2}}{4\alpha}}}{4\pi i}\,\,\int_{1-\mu-i\infty}^{1-\mu+i\infty}\,\Gamma\left(\frac{z}{2}\right)\,\phi(z)\,_{1}F_{1}\left(\frac{1-z}{2};\,\frac{1}{2};\,\frac{\beta^{2}}{4\alpha}\right)\,\left(\sqrt{\alpha}x\right)^{-z}\,dz\nonumber \\
+\frac{e^{-\frac{\beta^{2}}{4\alpha}}}{2}\,\sum_{\rho\in\mathcal{R}} & \text{Res}_{z=\rho}\left\{ \Gamma\left(\frac{z}{2}\right)\phi(z)\,_{1}F_{1}\left(\frac{1-z}{2};\,\frac{1}{2};\frac{\beta^{2}}{4\alpha}\right)\,(\sqrt{\alpha}x)^{-z}\,dz\right\} .\label{first sum}
\end{align}

By the conditions of the class $\mathcal{B}$, it is evident that
$\phi(z)$ has simple zeros located at the negative even integers
if $a(n)=a(-n)$. Also, from the fact that $\,_{1}F_{1}\left(\frac{1-z}{2};\,\frac{1}{2};\frac{\beta^{2}}{4\alpha}\right)$
is an entire function of $z$, we observe that all the poles that we need
to take into consideration on the shift of the line
are located at $z=0$ and $z=1$. It is immediate to see that their
residues are
\[
\text{Res}_{z=0}=2\phi(0)\,e^{\frac{\beta^{2}}{4\alpha}},\,\,\,\,\,\,\text{Res}_{z=1}=\sqrt{\frac{\pi}{\alpha}}\frac{\rho}{x}.
\]

Now, invoking the functional equation for $\phi(z)$ (\ref{This is the first Bochner ever})
with $\delta=0$ and performing the change of variables $z\leftrightarrow1-z$,
we see that the first term on the right-hand side of (\ref{first sum})
is equal to 
\begin{equation}
\frac{e^{-\frac{\beta^{2}}{4\alpha}}}{2\sqrt{\alpha}\,x}\,\sum_{n=1}^{\infty}\frac{b(n)}{2\pi i}\,\int_{\mu-i\infty}^{\mu+i\infty}\,\Gamma\left(\frac{z}{2}\right)\,_{1}F_{1}\left(\frac{z}{2};\,\frac{1}{2};\,\frac{\beta^{2}}{4\alpha}\right)\,\left(\frac{\mu_{n}}{\sqrt{\alpha}x}\right)^{-z}\,dz+\phi(0)+\sqrt{\frac{\pi}{\alpha}}\,\frac{\rho}{2x}\,e^{-\frac{\beta^{2}}{4\alpha}}.\label{term reminina g}
\end{equation}

The contour integral inside the infinite series can be evaluated
via (\ref{Mellin inverse 1F1}) upon a substitution of $\beta$ by
$i\beta$ and the use of Kummer's reflection formula (\ref{Kummer confluent transformation}).
Nevertheless, for completeness, we may present a closed evaluation.
Note that the power series defining the confluent Hypergeometric function
is absolutely convergent for arbitrary $\beta\in\mathbb{C}$ and $\alpha$
such that $\text{Re}(\alpha)>0$. Moreover, a simple application of
Stirling's formula (\ref{Stirling exact form on Introduction}) allows to reverse the orders of summation and integration
once we employ this series representation and we get 
\begin{align}
\frac{1}{2\pi i}\,\int_{\mu-i\infty}^{\mu+i\infty}\,\Gamma\left(\frac{z}{2}\right)\,_{1}F_{1}\left(\frac{z}{2};\,\frac{1}{2};\,\frac{\beta^{2}}{4\alpha}\right)\,y^{-z}\,dz & =\sum_{k=0}^{\infty}\frac{\sqrt{\pi}}{k!\,\left(\frac{1}{2}\right)_{k}}\,\left(\frac{\beta^{2}}{4\alpha}\right)^{k}\,\int_{\mu-i\infty}^{\mu+i\infty}\Gamma\left(\frac{z}{2}+k\right)\,y^{-z}dz\nonumber \\
=\sum_{k=0}^{\infty}\frac{\sqrt{\pi}}{k!\,\left(\frac{1}{2}\right)_{k}}\,\left(\frac{\beta^{2}y^{2}}{4\alpha}\right)^{k}\,\int_{\mu+2k-i\infty}^{\mu+2k+i\infty}\Gamma\left(\frac{z}{2}\right)\,y^{-z}dz & =2\,e^{-y^{2}}\sum_{k=0}^{\infty}\frac{1}{(2k)!}\,\left(\frac{\beta y}{\sqrt{\alpha}}\right)^{2k}=2\,e^{-y^{2}}\cosh\left(\frac{\beta y}{\sqrt{\alpha}}\right).\label{simple evalaution 1F1}
\end{align}

Combining (\ref{simple evalaution 1F1}), (\ref{term reminina g})
and (\ref{first sum}) we obtain the generalization of the reflection
formula for Jacobi's $\theta-$function, 
\begin{equation}
\sum_{n=1}^{\infty}a(n)\,e^{-\alpha\lambda_{n}^{2}x^{2}}\,\cos\left(\beta\lambda_{n}x\right)=\phi(0)+\sqrt{\frac{\pi}{\alpha}}\frac{\rho}{2\,x}\,e^{-\frac{\beta^{2}}{4\alpha}}+\frac{e^{-\frac{\beta^{2}}{4\alpha}}}{\sqrt{\alpha}x}\,\sum_{n=1}^{\infty}b(n)\,e^{-\frac{\mu_{n}^{2}}{\alpha\,x^{2}}}\,\cosh\left(\frac{\beta\mu_{n}}{\alpha\,x}\right).\label{Generalized Theta summation}
\end{equation}

Since $b(n)$ is even by hypothesis and we have assumed
that $\mu_{-n}=-\mu_{n}$, we can rewrite the last series in a symmetric
form by summing over the non-zero integers. This gives the desired
reflection formula (\ref{even final form general Jacobi}).

For the
odd case the proof is even simpler: just proceed as before and use an integral
representation similar to (\ref{Mellin inverse 1F1}), i.e., {[}\cite{ryzhik},
p. 503, eq. 3.952.7{]}
\begin{equation}
e^{-\alpha x^{2}}\sin(\beta x)=\frac{\beta e^{-\frac{\beta^{2}}{4\alpha}}}{4\pi i}\int_{\sigma-i\infty}^{\sigma+i\infty}\alpha^{-\frac{s+1}{2}}\Gamma\left(\frac{1+s}{2}\right)\,_{1}F_{1}\left(1-\frac{s}{2};\,\frac{3}{2};\,\frac{\beta^{2}}{4\alpha}\right)\,x^{-s}ds,\label{second representation involving sine}
\end{equation}
valid for $\text{Re}(\alpha)>0$ and $\sigma>-1$. 
\end{proof}


\begin{remark} \label{hurwitz remark}
The previous lemma is connected with N. J. Fine's proof of the functional
equation for the classical Hurwitz zeta function $\zeta(s,\,\alpha)$ {[}\cite{fine},
p. 361, eq. (3){]}. As done by Berndt for the case of Dirichlet series
satisfying Hecke's functional equation \cite{dirichlet and hecke},
we could define the generalized Hurwitz zeta function\footnote{The Dirichlet series appearing in (\ref{Hurwitz bochner class}) is
not a generalized/perturbed Dirichlet series in the sense of [\cite{dirichlet and hecke}, p.309] since it is not constructed from a sequence
$(\lambda_{n})_{n\in\mathbb{N}}$ attached to a Dirichlet series satisfying
Hecke's functional equation. 
} in the following way 
\begin{equation}
\phi(s,\,\alpha)=\sum_{n=1}^{\infty}\frac{a(n)}{(\lambda_{n}+\alpha)^{s}},\,\,\,\,\,\,\text{Re}(s)>\sigma_{a},\,\,\,\alpha>0\label{Hurwitz bochner class}
\end{equation}
and derive the analytic continuation and functional equation for it
by using the previous lemma. This observation connects our forthcoming
proof of the Selberg-Chowla formula with the proofs by Berndt and Kuzumaki,
which naturally employ the functional equation for $\zeta(s,\,\alpha)$ (see {[}\cite{dirichletserisVI}, p. 160{]} and {[}\cite{Kuzumaki},
eq. (16), (24) and (25){]}. 


\end{remark}

\begin{theorem}[Selberg-Chowla for positive definite binary quadratic forms]  \label{selberg-chowla non diagonal theorem} 
Let $\phi_{1}$ and $\phi_{2}$ be two Dirichlet series belonging
to the class $\mathcal{B}$ and $s$ be a complex number such that
$\text{Re}(s)>\mu>\max\left\{ \sigma_{a},\,\sigma_{b}\right\} $.
If $a_{1}(n),\,a_{2}(n)$ are even arithmetical functions,
then we have the following Selberg-Chowla formula 
\begin{align}
a^{s}\,\Gamma(s)\,\mathcal{Z}_{2}\left(s;\,Q;\,a_{1},\,a_{2};\,\lambda,\,\lambda^{\prime}\right) & =-4\mathfrak{\mathfrak{\phi}}_{2}(0)\,\Gamma(s)\phi_{1}(2s)+2\sqrt{\pi}\,\rho_{1}\,\Gamma\left(s-\frac{1}{2}\right)k^{1-2s}\phi_{2}(2s-1)\nonumber \\
+4k^{\frac{1}{2}-s}\, & \sum_{j=1}^{\infty}\sigma_{2s-1}\left(\nu_{j};\,b_{1},\,a_{2};\,\mu,\,\lambda^{\prime}\right)\,\cos\left(\frac{b\nu_{j}}{a}\right)\,\nu_{j}^{\frac{1}{2}-s}\,\,K_{s-\frac{1}{2}}\left(2k\nu_{j}\right),\label{First Selberg Chowla even non diagonal}
\end{align}
where $\left(\nu_{j}\right)_{j\in\mathbb{N}}=\left\{ \mu_{m}\lambda_{n}^{\prime}\right\} _{m,n=1}^{\infty}$
is the product sequence arranged in increasing order and $\sigma_{2s-1}(\nu_{j})$
denotes the generalized weighted divisor function of power $2s-1$
(see Definition \ref{definition generalized divisor 2.1.}). Moreover, $d:=b^{2}-4ac$ is the discriminant
of the quadratic form $Q$ and $k^{2}:=|d|/4a^{2}$. 

Analogously, if $a_{1}(n)$ and $a_{2}(n)$ are odd arithmetical functions,
we have the Selberg-Chowla formula 
\begin{equation}
a^{s}\,\Gamma(s)\,\mathcal{Z}_{2}\left(s;\,Q;\,a_{1},\,a_{2};\,\lambda,\,\lambda^{\prime}\right)=-4\,k^{\frac{1}{2}-s}\sum_{j=1}^{\infty}\sigma_{2s-1}\left(\nu_{j};\,b_{1},\,a_{2};\,\mu,\,\lambda^{\prime}\right)\,\sin\left(\frac{b\nu_{j}}{a}\right)\,\nu_{j}^{\frac{1}{2}-s}\,\,K_{s-\frac{1}{2}}\left(2k\nu_{j}\right).\label{First Selberg Chowla odd}
\end{equation}

Under the same assumptions, $\mathcal{Z}_{2}\left(s;\,Q;\,a_{1},\,a_{2};\,\lambda,\,\lambda^{\prime}\right)$
can be described by the second Selberg-Chowla formulas, 
\begin{align}
c^{s}\,\Gamma(s)\,\mathcal{Z}_{2}\left(s;\,Q;\,a_{1},\,a_{2};\,\lambda,\,\lambda^{\prime}\right) & =-4\mathfrak{\mathfrak{\phi}}_{1}(0)\,\Gamma(s)\phi_{2}(2s)+2\sqrt{\pi}\,\rho_{2}\,\Gamma\left(s-\frac{1}{2}\right)\,k^{\prime1-2s}\phi_{1}(2s-1)\nonumber \\
+4\,k^{\prime\frac{1}{2}-s}\, & \sum_{j=1}^{\infty}\sigma_{2s-1}\left(\nu_{j}^{\prime};\,b_{2},\,a_{1};\,\mu^{\prime},\,\lambda\right)\,\cos\left(\frac{b\nu_{j}^{\prime}}{c}\right)\,\nu_{j}^{\prime\frac{1}{2}-s}\,\,K_{s-\frac{1}{2}}\left(2k^{\prime}\nu_{j}^{\prime}\right),\label{Second Selberg Chowla even non-diagonal}
\end{align}
for $a_{1}(n),\,a_{2}(n)$ being even arithmetical functions and 
\begin{equation}
c^{s}\,\Gamma(s)\,\mathcal{Z}_{2}\left(s;\,Q;\,a_{1},\,a_{2};\,\lambda,\,\lambda^{\prime}\right)=-4\,k^{\prime\frac{1}{2}-s}\sum_{j=1}^{\infty}\sigma_{2s-1}\left(\nu_{j}^{\prime};\,b_{2},\,a_{1};\,\mu^{\prime},\,\lambda\right)\,\sin\left(\frac{b\nu_{j}^{\prime}}{c}\right)\,\nu_{j}^{\prime\frac{1}{2}-s}\,\,K_{s-\frac{1}{2}}\left(2k^{\prime}\nu_{j}^{\prime}\right),\label{2nd selberg Chowla odd}
\end{equation}
for $a_{1}(n),\,a_{2}(n)$ representing odd arithmetical functions. Here, $k^{\prime2}:=|d|/4c^{2}$
and the sequence $\nu_{j}^{\prime}$ is analogously defined by $\nu_{j}^{\prime}=\mu_{m}^{\prime}\lambda_{n}$. 
\end{theorem}

\begin{proof}
Throughout this proof we shall assume that $a_{i}(n)$,
$i=1,\,2$, are even arithmetical functions, as analogous computations
can be given for the odd case. As in the proof of Theorem \ref{theorem 2.1}, we
start the argument by writing the double series (\ref{Epstein zeta function non diagonal-1})
as 
\[
\mathcal{Z}_{2}\left(s;\,Q;\,a_{1},\,a_{2};\,\lambda,\,\lambda^{\prime}\right)=2\,a_{2}(0)a^{-s}\,\phi_{1}(2s)+2\,a_{1}(0)c^{-s}\,\phi_{2}(2s)+\sum_{m\neq0,\,n\neq0}\frac{a_{1}(m)\,a_{2}(n)}{\left(a\lambda_{m}^{2}+b\lambda_{m}\lambda_{n}^{\prime}+c\lambda_{n}^{\prime2}\right)^{s}}.
\]

Also, we choose a parameter $\mu$ satisfying $\mu>\max\left\{ \sigma_{a},\,\sigma_{b}\right\} $
and a fixed $s\in\mathbb{C}$ such that $\text{Re}(s)>\mu$. Of course,
in order to compute $\mathcal{Z}_{2}$, we need to evaluate the
double series 
\[
\sum_{n\neq0}a_{2}(n)\,\sum_{m\neq0}\frac{a_{1}(m)}{\left(a\lambda_{m}^{2}+b\lambda_{m}\lambda_{n}^{\prime}+c\lambda_{n}^{\prime2}\right)^{s}}.
\]

In analogy to the generalized Dirichlet series appearing in the proof
of Theorem \ref{theorem 2.1}, for a fixed $n$ let us denote by $\ell_{n}\left(s,\,Q,\,a_{1}\right)$
the inner infinite series with respect to $m$ appearing above. For 
$\text{Re}(s)>\mu$, $\ell_{n}(s,\,Q,\,a_{1})$ is explicitly
given by  
\begin{align*}
\ell_{n}\left(s,\,Q,\,a_{1}\right) & =\sum_{m\neq0}\frac{a_{1}(m)}{\left(a\lambda_{m}^{2}+b\lambda_{m}\lambda_{n}^{\prime}+c\lambda_{n}^{\prime2}\right)^{s}}=\frac{1}{\Gamma(s)}\,\sum_{m\neq0}a_{1}(m)\,\int_{0}^{\infty}y^{s-1}e^{-(a\lambda_{m}^{2}+b\lambda_{m}\lambda_{n}^{\prime}+c\lambda_{n}^{\prime2})\,y}dy\\
 & =\frac{1}{\Gamma(s)}\,\int_{0}^{\infty}y^{s-1}\,e^{-\frac{|d|}{4a}\lambda_{n}^{\prime2}\,y}\,\sum_{m\neq0}a_{1}(m)\,e^{-a\,y\,\left(\lambda_{m}+\frac{b}{2a}\lambda_{n}^{\prime}\right)^{2}}\,dy
\end{align*}
where the last step can be justified by the fact that $\text{Re}(s)>\mu>\sigma_{a}$
and by absolute convergence. Invoking the reflection formula (\ref{even final form general Jacobi})
with the roles of $\phi$ and $\psi$ reversed and taking the substitutions
$\alpha=\frac{1}{ay}>0$, $\beta=\frac{b}{a}\lambda_{n}^{\prime}$ and 
$x=1$, we obtain that  
\[
\sum_{m\neq0}a_{1}(m)\,e^{-ay\,\left(\lambda_{m}+\frac{b}{2a}\lambda_{n}^{\prime}\right)^{2}}=\frac{2}{\sqrt{ay}}\,\left\{ \sum_{m=1}^{\infty}b_{1}(m)\,e^{-\frac{\mu_{m}^{2}}{ay}}\,\cos\left(\frac{b}{a}\mu_{m}\lambda_{n}^{\prime}\right)-\psi_{1}(0)-\frac{\rho_{1}^{\star}\sqrt{\pi ay}}{2}\,e^{-\frac{b^{2}}{4a}\lambda_{n}^{\prime2}\,y}\right\} 
\]
and so $\ell_{n}(s,\,Q,\,a_{1})$ can be given by 
\begin{align*}
& \frac{2a^{-\frac{1}{2}}}{\Gamma(s)}\,\int_{0}^{\infty}y^{s-\frac{3}{2}}\,e^{-\frac{|d|}{4a}\lambda_{n}^{\prime2}\,y}\,\left\{ \sum_{m=1}^{\infty}b_{1}(m)\,e^{-\frac{\mu_{m}^{2}}{ay}}\,\cos\left(\frac{b}{a}\mu_{m}\lambda_{n}^{\prime}\right)-\psi_{1}(0)-\frac{\rho_{1}^{\star}\sqrt{\pi ay}}{2}\,e^{-\frac{b^{2}}{4a}\lambda_{n}^{\prime2}\,y}\right\} \,dy\\
 & =\frac{2\phi_{1}(0)}{\lambda_{n}^{\prime2s}}\,c^{-s}-\frac{2a^{-\frac{1}{2}}\,\Gamma\left(s-\frac{1}{2}\right)\psi_{1}(0)}{\Gamma(s)\,\lambda_{n}^{\prime2s-1}}\,\left(\frac{|d|}{4a}\,\right)^{\frac{1}{2}-s}+\frac{2\,a^{-\frac{1}{2}}}{\Gamma(s)}\sum_{m=1}^{\infty}b_{1}(m)\,\cos\left(\frac{b}{a}\mu_{m}\lambda_{n}^{\prime}\right)\,\int_{0}^{\infty}y^{s-\frac{3}{2}}e^{-\frac{\mu_{m}^{2}}{ay}}\,e^{-\frac{|d|}{4a}\lambda_{n}^{\prime2}y}\,dy\\
 & =\frac{2\phi_{1}(0)}{\lambda_{n}^{\prime2s}}\,c^{-s}+\rho_{1}\sqrt{\pi}\,\frac{a^{-\frac{1}{2}}\,\Gamma\left(s-\frac{1}{2}\right)}{\Gamma(s)\,\lambda_{n}^{\prime2s-1}}\,\left(\frac{|d|}{4a}\,\right)^{\frac{1}{2}-s}+\frac{4\,\left(ka\right)^{\frac{1}{2}-s}}{\Gamma(s)\sqrt{a}}\,\sum_{m=1}^{\infty}b_{1}(m)\,\cos\left(\frac{b}{a}\mu_{m}\lambda_{n}^{\prime}\right)\left(\frac{\mu_{m}}{\lambda_{n}^{\prime}}\right)^{s-\frac{1}{2}}K_{s-\frac{1}{2}}\left(2k\mu_{m}\lambda_{n}^{\prime}\right)
\end{align*}
with the second equality being justified again by absolute convergence.
We have also used the fact that $\psi_{1}(0)=-\frac{\rho_{1}}{2}\sqrt{\pi}$
and the integral representation of the modified Bessel function given
in (\ref{convolution exponentials equals K nu}). 

\bigskip{}

It suffices now to multiply the previous equality by $a_{2}(n)$
and to sum over the index $n\in\mathbb{Z}\setminus\{0\}$. Since we
are assuming that $\text{Re}(s)>\mu$, all the Dirichlet series involved
in this process converge absolutely. Similarly to what we have done
in the final part of the proof of Theorem \ref{theorem 2.1}, we can argue that
the resulting double series of the form 
\[
\sum_{m,n=1}^{\infty}b_{1}(m)\,a_{2}(n)\,\cos\left(\frac{b}{a}\mu_{m}\lambda_{n}^{\prime}\right)\,\left(\frac{\mu_{m}}{\lambda_{n}^{\prime}}\right)^{s-\frac{1}{2}}\,\,K_{s-\frac{1}{2}}\left(2k\,\mu_{m}\lambda_{n}^{\prime}\right)
\]
will converge absolutely due to the classical estimate for the Modified
Bessel function (\ref{Classical Estimate Bessel}). From the condition
$\text{Re}(s)>\mu>\max\left\{ \sigma_{a},\,\sigma_{b}\right\} $,
we see that the series defining the non-diagonal Epstein zeta function
is absolutely convergent as a double series and the order of the summation
of both series can be interchanged. This gives the second Selberg-Chowla
formula for the even case (\ref{Second Selberg Chowla even non-diagonal}),
which is obtained after replacing the role of $\phi_{1}$ by the one
of $\phi_{2}$ and then replace $a$ by $c$ and $k$ by $k^{\prime}$.
Invoking the definition of the generalized divisor function (\ref{On a generalized Divisor Function}), 
we obtain all the Selberg-Chowla formulas in the form above stated.
\end{proof}


By analytic continuation, it is possible to establish the validity
of the formulas (\ref{First Selberg Chowla even non diagonal}, \ref{First Selberg Chowla odd},
\ref{Second Selberg Chowla even non-diagonal}, \ref{2nd selberg Chowla odd})
for all complex values of $s$. Mimicking the steps leading to the proof of Proposition \ref{proposition 2.1}, it is
simple to see that the infinite series involving the generalized divisor
function are entire functions of $s$ and obey to a reflection formula
which is analogous to (\ref{Reflection Formula-2}). 

Similarly to the entire functions $H_{r_{1}}(s;\,\cdot)$ and $H_{r_{2}}(s;\,\cdot)$ (\ref{H1}, \ref{H2}), we
may see that the series on the first Selberg-Chowla formulas (\ref{First Selberg Chowla even non diagonal})
and (\ref{First Selberg Chowla odd}) can be written in a unified
way as
\begin{equation*}
H_{1}^{\delta}\left(s;\,Q;\,b_{1},\,a_{2}\right)=\sum_{m,n=1}^{\infty}b_{1}(m)\,a_{2}(n)\,\left\{ (1-\delta)\cos\left(\frac{b}{a}\mu_{m}\lambda_{n}^{\prime}\right)+\delta\sin\left(\frac{b}{a}\mu_{m}\lambda_{n}^{\prime}\right)\right\} \,\left(\frac{\mu_{m}}{\lambda_{n}^{\prime}}\right)^{s-\frac{1}{2}}\,\,K_{s-\frac{1}{2}}\left(2k\,\mu_{m}\lambda_{n}^{\prime}\right).\label{general H 1 with deltinha}
\end{equation*}
Analogously, the divisor
series on the second Selberg-Chowla formulas (\ref{Second Selberg Chowla even non-diagonal})
and (\ref{2nd selberg Chowla odd}) can be also written as 
\begin{equation*}
H_{2}^{\text{\ensuremath{\delta}}}\left(s;\,Q;\,b_{2},\,a_{1}\right)=\sum_{m,n=1}^{\infty}b_{2}(m)\,a_{1}(n)\,\left\{(1-\delta)\cos\left(\frac{b}{c}\mu_{m}^{\prime}\lambda_{n}\right)+\delta\sin\left(\frac{b}{c}\mu_{m}^{\prime}\lambda_{n}\right)\right\} \,\left(\frac{\mu_{m}^{\prime}}{\lambda_{n}}\right)^{s-\frac{1}{2}}\,K_{s-\frac{1}{2}}\left(2k^{\prime}\mu_{m}^{\prime}\lambda_{n}\right).\label{general H2 with other deltinha}
\end{equation*}

To prove a reflection formula connecting both entire functions, substitute
on the argument of $H_{1}^{\delta}(s,\,Q;\,\cdot)$ $s$ by $1-s$
and $Q$ by $Q^{-1}$. Since $Q^{-1}(x,y)=cx^{2}-bxy+ay^{2}$ and $K_{\nu}(z)=K_{-\nu}(z)$, one obtains without effort
\begin{equation}
H_{1}^{\delta}\left(1-s;\,Q^{-1};\,a_{1},\,b_{2};\,\lambda,\,\mu^{\prime}\right)=(-1)^{\delta}\,H_{2}^{\delta}\left(s;\,Q;\,b_{2},\,a_{1},\,\mu^{\prime},\,\lambda\right), \label{reflection formula with delltinha delltinha}
\end{equation}
which is employed to prove our next result. 

\begin{corollary}[The Analytic Continuation of the non-diagonal Epstein zeta function] \label{corollary 3.1}
Under the hypotheses given on the statement of Theorem \ref{selberg-chowla non diagonal theorem}, the non-diagonal Epstein zeta function (\ref{Epstein zeta function non diagonal-1}) can be continued to the entire complex plane as: 
\begin{enumerate}
\item A meromorphic function with a simple pole located at $s=1$ with residue
$\frac{2\pi\,\rho_{1}\rho_{2}}{\sqrt{|d|}}$ if $a_{1}(n)$
and $a_{2}(n)$ are even arithmetical functions and $\rho_{1},\,\rho_{2}$ respectively denote the residues of $\phi_{1},\,\phi_{2}$ at $s=1$. In particular,
if at least one of the Dirichlet series $\phi_{i}$ is entire, then
$\mathcal{Z}_{2}$ is entire. 
\item An entire function if $a_{1}(n),\,\,a_{2}(n)$ are
odd arithmetical functions.
\end{enumerate}

Thus, $\mathcal{Z}_{2}$ is a Dirichlet series belonging to the class
$\mathcal{A}$ and satisfying Hecke's functional equation 
\begin{equation}
\left(\frac{2}{\sqrt{|d|}}\right)^{-s}\,\Gamma(s)\,\mathcal{Z}_{2}\left(s;\,Q;\,a_{1},\,a_{2};\,\lambda,\,\lambda^{\prime}\right)=(-1)^{\delta}\,\left(\frac{2}{\sqrt{|d|}}\right)^{s-1}\,\Gamma(1-s)\,\mathcal{Z}_{2}\left(1-s;\,Q^{-1};\,b_{1},\,b_{2};\,\mu,\,\mu^{\prime}\right).\label{Functional Equation for Non diagonal YEYEYEYEY GENERAL SANDY}
\end{equation}
\end{corollary}

\begin{proof}
The proof of the continuation is similar to the proof of Corollary
\ref{The Analytic Continuation}: since $H_{1}^{\delta}(s,\,Q;\,\cdot)$ and $H_{2}^{\delta}(s,\,Q;\,\,\cdot)$ are entire functions
of the complex variable $s$, then it follows from the Selberg-Chowla
formula (\ref{First Selberg Chowla odd}) that $\mathcal{Z}_{2}$
is entire if $a_{1}(n)$ and $a_{2}(n)$ are odd arithmetical
functions. Furthermore, the functional equation (\ref{Functional Equation for Non diagonal YEYEYEYEY GENERAL SANDY})
with $\delta=1$ is an immediate consequence of (\ref{reflection formula with delltinha delltinha}). 

\medskip{}

It remains to show the corollary for the case where $a_{1}(n)$
and $a_{2}(n)$ are even arithmetical functions: of course,
the first terms on (\ref{First Selberg Chowla even non diagonal})
are the only ones contributing to the singularities of $\mathcal{Z}_{2}$
so that we just need to study the meromorphic part
\[
G\left(s,\,Q\right)=-4a^{-s}\mathfrak{\mathfrak{\phi}}_{2}(0)\,\phi_{1}(2s)+2\sqrt{\pi}\,\rho_{1}\,\frac{\Gamma\left(s-\frac{1}{2}\right)}{\Gamma(s)}a^{-s}k^{1-2s}\phi_{2}(2s-1).
\]

It is clear that $G(s,\,Q)$ has removable singularities at the points
$s=\frac{1}{2}-k$, $k\in\mathbb{N}$. This easily comes from the
functional equation and the fact that $\phi_{1}\in\mathcal{B}$. Now,
since $\phi_{1}(2s)$ and $\Gamma\left(s-\frac{1}{2}\right)$ are
analytic in a neighbourhood of $s=1$, we see that $\mathcal{Z}_{2}$
must have a pole at $s=1$ coming from the contribution of $\phi_{2}(2s-1)$.
The residue is very easy to compute and it is given by $\frac{2\pi\,\rho_{1}\rho_{2}}{\sqrt{|d|}}.$ 

As in Corollary \ref{The Analytic Continuation}, the proof of the functional equation is made
by comparing both representations of $\mathcal{Z}_{2}$, (\ref{First Selberg Chowla even non diagonal})
and (\ref{Second Selberg Chowla even non-diagonal}): start by multiplying (\ref{First Selberg Chowla even non diagonal}) by $a^{-s}$ and then 
substitute $s$ by $1-s$, $Q$ by $Q^{-1}$ and $a_{1}(n),\,a_{2}(n)$
by $b_{1}(n),\,b_{2}(n)$. We obtain
\begin{align}
\Gamma(1-s)\,\mathcal{Z}_{2}\left(1-s;\,Q^{-1};\,b_{1},\,b_{2};\,\mu,\,\mu^{\prime}\right) & =-4c^{s-1}\mathfrak{\mathfrak{\psi}}_{2}(0)\,\Gamma(1-s)\psi_{1}(2-2s)+\nonumber \\
+2\sqrt{\pi}\,\rho_{1}\,\Gamma\left(\frac{1}{2}-s\right)c^{s-1}\,k^{\prime2s-1}\,\psi_{2}(1-2s)+4\,\left(k^{\prime}c\right)^{s-\frac{1}{2}}\,&\sqrt{\frac{1}{c}} H_{1}^{0}(1-s;\,Q^{-1};\,a_{1},\,b_{2};\,\lambda,\,\mu^{\prime}),\label{First Selberg Chowla even non diagonal-1}
\end{align}
which, when combined with the reflection formula (\ref{reflection formula with delltinha delltinha})
and the functional equation for $\psi_{1}$ and $\phi_{2}$, yields 
\begin{align}
\Gamma(1-s)\,\mathcal{Z}_{2}\left(1-s;\,Q^{-1};\,b_{1},\,b_{2};\,\mu,\,\mu^{\prime}\right) & =-4c^{s-1}\mathfrak{\mathfrak{\psi}}_{2}(0)\,\Gamma\left(s-\frac{1}{2}\right)\,\phi_{1}\left(2s-1\right)+\nonumber \\
+2\sqrt{\pi}\,\rho_{1}\,c^{s-1}\,k^{\prime2s-1}\,\Gamma(s)\,\phi_{2}(2s)+4\,\left(k^{\prime}c\right)^{s-\frac{1}{2}}\, &\sqrt{\frac{1}{c}} H_{2}^{0}(s;\,Q;\,b_{2},\,a_{1},\,\mu^{\prime},\,\lambda).\label{We are the lucky ones we are we are}
\end{align}

Invoking now the simple properties of the class $\mathcal{B}$, $\phi_{i}(0)=-\frac{\rho_{i}^{\star}}{2}\sqrt{\pi}$
and $\psi_{i}(0)=-\frac{\rho_{i}}{2}\sqrt{\pi}$, and making these
substitutions on the first two terms of (\ref{We are the lucky ones we are we are}),
we arrive precisely at the right-hand side of the second Selberg-Chowla
formula (\ref{Second Selberg Chowla even non-diagonal}) multiplied
by an extra factor of $(2/\sqrt{|d|})^{1-2s}\,c^{-s}$. This completes
the proof.
\end{proof}


\begin{remark} \label{remark odd odd Epstein} 
If $a(n)$ is an even arithmetical function
and $\phi\in\mathcal{B}$, then $\phi(2s)\in\mathcal{A}$ with $r=1$.
Analogously, if $a(n)$ is an odd arithmetical function
and $\phi\in\mathcal{B}$ then $\phi(2s-1)\in\mathcal{A}$ with $r=\frac{3}{2}$.
In fact, by letting $Q(x,\,y)=x^{2}+y^{2}$,
we have that (\ref{First Selberg Chowla even non diagonal}) is a
reformulation of the first Selberg-Chowla formula restricted to the
class $\mathcal{A}$, (\ref{Again First-1}), for $r_{1}=r_{2}=\frac{1}{2}$.
This is why we have named the Epstein zeta function (\ref{Epstein as double series})
as ``diagonal''. Despite this, it is clear that the Selberg-Chowla 
formulas for odd arithmetical functions and some particular quadratic
form $Q(x,\,y)$ cannot be seen as reformulations or as particular cases of (\ref{Again First-1})
and (\ref{Again Second-1}). This happens because the arithmetical 
function associated to the Dirichlet series $\varphi(s):=\phi(2s-1)$
is $\mathfrak{c}(n)=a(n)\,\lambda_{n}$ and not $a(n)$.
Note that, by item 2. in Definition \ref{definition 3.1}, $\mathfrak{c}(n)$ is an even arithmetical function and it is possible to establish a Selberg-Chowla formula for its non-diagonal Epstein zeta function. Although we cannot apply Theorem \ref{selberg-chowla non diagonal theorem} in a direct
way (because $a(n)\,\lambda_{n}$ is not attached to a
Bochner Dirichlet series), it is not difficult to derive an analogue
of Lemma \ref{lemma 3.1} valid for this case. Let  $\mathcal{Z}_{2}(s;\,Q;\,\mathfrak{c}_{1},\,\mathfrak{c}_{2};\,\lambda,\,\lambda^{\prime})$
denote the following version of the non-diagonal Epstein zeta function
\begin{equation}
\mathcal{Z}_{2}(s;\,Q;\,\mathfrak{c}_{1},\,\mathfrak{c}_{2};\,\lambda,\,\lambda^{\prime})=\sum_{m,n\neq0}\frac{a_{1}(m)\lambda_{m}\,a_{2}(n)\lambda_{n}^{\prime}}{Q(\lambda_{m},\,\lambda_{n}^{\prime})^{s}},\,\,\,\,\,\,\text{Re}(s)>\sigma_{a}+1.\label{New Epstein with evenizing of things Berndt style}
\end{equation}

Then, for any $\text{Re}(s)>\mu>\max\left\{ \sigma_{a}+1,\,\sigma_{b}+1\right\} $,
the following Selberg-Chowla formula holds
\begin{align}
a^{s}\Gamma(s)\,\mathcal{Z}_{2}(s;\,Q;\,\mathfrak{c}_{1},\,\mathfrak{c}_{2};\,\lambda,\,\lambda^{\prime}) & =\,4k^{\frac{3}{2}-s}\,\sum_{j=1}^{\infty}\sigma_{2s-3}\left(\nu_{j};\,b_{1},\,a_{2};\,\mu,\,\lambda^{\prime}\right)\,\nu_{j}^{\frac{5}{2}-s}\,\nonumber \\
\times\left\{ \cos\left(\frac{b}{a}\,\nu_{j}\right)\right. & \left.K_{\frac{3}{2}-s}(2k\,\nu_{j})+\frac{2b}{\sqrt{|d|}}\,\sin\left(\frac{b}{a}\,\nu_{j}\right)\,K_{\frac{1}{2}-s}(2k\,\nu_{j})\right\},\label{First Selberg Chowla Berndt Case for analogues Redirus}
\end{align}
which gives the analytic continuation of the Dirichlet series (\ref{New Epstein with evenizing of things Berndt style})
as an entire complex function.  In particular, if $b=0$ and $a=c=1$,
(\ref{First Selberg Chowla Berndt Case for analogues Redirus}) reduces to (\ref{Again First-1})
with $r_{1}=r_{2}=\frac{3}{2}$. 
\end{remark}

\begin{center}\section{The zeros of a class of Dirichlet series} \label{section 4} \end{center} 

It is immediate to check from the Selberg-Chowla (or the functional
equation) that $\mathcal{Z}_{2}(s;\,a_{1},\,a_{2};\,\lambda,\,\lambda^{\prime})$
has trivial zeros located at the negative integers, this is, $s=-n$,
$n\in\mathbb{N}$. A similar comment can be made for $\mathcal{Z}_{2}(s;\,Q;\,a_{1},\,a_{2};\,\lambda,\,\lambda^{\prime})$.
Moreover, in a completely analogous way to the results given in {[}\cite{Bateman_Epstein},
p. 367, Theorem 3{]} and \cite{terras_ternary}, one can exhibit subclasses of the Epstein zeta
functions studied in the previous sections which do not obey to a generalized Riemann hypothesis,
i.e., which fail the condition of having all their nontrivial zeros
located in the critical line $\text{Re}(s)=r$. The next corollary
of the Selberg-Chowla formula gives this result. In order to state
it, we introduce the following notation: for a given $\xi>0$, we write $\mathcal{Z}_{2}(s;\,a_{1},\,a_{2};\,\lambda,\,\xi\lambda^{\prime})$
as the analogue of the diagonal Epstein zeta function given by the
double Dirichlet series 
\begin{equation}
\mathcal{Z}_{2}(s;\,a_{1},\,a_{2};\,\lambda,\,\xi\lambda^{\prime})=\sum_{m,n\neq0}^{\infty}\frac{a_{1}(m)\,a_{2}(n)}{\left(\lambda_{m}+\lambda_{n}^{\prime}\xi\right)^{s}},\,\,\,\,\,\,\text{Re}(s)>2\,\sigma_{a}.\label{Diagonal Epstein with parameter xssssi}
\end{equation}

\bigskip{}

\bigskip{}

\begin{corollary}[Failure of the Riemann Hypothesis for $\mathcal{Z}_{2}$]\label{corollary failure RH}
Suppose that $\phi_{1}$ and $\phi_{2}$ are Dirichlet series with real arithmetical functions $a_{1}(n)$ and $a_{2}(n)$,
satisfying Hecke's functional equation and belonging to the class
$\mathcal{A}$. Assume also that $\phi_{i}(0)<0$ for any $i=1,\,2$.
Then, for a sufficiently large $\xi>0$, $\mathcal{Z}_{2}(s;\,a_{1},\,a_{2};\,\lambda,\,\xi\lambda^{\prime})$
has a real zero on the interval $\left(r_{1},\,r_{1}+r_{2}\right)$,
for $r_{2}\leq r_{1}$. 

\medskip{}

Moreover, if $\phi_{i}$, $i=1,\,2$, are also real Bochner Dirichlet series
belonging to the class $\mathcal{B}$ and such that $\phi_{i}(0)<0$
then, for a quadratic form $Q$ having a sufficiently large $k:=\sqrt{|d|}/2a$,
$\mathcal{Z}_{2}\left(s;\,Q;\,a_{1},\,a_{2};\,\lambda,\,\lambda^{\prime}\right)$
has a real zero on the interval $\left(\frac{1}{2},\,1\right)$. 
\end{corollary}

\begin{proof}
We shall prove the result only for the diagonal Epstein zeta function
$\mathcal{Z}_{2}(s;\,a_{1},\,a_{2};\,\lambda,\,\xi\lambda^{\prime})$,
since similar considerations hold for $\mathcal{Z}_{2}(s;\,Q;\,a_{1},\,a_{2};\,\lambda,\,\lambda^{\prime})$.
The idea of the proof is to write a Selberg-Chowla formula for (\ref{Diagonal Epstein with parameter xssssi}).
To do this, let us replace on (\ref{Again First-1}) the sequence
$\lambda_{n}^{\prime}$ by $\xi\lambda_{n}^{\prime}$, where $\xi>0$.
By virtue of Hecke's functional equation, we also need to replace
$\mu_{n}^{\prime}$ by $\mu_{n}^{\prime}/\xi$, $b_{2}(n)$
by $\xi^{-r_{2}}b_{2}(n)$ and $\rho_{2}$ by $\xi^{-r_{2}}\rho_{2}$.
From (\ref{Again First-1}), the following Selberg-Chowla formula holds 
\begin{align}
\Gamma(s)\,\sum_{m,n\neq0}^{\infty}\frac{a_{1}(m)\,a_{2}(n)}{\left(\lambda_{m}+\lambda_{n}^{\prime}\xi\right)^{s}} & =-\phi_{2}(0)\,\Gamma(s)\,\phi_{1}(s)+\rho_{1}\Gamma(r_{1})\,\Gamma(s-r_{1})\,\xi^{r_{1}-s}\,\phi_{2}\left(s-r_{1}\right)\nonumber \\
+2\,\xi^{\frac{r_{1}-s}{2}} & \,\sum_{m,n=1}^{\infty}\sigma_{s-r_{1}}\left(\nu_{j};\,b_{1},\,a_{2}\right)\,\nu_{j}^{\frac{r_{1}-s}{2}}K_{r_{1}-s}\left(2\sqrt{\mu_{m}\lambda_{n}^{\prime}\,\xi}\right),\label{Diagonal Parameter 1st}
\end{align}
which provides the analytic continuation of (\ref{Diagonal Epstein with parameter xssssi}).
Without loss of generality, assume that $r_{2}\leq r_{1}$: from Corollary \ref{The Analytic Continuation} we know that the
only pole that $\mathcal{Z}_{2}$ possesses is located at $s=r_{1}+r_{2}$
with residue given by 
\[
\text{Res}_{s=r_{1}+r_{2}}\,\mathcal{Z}_{2}(s;\,a_{1},\,a_{2};\,\lambda,\,\xi\lambda^{\prime})=\frac{\Gamma(r_{1})\Gamma(r_{2})}{\Gamma(r_{1}+r_{2})}\,\xi^{-r_{2}}\,\rho_{1}\rho_{2}.
\]

Thus, since $\rho_{i}>0$ (because $\phi_{i}(0)<0$ by hypothesis)
and $\xi>0$, the latter is a positive real number, so that $\lim_{s\rightarrow r_{1}+r_{2}^{-}}\mathcal{Z}_{2}\left(s;\,a;\,\lambda;\,\xi\right)=-\infty$.
We may now compute $\mathcal{Z}_{2}$ at $s=r_{1}$: to do this, let
us write the Laurent series for $\phi_{1}(s)$ around $s=r_{1}$ as
\begin{equation}
\phi_{1}(s)=\frac{\rho_{1}}{s-r_{1}}+\rho_{0,1}+O\left(s-r_{1}\right),\label{meromorphic phii 1}
\end{equation}
and use the meromorphic expansions for $\Gamma(s)$ and $\phi_{1}(s)$
around $s=0$,
\begin{equation}
\Gamma(s)=\frac{1}{s}-\gamma+O(s),\label{Expansion Gamma 0}
\end{equation}
\begin{equation}
\phi_{1}(s)=\phi_{1}(0)+\phi_{1}^{\prime}(0)\,s+O(s^{2}),\label{Taylor phi1}
\end{equation}
where $\gamma$ is Euler's constant. Due to the uniform convergence of the infinite series in (\ref{Diagonal Parameter 1st}) (by Proposition \ref{proposition 2.1}),
taking the limit $s\rightarrow r_{1}$ on the right-hand side of (\ref{Diagonal Parameter 1st})
gives 
\begin{equation}
-\phi_{2}(0)\left\{ \rho_{0,1}+\rho_{1}\log(\xi)+\rho_{1}\gamma+\rho_{1}\,\frac{\Gamma^{\prime}(r_{1})}{\Gamma(r_{1})}\right\} +\rho_{1}\phi_{2}^{\prime}(0)+\frac{2}{\Gamma\left(r_{1}\right)}\,\sum_{j=1}^{\infty}d\left(\nu_{j};\,b_{1},\,a_{1};\,\mu,\,\lambda^{\prime}\right)\,K_{0}\left(2\xi\sqrt{\nu_{j}}\right),\label{Value at r_1}
\end{equation}
where the arithmetical function $d$ is the generalization of the usual
divisor function $d(n)$ provided by formula (\ref{generalized Divisor function}).

Now, we study the infinite series at the right-hand side of (\ref{Value at r_1}):
from the inequality {[}\cite{Bateman_Epstein}, p. 368, Lemma 3{]}
\begin{equation}
K_{0}(x)\leq\left(\frac{\pi}{2x}\right)^{1/2}e^{-x}\left[1-\frac{1}{8x}+\frac{9}{128x^{2}}\right],\label{IV - 82-1-1}
\end{equation}
we can deduce that the infinite series in (\ref{Value at r_1}) is dominated by 
\begin{equation}
\left(\frac{\pi}{4\xi}\right)^{1/2}\sum_{m,n=1}^{\infty}b(m)\,a(n)\,\left(\mu_{m}\lambda_{n}\right)^{-\frac{1}{4}}\,e^{-2\xi\sqrt{\mu_{m}\lambda_{n}}}\left[1-\frac{1}{16\xi\sqrt{\mu_{m}\lambda_{n}}}+\frac{9}{512\xi^{2}\mu_{m}\lambda_{n}}\right],\label{thing that goes goes gores away to zeeeero}
\end{equation}
which clearly goes to zero when $\xi\rightarrow\infty$. Hence, it
suffices to choose $\xi$ for which the residual term in (\ref{Value at r_1})
is positive. But this condition is satisfied whenever
\begin{equation}
\xi>\exp\left\{ \frac{\phi_{2}^{\prime}(0)}{\phi_{2}(0)}-\frac{\rho_{0,1}}{\rho_{1}}-\gamma-\frac{\Gamma^{\prime}(r_{1})}{\Gamma(r_{1})}\right\} .\label{condition no zero no name}
\end{equation}
Since there exists some $\xi_{0}$ such that, for any $\xi>\xi_{0}$,
(\ref{thing that goes goes gores away to zeeeero}) is less than the
right-hand side of (\ref{condition no zero no name}), if we take
$\xi>\xi_{0}$ we have that $\mathcal{Z}_{2}\left(r_{1};\,a_{1},\,a_{2};\,\lambda,\,\xi\lambda^{\prime}\right)>0$
and so there must be some real zero of the Epstein zeta function in
the interval $(r_{1},\,r_{1}+r_{2})$.
\end{proof}

\bigskip{}

Despite the fact that, in general, most particular cases of the Epstein zeta functions
defined in the previous sections may have infinitely many zeros off their
critical lines, as proved by Davenport, Heilbronn and Voronin for the classical Epstein zeta function \cite{davenport heilbronn, voronin}, a large class of them will certainly admit infinitely
many zeros on the critical line, even under our general setting. For the classical Epstein zeta function
$Z_{2}(s,\,Q)$, this fact was proved for the first time by Potter
and Titchmarsh \cite{Titchmarsh_Potter}. A simpler proof, similar to
Hardy's proof of the same theorem for $\zeta(s)$, was given by Kober \cite{kober_zeros}.

Since we have generalized the Selberg-Chowla formula and the continuation of Epstein zeta functions in the previous sections, we now use this advantage
of representing the double Dirichlet series $\mathcal{Z}_{2}(s;\,a_{1},\,a_{2};\,\lambda,\,\lambda^{\prime})$
in terms of the single series $\phi_{1}(s)$ and $\phi_{2}(s)$ in
order to relate the distribution of zeros of $\phi_{1}(s)$ and $\phi_{2}(s)$
with the asymptotic order of $\mathcal{Z}_{2}$ in its critical line $\text{Re}(s)=r$.


We also prove that if $\mathcal{Z}_{2}$ satisfies Hardy's Theorem, then $\phi(s)$ must satisfy Hardy's Theorem as well. 





The bridge connecting the truth of Hardy's theorem for $\phi(s)$ and for $\mathcal{Z}_{2}$ lies precisely in Selberg-Chowla's formula (\ref{Selberg Chowla for diagonal 2k}). However, to derive in a precise form this connection, we still need to write it in a suitable form, i.e., involving an integral representation of $\phi$ on the critical line. In the next lemma we rewrite (\ref{Selberg Chowla for diagonal 2k}) in an integral form. There, we shall consider the subsequence of multidimensional diagonal Epstein zeta functions, $\mathcal{Z}_{2^{k}}\left(s;\,a;\,\lambda\right)$,
defined in Remark \ref{dyadic epstein as remark}. 
\bigskip{}

\begin{lemma} \label{lemma integral representation}
Let $\mathcal{Z}_{2^{k}}(s;\,a;\,\lambda)$ be the dyadic
Epstein zeta function defined by (\ref{dyadic Epstein 2^k}). Then the following
representation of $\mathcal{Z}_{2^{k+1}}(s;\,a;\,\lambda)$
at the critical line, $\text{Re}(s)=2^{k}r$, holds
\begin{align}
\Gamma(2^{k}r+2it)\,\mathcal{Z}_{2^{k+1}}\left(2^{k}r+2it;\,a;\,\lambda\right) & =-2\mathcal{Z}_{2^{k}}\left(0;\,a;\,\lambda\right)\,\Gamma(2^{k}r+2it)\,\mathcal{Z}_{2^{k}}\left(2^{k}r+2it;\,a;\,\lambda\right)+\nonumber \\
+2\,\Gamma^{2^{k}}(r)\,\rho^{2^{k}}\,\Gamma(2it)\,\mathcal{Z}_{2^{k}}\left(2it;\,a;\,\lambda\right) & +\frac{1}{2\pi}\,\int_{-\infty}^{\infty}\Gamma\left(2^{k-1}r+i(y-t)\right)\,\mathcal{Z}_{2^{k}}\left(2^{k-1}r+i(y-t);\,b;\,\lambda\right)\times\label{Essssential to prove Deuring}\\
 & \,\,\,\,\,\,\,\,\,\,\,\,\,\,\,\,\,\,\,\,\,\,\,\,\,\,\,\,\,\,\times\Gamma\left(2^{k-1}r+i(y+t)\right)\,\mathcal{Z}_{2^{k}}\left(2^{k-1}r+i(y+t);\,a;\,\lambda\right)dy. \nonumber 
\end{align}
\end{lemma}
\begin{proof}
Start with the Selberg-Chowla formula (\ref{Selberg Chowla for diagonal 2k})
and write $s=2^{k}r+2it$, $t\in\mathbb{R}$. We obtain the representation 
\begin{align}
\Gamma(2^{k}r+2it)\,\mathcal{Z}_{2^{k+1}}\left(2^{k}r+2it;\,a;\,\lambda\right) & =-\mathcal{Z}_{2^{k}}\left(0;\,a;\,\lambda\right)\,\Gamma(2^{k}r+2it)\,\mathcal{Z}_{2^{k}}\left(2^{k}r+2it;\,a;\,\lambda\right)+\nonumber \\
+\Gamma^{2^{k}}(r)\,\rho^{2^{k}}\,\Gamma(2it)\,\mathcal{Z}_{2^{k}}\left(2it;\,a;\,\lambda\right) & +2\,\sum_{m,n=1}^{\infty}\mathfrak{V}_{2^{k}}(m)\,\mathfrak{U}_{2^{k}}(n)\,\left(\frac{\Omega_{m}}{\Lambda_{n}}\right)^{it}\,K_{2it}\left(2\sqrt{\Omega_{m}\Lambda_{n}}\right),\label{formula with it representation}
\end{align}
where the definitions of the arithmetical functions $\mathfrak{U}_{2^{k}}(n)$
and $\mathfrak{V}_{2^{k}}(m)$ are given on the second section (see Remark \ref{dyadic epstein as remark}). To
proceed further, we need to give an integral representation for the
double series appearing on the right-hand side of (\ref{formula with it representation}).
Our idea is then to return to the integral representation for the modified Bessel function (\ref{MacDonald Representation}) and
rewrite it in a more symmetric form
\begin{equation}
K_{\nu}(x)=\frac{1}{2\pi i}\,\int_{\mu-i\infty}^{\mu+i\infty}2^{w-2}\Gamma\left(\frac{w+\nu}{2}\right)\Gamma\left(\frac{w-\nu}{2}\right)\,x^{-w}\,dw,\,\,\,\,\,\mu>\text{Re}(\nu).\label{On the modified Bessel function}
\end{equation}
 Replacing $\nu$ by $2it$ in (\ref{On the modified Bessel function})
and letting $\mu>\max\left\{ 0,\,2^{k}\sigma_{a},\,2^{k}\sigma_{b},\,2^{k+1}r\right\} $,
we can write $2\,H_{2^{k}r}(2^{k}r+2it;\,a;\,\lambda)$
as 
\begin{align}
\sum_{m,n=1}^{\infty}\mathfrak{V}_{2^{k}}(m)\,\mathfrak{U}_{2^{k}}(n)\,\left(\frac{\Omega_{m}}{\Lambda_{n}}\right)^{it}\frac{1}{\pi i}\int_{\mu-i\infty}^{\mu+i\infty}2^{w-2}\Gamma\left(\frac{w+2it}{2}\right)\Gamma\left(\frac{w-2it}{2}\right)\,\left(2\sqrt{\Omega_{m}\Lambda_{n}}\right)^{-w}\,dw\nonumber \\
=\frac{1}{4\pi i}\,\int_{\mu-i\infty}^{\mu+i\infty}\Gamma\left(\frac{w-2it}{2}\right)\,\mathcal{Z}_{2^{k}}\left(\frac{w-2it}{2};\,b;\,\lambda\right)\,\Gamma\left(\frac{w+2it}{2}\right)\,\mathcal{Z}_{2^{k}}\left(\frac{w+2it}{2};\,a;\,\lambda\right)\,dw,\label{First Integral Representation}
\end{align}
since this choice of $\mu$ allowed to interchange the double series
and the Mellin integral (see the proof of Theorem \ref{theorem 2.1}). Let us now
shift the line of integration in (\ref{First Integral Representation})
to $\text{Re}(w)=2^{k}r$ which, by the choice of $\mu$, is located
on the left of the line $\text{Re}(w)=\mu$. Doing this requires once
more to integrate along a positively oriented rectangular contour
$\mathcal{R}$ with vertices $\mu\pm iT$ and $2^{k}r\pm iT$, $T>0$.
The application of the Phragm\'en-Lindel\"of principle for $\mathcal{Z}_{2^{k}}$
(\ref{Estimate general}) and Stirling's formula allows to conclude
that the integrals on the horizontal segments vanish when $T\rightarrow\infty$.
By shifting the line of integration, we note that the integrand has
two simple poles on the interior of $\mathcal{R}$ in the form $w=2^{k+1}r\pm2it$
and their residues are, respectively, 
\begin{align}
R_{2^{k+1}r+2it} & =2\,\Gamma^{2^{k}}(r)\,\rho^{\star2^{k}}\Gamma(2^{k}r+2it)\,\mathcal{Z}_{2^{k}}\left(2^{k}r+2it;\,a;\,\lambda\right)\nonumber \\
 & =-2\,\mathcal{Z}_{2^{k}}\left(0;\,a;\,\lambda\right)\,\Gamma(2^{k}r+2it)\,\mathcal{Z}_{2^{k}}\left(2^{k}r+2it;\,a;\,\lambda\right),\label{Residue 1}
\end{align}
where we have used the fact that $\mathcal{Z}_{2^{k}}(s;\,a;\,\lambda)\in\mathcal{A}$ (see Corollary \ref{analytic continuation multi}), and 
\begin{equation}
R_{2^{k+1}r-2it}=2\,\Gamma^{2^{k}}(r)\,\rho^{2^{k}}\,\Gamma(2it)\,\mathcal{Z}_{2^{k}}\left(2it;\,a;\,\lambda\right).\label{Residue 2}
\end{equation}

Using the residue residue theorem and adding the residues (\ref{Residue 1})
and (\ref{Residue 2}) to (\ref{formula with it representation}),
we obtain the formula 
\begin{align*}
\Gamma(2^{k}r+2it)\,\mathcal{Z}_{2^{k+1}}\left(2^{k}r+2it;\,a;\,\lambda\right) & =-2\mathcal{Z}_{2^{k}}\left(0;\,a;\,\lambda\right)\,\Gamma(2^{k}r+2it)\,\mathcal{Z}_{2^{k}}\left(2^{k}r+2it;\,a;\,\lambda\right)+\\
+2\,\Gamma^{2^{k}}(r)\,\rho^{2^{k}}\,\Gamma(2it)\,\mathcal{Z}_{2^{k}}\left(2it;\,a;\,\lambda\right) & +\frac{1}{4\pi i}\,\int_{2^{k}r-i\infty}^{2^{k}r+i\infty}\Gamma\left(\frac{w-2it}{2}\right)\,\mathcal{Z}_{2^{k}}\left(\frac{w-2it}{2};\,b;\,\lambda\right)\,\times\\
& \,\,\,\,\,\,\,\,\,\,\,\,\,\,\,\,\,\,\,\,\,\,\,\,\,\,\,\,\,\,\,\,\,\,\,\,\,\,\,\,\,\times \Gamma\left(\frac{w+2it}{2}\right)\,\mathcal{Z}_{2^{k}}\left(\frac{w+2it}{2};\,a;\,\lambda\right)\,dw,
\end{align*}
which yields (\ref{Essssential to prove Deuring}) after we write
$w=2^{k}r+2iy$ on the contour integral.
\end{proof}
\bigskip{}

We now establish our Main Theorem: 
\begin{theorem} \label{deuring to hold}
Let $\phi(s)$ be a Dirichlet series satisfying Hecke's functional
equation in the sense of Definition \ref{definition 1.1.} and belonging to the class
$\mathcal{A}$. Suppose that the parameters of $\phi$ and $\psi$
satisfy $\overline{a}(n)=b(n)$ and $\lambda_{n}=\mu_{n}$. Also, let $\sigma_{a}$
denote the abcissa of absolute convergence of $\phi$. If one of
the following conditions holds: 
\begin{enumerate}
\item The abcissa of absolute convergence satisfies $\sigma_{a}<\min\{1,r\}+\frac{1}{2}\,\max\{1,r\}$
and, for some $k\geq0$, the diagonal Epstein zeta function $\mathcal{Z}_{2^{k+1}}$,
(\ref{dyadic Epstein 2^k}), associated with $\phi$ has infinitely
many zeros in its critical line $\text{Re}(s)=2^{k}r$. 
\item The abcissa of absolute convergence satisfies $\sigma_{a}<\min\{1,r\}+\frac{1}{2}\max\{1,r\}$ and,
for some $k\geq0$, the diagonal Epstein zeta function $\mathcal{Z}_{2^{k+1}}$
obeys to a subconvex estimate on the critical line of the form
\begin{equation}
\mathcal{Z}_{2^{k+1}}(2^{k}r+it;\,a;\,\lambda)=o\left(|t|^{2^{k}-\frac{1}{2}}\right),\,\,\,\,\,|t|\rightarrow\infty.\label{condition critical line for r=00005Cleq1}
\end{equation}
\end{enumerate}
Then $\phi(s)$ has infinitely many zeros on the critical line $\text{Re}(s)=\frac{r}{2}$. 
\end{theorem}

\begin{proof}
Since $\sigma_{a}=\sigma_{b}$ (because $\overline{a}(n)=\,b(n)$
and $\lambda_{n}=\mu_{n}$) and $\phi_{1}=\phi_{2}=\phi\in\mathcal{A}$,
the Selberg-Chowla representation for the diagonal Epstein zeta function
$\mathcal{Z}_{2}(s;\,a_{1},\,a_{2};\,\lambda,\,\lambda^{\prime})$
reduces to (\ref{formula with it representation}) with $k=0$. For the case
$k=0$, (\ref{Essssential to prove Deuring}) gives 
\begin{align}
\Gamma(r+2it)\,\mathcal{Z}_{2}\left(r+2it;\,a;\,\lambda\right) & =-2\phi(0)\,\Gamma(r+2it)\,\phi\left(r+2it\right)+2\,\Gamma(r)\,\rho\,\Gamma(2it)\,\phi\left(2it\right)\nonumber \\
+\frac{1}{2\pi}\,\int_{-\infty}^{\infty} & \Gamma\left(\frac{r}{2}+i(y-t)\right)\psi\left(\frac{r}{2}+i(y-t)\right)\,\Gamma\left(\frac{r}{2}+i(y+t)\right)\,\phi\left(\frac{r}{2}+i(y+t)\right)\,dy.\label{Selberg Chowla in an integral form of writing}
\end{align}

The idea of this proof is to estimate both sides
of the previous equality when $|t|\rightarrow\infty$. We shall give
an estimate for the right-hand side of (\ref{Selberg Chowla in an integral form of writing})
under the assumption that $\phi(s)$ does not possess infinitely many
zeros at the line $\text{Re}(s)=\frac{r}{2}$ which will contradict the assumptions of our Theorem. 

By the hypothesis $b(n)=\overline{a}(n)$,
we know that the integrand in the previous equation must be a real function of $y\in\mathbb{R}$. For convenience, we shall write
$\Gamma(s)\,\phi(s)$ as $R_{\phi}(s)$ and $\Gamma(s)\,\psi(s)$
as $R_{\psi}(s)$ (see subsection \ref{notation and def to refer}). Supposing that $\phi\left(s\right)$ has finitely
many zeros on the line $\text{Re}(s)=\frac{r}{2}$, there exists some
$T_{0}>0$ such that, if $|y-t|>T_{0}$ and $|y+t|>T_{0}$, the integrand
has a constant sign. Hence, the following equality holds 
\begin{equation}
\left|\int_{-t+T_{0}}^{t-T_{0}}R_{\psi}\left(\frac{r}{2}+i\left(y-t\right)\right)\,R_{\phi}\left(\frac{r}{2}+i\left(y+t\right)\right)\,dy\right|=\int_{-t+T_{0}}^{t-T_{0}}\left|R_{\psi}\left(\frac{r}{2}+i\left(y-t\right)\right)\,R_{\phi}\left(\frac{r}{2}+i\left(y+t\right)\right)\right|\,dy.\label{contradiction hypothesissisis}
\end{equation}

The idea now is to see that the right-hand side of (\ref{contradiction hypothesissisis})
will provide a lower bound for the integral on the right-hand side
of (\ref{Selberg Chowla in an integral form of writing}) which, in
its turn, will contradict hypothesis 1. and 2. of our statement. If we take a partition
of the integral in (\ref{Selberg Chowla in an integral form of writing})
as 
\begin{equation}
\left\{ \int_{t-T_{0}}^{\infty}+\int_{-\infty}^{-t+T_{0}}+\int_{-t+T_{0}}^{t-T_{0}}\right\} R_{\psi}\left(\frac{r}{2}+i\left(y-t\right)\right)\,R_{\phi}\left(\frac{r}{2}+i\left(y+t\right)\right)\,dy,\label{decomposition of H into three integrals}
\end{equation}
we see that, in the third of these, we can use the contradiction hypothesis
(\ref{contradiction hypothesissisis}). 

Let us denote each integral in the previous partition by $\mathcal{A}_{i}(t)$.
We first check that, for any $\delta>0$, the first integrals satisfy
the asymptotic order
\begin{equation}
\mathcal{A}_{1}(t),\,\mathcal{A}_{2}(t)=O\left(|t|^{\sigma_{a}-\frac{1}{2}+\delta}\,e^{-\pi|t|}\right).\label{estimate auxiliar integrals}
\end{equation}

The estimate (\ref{estimate auxiliar integrals}) comes directly from
Stirling's formula and a simple application of the Phragm\'en-Lindel\"of principle (\ref{Lindelof Phragmen for Hecke}) for $\phi$ and $\psi$ on the critical line $\text{Re}(s)=\frac{r}{2}$.
It suffices to verify (\ref{estimate auxiliar integrals}) only for
$\mathcal{A}_{1}(t)$, which in its turn concludes the same verification
for $\mathcal{A}_{2}(t)$ since, from the functional equation (\ref{Hecke Dirichlet series Functional}) for $\phi(s)$, $\mathcal{A}_{1}(t)=\mathcal{A}_{2}(t)$.
From the fact that $\phi$ and $\psi$ have the same abcissa of absolute
convergence, $\sigma_{a}$, $R_{\phi}$ and $R_{\psi}$
obey to the estimates
\begin{equation}
R_{\phi}\left(\frac{r}{2}+it\right),\,\,R_{\psi}\left(\frac{r}{2}+it\right)=O\left(|t|^{\sigma_{a}-\frac{1}{2}+\delta}e^{-\frac{\pi}{2}|t|}\right),\,\,\,\,|t|\rightarrow\infty,
\end{equation}
which can be used inside the integral for $\mathcal{A}_{1}(t)$, after
we split $\mathcal{A}_{1}$ as follows
\[
\int_{-T_{0}+t}^{\infty}R_{\psi}\left(\frac{r}{2}+i\left(y-t\right)\right)\,R_{\phi}\left(\frac{r}{2}+i\left(y+t\right)\right)\,dy=\left\{ \int_{-T_{0}}^{T_{0}}+\int_{T_{0}}^{\infty}\right\} R_{\psi}\left(\frac{r}{2}+iu\right)\,R_{\phi}\left(\frac{r}{2}+i\left(u+2t\right)\right)\,du.
\]
\medskip{}

Thus, by a simple application of Stirling's formula one sees immediately
that 
\begin{align}
\int_{-T_{0}}^{T_{0}}R_{\psi}\left(\frac{r}{2}+iu\right)\,R_{\phi}\left(\frac{r}{2}+i\left(u+2t\right)\right)\,du&=O\left(\int_{-T_{0}}^{T_{0}}\left|R_{\psi}\left(\frac{r}{2}+iu\right)\right|\,|u+2t|^{\sigma_{a}-\frac{1}{2}+\delta}e^{-\frac{\pi}{2}|u+2t|}\,du\right)\nonumber\\
&=O\left(|t|^{\sigma_{a}-\frac{1}{2}+\delta}\,e^{-\pi|t|}\right)\label{First Stirling}
\end{align}
while, analogously, 
\begin{align}
\int_{T_{0}}^{\infty}R_{\psi}\left(\frac{r}{2}+iu\right)\,R_{\phi}\left(\frac{r}{2}+i\left(u+2t\right)\right)\,du&=O\left(|t|^{2\sigma_{a}+2\delta}\,e^{-\pi|t|}\,\int_{T_{0}/2t}^{\infty}\,(u^{2}+u)^{\sigma_{a}-\frac{1}{2}+\delta}\,e^{-2\pi ut}\,du\right)\nonumber\\
&=O\left(|t|^{\sigma_{a}-\frac{1}{2}+\delta}\,e^{-\pi|t|}\right)\label{second Stirling}
\end{align}
providing (\ref{estimate auxiliar integrals}). 
\\

In the interval of integration considered on the third integral $\mathcal{A}_{3}(t)$,
we may invoke Stirling's formula once more, since we can take $T_{0}$
sufficiently large and, in this interval, we have $|y-t|>T_{0}$ and $|y+t|>T_{0}$.
From (\ref{Stirling exact form on Introduction}), we have that
\begin{equation}
\left|\Gamma\left(\frac{r}{2}+i\,(y-t)\right)\,\Gamma\left(\frac{r}{2}+i\,(y+t)\right)\right|=2\pi\,\left(t^{2}-y^{2}\right)^{\frac{r-1}{2}}e^{-\pi|t|}\left(1+O(|t|^{-1})\right).\label{exact stirling here}
\end{equation}

\medskip{}

Let $c$ be the least positive integer such that $a(c)\neq0$: invoking (\ref{exact stirling here})
on the right-hand side of (\ref{contradiction hypothesissisis}),
we obtain the inequality 
\begin{align}
\int_{-t+T_{0}}^{t-T_{0}}\left|R_{\psi}\left(\frac{r}{2}+i\left(y-t\right)\right)\,R_{\phi}\left(\frac{r}{2}+i\left(y+t\right)\right)\right|\,dy & =\lambda_{c}^{-r}\int_{-t+T_{0}}^{t-T_{0}}\left|\lambda_{c}^{r+2iy}R_{\psi}\left(\frac{r}{2}+i\left(y-t\right)\right)R_{\phi}\left(\frac{r}{2}+i\left(y+t\right)\right)\right|dy\nonumber \\
\geq2\pi\lambda_{c}^{-r}\,e^{-\pi|t|}\left|\int_{\frac{r}{2}-i(t-T_{0})}^{\frac{r}{2}+i(t-T_{0})}\,\lambda_{c}^{2z}\right. & \left.\left(t^{2}+\left(z-\frac{r}{2}\right)^{2}\right)^{\frac{r-1}{2}}\,\psi(z-it)\,\phi(z+it)dz\,\left\{ 1+O\left(|t|^{-1}\right)\right\} \right|.\label{inequality r=00005Cleq1}
\end{align}

By Cauchy's Theorem applied to the positively oriented rectangular
contour with vertices $\frac{r}{2}\pm i\,(t-T_{0})$ and $\sigma_{a}+1\pm i\,(t-T_{0})$,
we can write the contour integral on the right-hand side of (\ref{inequality r=00005Cleq1})
as 
\begin{equation}
\left\{ \int_{\frac{r}{2}-i(t-T_{0})}^{\sigma_{a}+1-i(t-T_{0})}+\int_{\sigma_{a}+1-i(t-T_{0})}^{\sigma_{a}+1+i(t-T_{0})}+\int_{\sigma_{a}+1+i(t-T_{0})}^{\frac{r}{2}+i(t-T_{0})}\right\} \,\lambda_{c}^{2z}\,\left(t^{2}+\left(z-\frac{r}{2}\right)^{2}\right)^{\frac{r-1}{2}}\,\psi(z-it)\,\phi(z+it)\,dz.\label{another decomposition of integrals}
\end{equation}

\medskip{}

\medskip{}

The first and third integrals in the above partition can be easily evaluated with the aid
of Phragm\'en-Lindel\"of estimates. Note that in the first of these, the
factor $\phi(z+it)$ does not depend on $t$ and since $T_{0}$ is
a fixed number then we just need to apply (\ref{Lindelof Phragmen for Hecke})
for $\psi(z-it)$. By symmetry, the same happens on the third integral
with the roles of $\phi$ and $\psi$ being reversed. Denoting each
integral on (\ref{another decomposition of integrals}) by $\mathcal{A}_{3j}$,
we find that, for every $\delta>0$, 
\begin{equation}
\mathcal{A}_{31}(t),\,\,\mathcal{A}_{33}(t)=O\left(|t|^{\sigma_{a}+\frac{r}{2}-1+\delta}\right)\label{Estimates for A31 A33}.
\end{equation}
\medskip{}

In order to estimate the middle integral, $\mathcal{A}_{32}$, we can use the
Dirichlet series representation of the product $\psi(z-it)\,\phi(z+it)$. We arrive at ($k:=\sigma_{a}+1-\frac{r}{2}$) 
\begin{align}
\int_{\sigma_{a}+1-i(t-T_{0})}^{\sigma_{a}+1+i(t-T_{0})}\,\lambda_{c}^{2z}\,\left(t^{2}+\left(z-\frac{r}{2}\right)^{2}\right)^{\frac{r-1}{2}}\,\psi(z-it)\,\phi(z+it)\,dz & =|a(c)|^{2}\,\int_{k-i(t-T_{0})}^{k+i(t-T_{0})}\left(t^{2}+z^{2}\right)^{\frac{r-1}{2}}dz\label{writing as Dirichlet series on the first proof}\\
+\sum_{n=c+1}^{\infty}{\overline{a}}(m)\,a(n)\,\left(\frac{\lambda_{m}}{\lambda_{n}}\right)^{it}\left(\frac{\lambda_{c}^{2}}{\lambda_{m}\lambda_{n}}\right)^{\frac{r}{2}}\int_{k-i(t-T_{0})}^{k+i(t-T_{0})}\left(t^{2}+z^{2}\right)^{\frac{r-1}{2}} & \left(\frac{\lambda_{c}^{2}}{\lambda_{m}\lambda_{n}}\right)^{z}\,dz=2i\,|a(c)|^{2}\,|t|^{r}+O\left(|t|^{r-1}\right)\nonumber
\end{align}
upon an integration by parts for the integrals in the infinite series.
Since $\sigma_{a}<\min\{1,r\}+\frac{1}{2}\max\{1,r\}$
by hypothesis, (\ref{inequality r=00005Cleq1}) gives 
\begin{align}
\int_{-t+T_{0}}^{t-T_{0}}\left|R_{\psi}\left(\frac{r}{2}+i\left(y-t\right)\right)R_{\phi}\left(\frac{r}{2}+i\left(y+t\right)\right)\right|\,dy & \geq2\pi\lambda_{c}^{-r}e^{-\pi|t|}\left|2i|a(c)|^{2}\,|t|^{r}+O(|t|^{r-1})+O\left(|t|^{\sigma_{a}+\frac{r}{2}-1+\delta}\right)\right|\nonumber \\
=4\pi\,|a(c)|^{2}\lambda_{c}^{-r}\,|t|^{r}e^{-\pi|t|} & +O\left(|t|^{r-1}\,e^{-\pi|t|}\right)+O\left(|t|^{\sigma_{a}+\frac{r}{2}-1}\,e^{-\pi|t|}\right).\label{preserving the constant intouchable}
\end{align}
\medskip{}

Now, let us note that the second $O$ term in (\ref{preserving the constant intouchable}) always bounds the first
one. As usual, define
\[
\mu(\sigma)=\inf\left\{ \xi\,:\,\phi(\sigma+it)=O(|t|^{\xi})\right\}.
\]
From (\ref{order phragmen introduction}), the general theory of
$\mu(\sigma)$ {[}\cite{titchmarsh_theory_of_functions}, p. 299{]} and the hypothesis
of Theorem \ref{deuring to hold}, we know that $\mu(\sigma)\leq r-2\sigma$ for $\sigma\leq r-\sigma_{a}$.
Since $\mu(\sigma)$ is decreasing, we must have that $0\leq\mu\left(r-\sigma_{a}\right)\leq2\sigma_{a}-r$
and this implies $\sigma_{a}\geq\frac{r}{2}$.
From the contradiction hypothesis (\ref{contradiction hypothesissisis}), (\ref{preserving the constant intouchable})
gives the lower bound for $\mathcal{A}_{3}(t)$
\begin{equation}
|\mathcal{A}_{3}(t)|\geq4\pi\,|a(c)|^{2}\lambda_{c}^{-r}\,|t|^{r}\,e^{-\pi|t|}+O\left(|t|^{\sigma_{a}+\frac{r}{2}-1}\,e^{-\pi|t|}\right),\,\,\,\,\,\,\,|t|>T_{0}.\label{estimate 1 for r=00005Cleq1}
\end{equation}
\medskip{}

Returning to the Selberg-Chowla formula (\ref{Selberg Chowla in an integral form of writing}),
together with (\ref{estimate 1 for r=00005Cleq1}), the hypothesis (\ref{contradiction hypothesissisis})
and the definition of the integrals $\mathcal{A}_{i}(t)$, we deduce the following inequality
\begin{align}
2\,|a(c)|^{2}\lambda_{c}^{-r}\,|t|^{r}\,e^{-\pi|t|} + O\left(|t|^{\sigma_{a}+\frac{r}{2}-1}\,e^{-\pi|t|}\right) & \leq\frac{1}{2\pi}\,|\mathcal{A}_{3}(t)|=\left|\Gamma\left(r+it\right)\,\mathcal{Z}_{2}(r+it;\,a;\,\lambda)-\right.\nonumber \\
-2\phi(0)\,\Gamma(r+2it)\,\phi\left(r+2it\right) &  \left.+2\,\Gamma(r)\,\rho\,\Gamma(2it)\,\phi\left(2it\right)-\frac{1}{2\pi}\mathcal{A}_{1}(t)-\frac{1}{2\pi}\mathcal{A}_{2}(t)\right|.\label{Middle Inequality already to prove}
\end{align}

To contradict hypothesis 1. and 2. for $k=1$, we just need to estimate
the remaining terms appearing on the right-hand side of (\ref{Middle Inequality already to prove}). We have seen already that $\mathcal{A}_{1}(t)$ and $\mathcal{A}_{2}(t)$
satisfy (\ref{estimate auxiliar integrals}), so we just need to estimate
the remaining two. From (\ref{Lindelof Phragmen for Hecke}) and (\ref{Stirling exact form on Introduction}),
it is immediate to see that, for every positive $\delta$, 
\begin{equation}
2\,\Gamma(r)\,\rho\,\Gamma(2it)\,\phi\left(2it\right)-2\phi(0)\,\Gamma(r+2it)\,\phi\left(r+2it\right)=O\left(|t|^{\sigma_{a}-\frac{1}{2}+\delta}\,e^{-\pi|t|}\right),\,\,\,\,|t|\rightarrow\infty \label{estimate residual}.
\end{equation}

Using (\ref{Middle Inequality already to prove}) and the fact that
$\sigma_{a}<r+\frac{1}{2}$, we obtain the lower bound
for the Epstein zeta function 
\begin{equation}
|\Gamma\left(r+2it\right)\,\mathcal{Z}_{2}\left(r+2it;\,a;\,\lambda\right)|\geq2\,|a(c)|^{2}\lambda_{c}^{-r}\,|t|^{r}\,e^{-\pi|t|}+O\left(|t|^{\sigma_{a}+\frac{r}{2}-1}\,e^{-\pi|t|}\right)+O\left(|t|^{\sigma_{a}-\frac{1}{2}+\delta}\,e^{-\pi|t|}\right)\label{the beginning of the boots}
\end{equation}
valid for any $|t|>T_{0}$. Using Stirling's formula (\ref{Stirling exact form on Introduction}) and enlarging $T_{0}$ if needed, (\ref{the beginning of the boots}) immediately implies that, for some
positive constant $B$, the order of $\mathcal{Z}_{2}(s;\,a;\,\lambda)$
at the critical line satisfies the lower bound
\begin{equation}
|\mathcal{Z}_{2}\left(r+2it;\,a;\,\lambda\right)|>B\,|t|^{\frac{1}{2}},\,\,\,\,\,\,|t|>T_{0},\label{BOOTSTRAP}
\end{equation}
and so $\mathcal{Z}_{2}(s;\,a;\,\lambda)$ cannot have
infinitely many zeros on the critical line $\text{Re}(s)=r$, since
it is always greater than a positive function for every $|\text{Im}(s)|>2T_{0}$.
This also contradicts (\ref{condition critical line for r=00005Cleq1})
for $k=1$. To prove the Theorem for a general $k\in \mathbb{N}$, we shall use
Lemma \ref{lemma integral representation} and prove by induction that (\ref{BOOTSTRAP}) implies
(\ref{condition critical line for r=00005Cleq1}). 
\medskip{}

The following claim
contains the necessary assertion to finish the proof: 
\begin{claim} \label{inductive claim}
Suppose that, for some $T_{k}>0$, the $2^{k}-$th Epstein zeta function
has the lower bound on its critical line
\begin{equation}
|\mathcal{Z}_{2^{k}}(2^{k-1}r+it;\,a;\,\lambda)|>B\,|t|^{\frac{2^{k}-1}{2}},\,\,\,\,\,|t|>T_{k},\label{induction condition}
\end{equation}
for some $B>0$. Then we have that $\mathcal{Z}_{2^{k+1}}$ has a
similar bound of the form 
\begin{equation}
|\mathcal{Z}_{2^{k+1}}(2^{k}r+it;\,a;\,\lambda)|>B^{\prime}\,|t|^{\frac{2^{k+1}-1}{2}},\,\,\,\,\,|t|>2T_{k},\,\,\,B^{\prime}>0.\label{inequality lower r=00005Cleq1 Z 2k+1}
\end{equation}
\end{claim}

\begin{proof} [Proof of Claim \ref{inductive claim}]
The idea is to mimic the argument just given. Use (\ref{Essssential to prove Deuring})
and estimate the residual terms. It is clear from (\ref{Estimate general}) and Stirling's
formula that these are bounded as $O\left(|t|^{2^{k}\sigma_{a}-\frac{1}{2}+\delta}\right)$.
To estimate the integral, we take once more a partition similar to
(\ref{decomposition of H into three integrals}). As before, if we
denote each one of these integrals by $\mathcal{A}_{i}^{(k)}(t)$,
$i=1,\,2,\,3$, we can mimic the proof of (\ref{estimate auxiliar integrals})
to obtain
\[
\mathcal{A}_{1}^{(k)}(t),\,\,\mathcal{A}_{2}^{(k)}(t)=O\left(|t|^{2^{k}\,\sigma_{a}-\frac{1}{2}+\delta}\,e^{-\pi|t|}\right),\,\,\,\,\,\,|t|\rightarrow\infty.
\]

On the third integral, $\mathcal{A}_{3}^{(k)}(t)$, we proceed as
in the equality (\ref{contradiction hypothesissisis}): by the inductive
hypothesis (\ref{induction condition}), we know that $\mathcal{Z}_{2^{k}}(s;\,a;\,\lambda)$
has constant sign on the line $\text{Re}(s)=2^{k-1}r$ for $|\text{Im}(s)|>T_{k}$.
Employing directly (\ref{induction condition}) on the third integral
$\mathcal{A}_{3}^{(k)}(t)$ we obtain, for any $|t|>T_{k}$, 
\begin{align*}
\left|\mathcal{A}_{3}^{(k)}(t)\right| & =\int_{-t+T_{k}}^{t-T_{k}}\left|\Gamma\left(2^{k-1}r+i(y-t)\right)\,\mathcal{Z}_{2^{k}}\left(2^{k-1}r+i(y-t);\,b;\,\lambda\right)\right. \times\\
 & \,\hfill\left. \times \Gamma\left(2^{k-1}r+i(y+t)\right)\,\mathcal{Z}_{2^{k}}\left(2^{k-1}r+i(y+t);\,a;\,\lambda\right)\right|\,dy>\\
> & B\,e^{-\pi|t|}\left|\int_{-t+T_{k}}^{t-T_{k}}\left(t^{2}-y^{2}\right)^{2^{k-1}(r+1)-1}\,dy\right|=D\,\,|t|^{2^{k}(r+1)-1}\,e^{-\pi |t|}+O\left(|t|^{2^{k}(r+1)-2}\,e^{-\pi |t|}\right)
\end{align*}
for some $D>0$. From the Selberg-Chowla formula (\ref{Essssential to prove Deuring})
we have that, for $|t|>T_{k}$ and some positive constant $D^{\prime}$, 
\begin{equation}
|\Gamma(2^{k}r+2it)\,\mathcal{Z}_{^{2^{k+1}}}(2^{k}r+2it;\,a;\,\lambda)|>D^{\prime}\,|t|^{2^{k}(r+1)-1}\,e^{-\pi|t|},\label{inequality yey for product gamma with Epstein}
\end{equation}
which implies (\ref{inequality lower r=00005Cleq1 Z 2k+1}).
\end{proof}

\medskip{}

Returning to the proof of Theorem \ref{deuring to hold}, the assumption that $\phi(s)$
does not have infinitely many zeros on the critical line implies
(\ref{induction condition}) for $k=1$ and $T_{1}:=2T_{0}$. But then the previous claim shows that this contradicts conditions 1. and 2. on the statement and the proof follows. \end{proof}

\begin{remark} \label{theorem 4.1 to general residual}
It is clear from the previous proof that the argument can be extended
to all Dirichlet series satisfying definition \ref{definition 1.1.} such that $\Gamma(s)\,\phi(s)$ is analytic on the critical line $\text{Re}(s)=r/2$. Hence, the imposition
$\phi\in\mathcal{A}$ in the statement of Theorem \ref{deuring to hold} can be removed.
Note that the only time we have used the fact that $\phi\in\mathcal{A}$ (besides in assuming that $R_{\phi}(s)$ is holomorphic on the critical line) 
was in the estimate (\ref{estimate residual}) of the residual terms
appearing in (\ref{Again First-1}), (\ref{Again Second-1}). Writing $s=\sigma+it$ in (\ref{integral varphi 2}), we know that, as $|t|$ tends to infinity, $\Gamma(\sigma+it-w)\,\phi(\sigma+it-w)=O\left(|t|^{\sigma_{a}-\frac{1}{2}+\delta}\,e^{-\frac{\pi}{2}|t|}\right)$
uniformly in the curve $C_{1}$, so that
\begin{align*}
\varphi_{2}(\sigma+it;\,a;\,\lambda) & =O\left(|t|^{\sigma_{a}-\frac{1}{2}+\delta}\,e^{-\frac{\pi}{2}|t|}\,\int_{C_{1}}|\Gamma(w)\,\phi(w)|\,|dw|\right)=O\left(|t|^{\sigma_{a}-\frac{1}{2}+\delta}\,e^{-\frac{\pi}{2}|t|}\right),\,\,\,\,\,|t|\rightarrow\infty,
\end{align*}
which is precisely (\ref{estimate residual}).

\end{remark}

\begin{remark} \label{Ramachandra Remark}
Note that the condition $\sigma_{a}<1+\frac{r}{2}$ was
essential to move the line of integration on the third integral $\mathcal{A}_{3}(t)$
from the critical line $\text{Re}(s)=\frac{r}{2}$ to $\text{Re}(s)=\sigma_{a}+1$.
In fact, this condition is an essential tool to prove that the estimate
(\ref{Estimates for A31 A33}) of the horizontal integrals is actually
bounded by $|t|^{r}$, hence allowing to establish the lower bound
(\ref{estimate 1 for r=00005Cleq1}). If, however, we disregard such generality and assume that the generalized Dirichlet series $\phi(s)$ is a Titchmarsh series {[}\cite{ramachandra}, Chapter 2, p. 39{]},
then a Theorem of Ramachandra {[}\cite{balasubramanian}, p. 570, Thm.
3{]} would still allow to conclude a lower bound of the form (\ref{BOOTSTRAP}). Examples of ``Titchmarsh series'' include the class of Hecke Dirichlet
series with signature $(\lambda,\,r,\,\gamma)$ and so the condition
$\sigma_{a}<1+\frac{r}{2}$ in the statement of the previous theorem can be suspended for this class. 
In any case, the imposition $\sigma_{a}<r+\frac{1}{2}$
needs to be kept since it assures that $|t|^{r}$ exceeds the estimates
of the integrals $\mathcal{A}_{1}(t)$ and $\mathcal{A}_{2}(t)$ (\ref{estimate auxiliar integrals}),
as well as the estimates for the residual terms (\ref{estimate residual})
appearing in the Selberg-Chowla formula. For an application of Ramachandra's result, see \cite{selberg class hardy},
where an adaptation of the Potter-Titchmarsh method \cite{Titchmarsh_Potter},
together with Ramachandra's Theorem, are essential to give a proof
of Hardy's Theorem for certain Dirichlet series belonging to
the Selberg class with degree two. 
\end{remark}
\begin{remark}
Condition 2. (\ref{condition critical line for r=00005Cleq1}) on the previous theorem was invoked by Deuring to prove
the infinitude of zeros of $\zeta(s)$ on $\text{Re}(s)=\frac{1}{2}$. This argument reduces the proof of Hardy's Theorem to the establishment of the order of magnitude of a Dirichlet series (in this case $\zeta_{2}(s)$) on its critical line $\text{Re}(s)=\frac{1}{2}$. Although this may be extraordinarily difficult to establish in general, note that we do not need explicitly (\ref{condition critical line for r=00005Cleq1}) to apply our method. What is actually required to complete the proof of Theorem \ref{deuring to hold} is to verify that one lower bound of the form (\ref{inequality lower r=00005Cleq1 Z 2k+1}) cannot hold, i.e., to find at least one dimension in which the absolute value of the Epstein zeta function, $|\mathcal{Z}_{2^k}(s;\,\cdot)|$, "oscillates" along a curve of the form $C|t|^{\alpha},\,\,0\leq\alpha\leq2^{k-1}-\frac{1}{2}$.  This is why the first condition may add simplicity to the
theorem, as the infinitude of zeros of $\mathcal{Z}_{2}(s;\,a;\,\lambda)$
can be established independently. This is certainly the case for the
Dirichlet series $\phi(s)=\zeta_{2}(s)=\sum r_{2}(n)/n^{s}$, whose
infinitude of zeros on the critical line $\text{Re}(s)=\frac{1}{2}$
can be easily proved by invoking the fact that $\zeta_{4}(s)$ has
infinitely many zeros on the line $\text{Re}(s)=1$. In its turn, the infinitude
of zeros of $\zeta_{4}(s)$ in the line $\text{Re}(s)=1$ may be obtained 
in an independent way, by invoking, for instance, Jacobi's 4-square Theorem, which is a quantitative way of telling what is the projection of the modular form $\theta^{4}(z)$ onto the subspace of Eisenstein series. To prove Hardy's Theorem to multidimensional Epstein zeta functions, a similar reasoning is made by using a version of a formula of Siegel (see
Examples \ref{hardy and jacobi} and \ref{epstein example}). 
\end{remark}

Based upon the previous Theorem, we can generalize several parallel results
which are scattered in the literature, including proofs involving
the behavior of the $\theta-$function. The following Corollary appears to be new.


\begin{corollary}\label{corollary on theta}
Assume that $\phi(s)$ is a Dirichlet series satisfying the conditions
of Theorem \ref{deuring to hold} and $\Theta(z;\,a;\,\lambda)$ denote its generalized 
$\theta-$function (\ref{definition generalized Theta function nnnn.}). If 
\begin{equation}
\Theta\left(e^{i\left(\frac{\pi}{2}-\epsilon\right)};\,a;\,\lambda\right)=o\left(\epsilon^{-\frac{r+1}{2}}\right),\,\,\,\,\,\epsilon\rightarrow0^{+},\label{contradiction hypothesis onceee more}
\end{equation}
then $\phi(s)$ has infinitely many zeros on the critical line $\text{Re}(s)=\frac{r}{2}$.
\end{corollary}

\begin{proof}
Assume that $\phi(s)$ has finitely many zeros on the critical line
$\text{Re}(s)=\frac{r}{2}$. It follows from the proof of the previous Theorem
that, at its critical line $\text{Re}(s)=r$, the Epstein zeta function $\mathcal{Z}_{2}(s;\,a;\,\lambda)$ satisfies the inequality
\begin{equation}
|\Gamma\left(r+2it\right)\,\mathcal{Z}_{2}\left(r+2it;\,a;\,\lambda\right)|\geq2\,|a(c)|^{2}\lambda_{c}^{-r}\,|t|^{r}\,e^{-\pi|t|}+O\left(|t|^{\sigma_{a}+\frac{r}{2}-1}\,e^{-\pi|t|}\right)+O\left(|t|^{\sigma_{a}-\frac{1}{2}+\delta}\,e^{-\pi|t|}\right),\,\,\,\,|t|>T_{0}.\label{inequality of the thing}
\end{equation}

For $\text{Re}(z)>0$ and $\mu>2\sigma_{a}$, we have from
(\ref{Cahen Mellin integral for applications}) that $\Theta_{2}(z;\,a;\,\lambda)$
admits the representation as the Mellin integral
\begin{align}
\Theta_{2}(z;\,a;\,\lambda) & =\frac{1}{2\pi i}\,\int_{\mu-i\infty}^{\mu+i\infty}\Gamma(s)\,\mathcal{Z}_{2}(s;\,a;\,\lambda)\,z^{-s}ds\nonumber \\
 & =\frac{1}{2\pi i}\,\int_{C_{1}}\Gamma(s)\,\mathcal{Z}_{2}(s;\,a;\,\lambda)\,z^{-s}ds+\frac{1}{2\pi i}\,\int_{C_{2}}\Gamma(s)\,\mathcal{Z}_{2}(s;\,a;\,\lambda)\,z^{-s}ds,\label{Expression Theta 2k+1}
\end{align}
where $C_{1}$ and $C_{2}$ are the paths given by 
\[
C_{1}=\left(r-i\infty,\,r-2iT_{0}\right)\cup\left(r+2iT_{0},\,r+i\infty\right)
\]
and
\[
C_{2}=\left(r-2iT_{0},\,\mu-2iT_{0}\right)\cup\left(\mu-2iT_{0},\,\mu+2iT_{0}\right)\cup\left(\mu+2iT_{0},\,r+2iT_{0}\right).
\]

If we now take $z=\exp\left\{ i\left(\frac{\pi}{2}-\epsilon\right)\right\} $
in (\ref{Expression Theta 2k+1}), we have that 
\begin{equation}
\Theta_{2}\left(e^{i\left(\frac{\pi}{2}-\epsilon\right)};\,a;\,\lambda\right)=\frac{1}{2\pi i}\,\int_{C_{1}}\Gamma(s)\,\mathcal{Z}_{2}(s;\,a;\,\lambda)\,e^{-i\left(\frac{\pi}{2}-\epsilon\right)s}ds+O(1).\label{integral expressing theta once more}
\end{equation}

From the functional equation for $\mathcal{Z}_{2}(s;\,a;\,\lambda)$
(\ref{Functional equation our Equation-1}) and the assumption that
$b(n)=\overline{a}(n)$, $\Gamma(s)\,\mathcal{Z}_{2}(s;\,a;\,\lambda)$
is a real valued function on the line $\text{Re}(s)=r$. From this
observation and combining (\ref{integral expressing theta once more})
with (\ref{inequality of the thing}), we obtain that the inequality
\begin{align}
\left|\Theta_{2}\left(e^{i\left(\frac{\pi}{2}-\epsilon\right)};\,a;\,\lambda\right)\right| & =\frac{1}{2\pi}\,\int_{|y|>2T_{0}}\left|\Gamma\left(r+iy\right)\,\mathcal{Z}_{2}(r+iy;\,a;\,\lambda)\right|\,e^{(\frac{\pi}{2}-\epsilon)y}dy+O(1)\nonumber \\
&\geq\frac{|a(c)|^{2}(2\lambda_{c})^{-r}}{\pi} \,\int_{2T_{0}}^{\infty}\,y^{r}\,e^{-\epsilon y}\,dy+O\left(\int_{2T_{0}}^{\infty}y^{\sigma_{a}+\frac{\max\{1,r\}}{2}-1}e^{-\epsilon y}dy\right)\nonumber\\
&=\frac{|a(c)|^{2}(2\lambda_{c})^{-r}\,\Gamma(r+1)}{\pi}\,\epsilon^{-(r+1)}+O\left(\epsilon^{-\sigma_{a}-\frac{\max\{1,r\}}{2}}\right),\label{Contradiction hypothesis for proof corollary Deuring once more}
\end{align}
is valid for every positive $\epsilon$. However, since 
\begin{equation}
\Theta_{2}\left(z;\,a;\,\lambda\right)=\sum_{m,n\neq0}^{\infty}a(m)\,a(n)\,e^{-(\lambda_{m}+\lambda_{n})\,z}=2a(0)\,\Theta(z;\,a;\,\lambda)+\Theta^{2}(z;\,a;\,\lambda),\label{the identity of square}
\end{equation}
it follows from (\ref{contradiction hypothesis onceee more}) that
$\Theta_{2}\left(e^{i\left(\frac{\pi}{2}-\epsilon\right)};\,a;\,\lambda\right)=o\left(\epsilon^{-r-1}\right)$
as $\epsilon\rightarrow0^{+}$, which contradicts (\ref{Contradiction hypothesis for proof corollary Deuring once more}).
\end{proof}

\begin{remark} \label{remark Hecke condition}
Note that, when restricted to the class $\mathcal{A}$, as well as to the Dirichlet series satisfying definition \ref{definition 1.1.}, the previous corollary extends Berndt's result {[}\cite{berndt_zeros_(ii)}, p. 679,
Theorem 1{]}. When restricted to the class of Dirichlet series satisfying Theorem
\ref{deuring to hold}, (\ref{contradiction hypothesis onceee more}) improves Hecke's condition (see {[}\cite{Hecke_middle_line},
p. 74, Satz 1{]}), 
\begin{equation}
\Theta\left(e^{i\left(\frac{\pi}{2}-\epsilon\right)};\,a;\,\lambda\right)=O(\epsilon^{-\beta}), \,\,\,\,\,\,\, 0\leq\beta<\frac{r+1}{2}. \label{Hecke's Condition in his paper}
\end{equation}
\end{remark}



\begin{corollary}\label{corollary the one with residue}
Let $\phi(s)$ be a Dirichlet series with $r<1$ and satisfying the
conditions of Theorem \ref{deuring to hold}. Assume also that $a(n)$ satisfies $\arg\left\{ a(n)\right\} _{n\in\mathbb{N}}=\text{constant}$.
Then $\phi(s)$ has infinitely many zeros on the critical line $\text{Re}(s)=\frac{r}{2}$.  Moreover, if $r=1$ and if $a(n)$ satisfies the condition 
\begin{equation}
|a(c)|>\sqrt{2\pi\lambda_{c}}\,|\rho|,\label{condition on r=00003D1 for |a(c)|}
\end{equation}
where $\rho$ denotes the residue of $\phi$ at $s=r=1$ and $c$ denotes the least positive integer for which $a(c)\neq 0$, then $\phi(s)$
has infinitely many zeros at the critical line $\text{Re}(s)=\frac{1}{2}$.
\end{corollary}
\begin{proof}
Assume first that $0<r<1$ and that $\phi(s)$ does not have infinitely
many zeros on the critical line $\text{Re}(s)=\frac{r}{2}$. Then
an application of (\ref{Contradiction hypothesis for proof corollary Deuring once more}) gives 
\begin{equation}
\epsilon^{r+1}\left|\Theta_{2}\left(e^{i\left(\frac{\pi}{2}-\epsilon\right)};\,a;\,\lambda\right)\right|\geq C,\label{contradition conditions ALMOST THERE}
\end{equation}
for some positive constant $C$. On the
other hand, using Bochner's modular relation for $\Theta_{2}(z;\,a;\,\lambda)$
(\ref{reflection formula Bochner Epstein two-dimensional}) and the
hypothesis of Theorem \ref{deuring to hold}, $b(n)=\,\overline{a}(n)$,
we know that 
\begin{equation}
\Theta_{2}\left(z;\,a;\,\lambda\right)=z^{-2r}\sum_{m,n=1}^{\infty}\overline{a}(m)\,\overline{a}(n)\,e^{-(\lambda_{m}+\lambda_{n})/z}+\Gamma^{2}(r)\,\rho^{2}z^{-2r}+\mathcal{Z}_{2}(0,\,a,\,\lambda).\label{Reflection Formula Here!}
\end{equation}

Using the fact that $\arg_{n\in\mathbb{N}}\{a(n)\}$ is
constant, we have from the previous relation, 
\begin{align}
\left|\Theta_{2}\left(e^{i\left(\frac{\pi}{2}-\epsilon\right)};\,a;\,\lambda\right)\right| & \leq\sum_{m,n\neq0}^{\infty}\left|a(m)\,a(n)\,e^{-(\lambda_{m}+\lambda_{n})\sin(\epsilon)}\right|=\left|\sum_{m,n=1}^{\infty}a(m)\,a(n)\,e^{-(\lambda_{m}+\lambda_{n})\sin(\epsilon)}\right|\nonumber \\
 & =\left|\sin^{-2r}(\epsilon)\,\sum_{m,n\neq0}^{\infty}\overline{a}(m)\,\overline{a}(n)\,e^{-(\lambda_{m}+\lambda_{n})/\sin(\epsilon)}+\Gamma^{2}(r)\,\rho^{2}\sin^{-2r}(\epsilon)+\mathcal{Z}_{2}(0,\,a,\,\lambda)\right|\nonumber \\
 & =\Gamma^{2}(r)\,|\rho|^{2}\sin^{-2r}(\epsilon)+O(1).\label{The application of the Reflection formula for deducing infinitude kober styyyle}
\end{align}

Since, by hypothesis, (\ref{contradition conditions ALMOST THERE})
holds, then we must have
\begin{equation}
0<C\leq\epsilon^{r+1}\left|\Theta_{2}\left(e^{i\left(\frac{\pi}{2}-\epsilon\right)};\,a;\,\lambda\right)\right|=\epsilon^{r+1}\,\Gamma^{2}(r)\,|\rho|^{2}\sin^{-2r}(\epsilon)+O(\epsilon^{r+1})\label{K contradiction}
\end{equation}
for every $\epsilon>0$. However,
since $r<1$, we see that the right-hand side of (\ref{K contradiction})
tends to zero as $\epsilon\rightarrow0^{+}$, contradicting (\ref{contradition conditions ALMOST THERE}).
This concludes the proof for the first case. 

\medskip{}

Assume now that $r=1$: from (\ref{Contradiction hypothesis for proof corollary Deuring once more})
with $r=1$, we know that the constant $C$ given above can be explicitly written as $C=|a(c)|^{2}/2\pi\lambda_{c}$,
and so (\ref{K contradiction}) contradicts (\ref{condition on r=00003D1 for |a(c)|}).
\end{proof}


\begin{remark} \label{fekete remark}
By a Theorem of Fekete \cite{fekete_dirichlet}, which in its turn extends a classical
Theorem of Landau on Dirichlet series, 
the condition $\arg_{n\in\mathbb{N}}\left\{ a(n)\right\} =\text{const.}$
implies that $\phi(s)$ cannot be entire. Moreover, if we assume that
$\phi\in\mathcal{A}$, then the previous Corollary is only valid for
$\sigma_{a}=r$. In any case it is not strictly necessary that $\arg_{n\in\mathbb{N}}\left\{ a(n)\right\} =\text{const.}$
to arrive at the conclusion of Corollary \ref{corollary on theta}. In fact, it suffices to
suppose that $a(n)$ satisfies $|a(n)|\leq\mathfrak{c}(n)$
for every $n\in\mathbb{N}$, with $\mathfrak{c}(n)$ being such that
the Dirichlet series $\varphi(s)=\sum\mathfrak{c}(n)\,\lambda_{n}^{-s}$
satisfies Hecke's functional equation with $r\leq1$, as well as
the conditions of Theorem \ref{deuring to hold} For example, Dirichlet $L$-functions
attached to even primitive Dirichlet characters satisfy this condition. 

\end{remark}
\bigskip{}

It is also of interest to note that (\ref{BOOTSTRAP}) contradicts
any bound of the type $\mathcal{Z}_{2}\left(r+it;\,a;\,\lambda\right)=o\left(|t|^{\frac{1}{2}}\right)$,
which may be obtained independently via more sophisticated methods
employing exponential sums, as those in \cite{exponential_sums_jutilla}. Under our general setting, however, it is not simple to assure these subconvex estimates for the diagonal Epstein zeta functions $\mathcal{Z}_{2^{k}}$
on their critical lines. In general, all the information about the
behavior of $\mathcal{Z}_{2^{k}}$ on the line $\text{Re}(s)=2^{k-1}r$
comes exclusively from the Phragm\'en-Lindel\"of principle. 

Based solely on these classical estimates, in the next result we impose a condition
which implies that, for a sufficiently large $k$, the order of the
dyadic Epstein zeta function $\mathcal{Z}_{2^{k+1}}$ at the critical
line $\text{Re}(s)=2^{k}r$ will eventually contradict (\ref{induction condition}). 

We remark that the next result was already proved by Berndt \cite{berndt_zeros_(i)}
under a different setting. Since our proof is drastically different
from his and it has the advantage of avoiding the exponential
integrals typical in most part of
the proofs of Hardy's Theorem\footnote{see, for instance, \cite{selberg class hardy, Titchmarsh_Potter, chadrasekharan_narasimhan_ideal classes, landau_handbuch}, where exponential integrals are used to deduce analogues of Hardy's Theorem. The idea behind all these proofs, originally due to Landau \cite{landau_hardy},
is to contrast the behaviors of $\left|\int_{T}^{2T}R_{\phi}\left(\frac{r}{2}+it\right)\,dt\right|$
and $\int_{T}^{2T}|R_{\phi}(\frac{r}{2}+it)|\,dt$ as $T\rightarrow\infty$.
Due to the oscillations resulting from Stirling's formula, the integral
$\int_{T}^{2T}R_{\phi}\left(\frac{r}{2}+it\right)\,dt$ has substantial
cancellation and the bound for it is usually found by appealing to
exponential integrals, which cannot be avoided even if we use, for
the case of Riemann's $\zeta-$function, weak versions of the approximate
functional equation and the Riemann-Siegel formula (see {[}\cite{sangale},
p. 43{]} and {[}\cite{ivic_hardyz}, Chpt. 2{]}). By appealing to
the freedom of selecting a suitable dimension for the Epstein zeta function,
our proof of the corollary \ref{narrow abcissa corollary} avoids non-elementary estimates and it seeks
the contradiction by finding an
Epstein zeta function not satisfying the lower bound (\ref{induction condition})
in its critical line. }, we present our alternative short proof. Being an immediate application of the argument developed in the proof of Theorem \ref{deuring to hold}, we remark once more
that the foregoing corollary can be extended to the
Dirichlet series satisfying definition \ref{definition 1.1.}, not necessarily in $\mathcal{A}$. 

\begin{corollary} \label{narrow abcissa corollary}
Let $\phi(s)$ be a Dirichlet series satisfying the conditions of
Theorem \ref{deuring to hold}. Moreover, assume that its abcissa of absolute convergence satisfies
\begin{equation}
\sigma_{a}<\frac{r+1}{2}.\label{Condition Narrow critical strip}
\end{equation}

Then $\phi(s)$ has infinitely many zeros on the critical line $\text{Re}(s)=\frac{r}{2}$.
\end{corollary}

\begin{proof}
By the Phragm\'en-Lindel\"of principle applied to the Dirichlet series
$\mathcal{Z}_{2^{k+1}}(s;\,a;\,\lambda)$, we have that, for every positive $\delta$, 
\begin{equation}
\mathcal{Z}_{2^{k+1}}(2^{k}r+2it;\,a;\,\lambda)=O\left(|t|^{2^{k+1}\sigma_{a}-2^{k}r+\delta}\right),\,\,\,\,\,|t|\rightarrow\infty.\label{Phragmen contradictory here}
\end{equation}

However, if $\phi(s)$ has finitely many zeros on the line $\text{Re}(s)=\frac{r}{2}$,
then there is some $T_{0}>0$ such that, for every $k\in\mathbb{N}_{0}$
(see claim \ref{inductive claim})
\begin{equation}
|\mathcal{Z}_{^{2^{k+1}}}(2^{k}r+2it;\,a;\,\lambda)|>B\,|t|^{\frac{2^{k+1}-1}{2}},\,\,\,\,\,|t|>T_{k}:=2^{k}\,T_{0},\label{comparison once more}
\end{equation}
for some $B>0$ depending on $k$. After a direct comparison between
(\ref{Phragmen contradictory here}) and (\ref{comparison once more}),
and letting $|t|$ tend to infinity, we deduce that the inequality $\sigma_{a}\geq\frac{r+1}{2}-\frac{1}{2^{k+2}}-\delta^{\prime}$ must hold for every $k\in\mathbb{N}_{0}$ and every $\delta^{\prime}>0$. This contradicts (\ref{Condition Narrow critical strip})
and the corollary follows.
\end{proof}

\bigskip{}

Another application of the method developed in Theorem \ref{deuring to hold} is the
proof that, for the class of Dirichlet series considered in this section,
every combination of bounded vertical shifts of $R_{\phi}(s)$ has
infinitely many zeros on the critical line $\text{Re}(s)=\frac{r}{2}$. 

Let $\left(c_{j}\right)_{j\in\mathbb{N}}$ be a sequence of real numbers
such that $\sum_{j=1}^{\infty}|c_{j}|<\infty$ and $\left(\tau_{j}\right)_{j\in\mathbb{N}}$
be a bounded sequence of real numbers. For all $s\in\mathbb{C}$,
define
\[
F_{\phi}(s):=\sum_{j=1}^{\infty}c_{j}\,\Gamma\left(s+i\tau_{j}\right)\,\phi(s+i\tau_{j})=\sum^{\infty}_{j=1} c_{j}\,R_{\phi}(s+i\tau_{j}).
\]

Note that, if $\phi\in\mathcal{A}$, $F_{\phi}(s)$ has poles located
at $s=-i\tau_{j}$ and $s=r-i\tau_{j}$ for all $j\in\mathbb{N}$.
From the fact that the poles of $\phi$ are, in general, confined
to a compact set (see definition \ref{definition 1.1.}) and the boundedness of the
sequence $(\tau_{j})_{j\in\mathbb{N}}$, one can show that there exists
a bounded set $D\subset\mathbb{C}$ such that $F_{\phi}(s)$ is analytic
on $\mathbb{C}\setminus D$. For the case where $\phi\in\mathcal{A}$,
$D$ can be taken as the union of two bounded vertical intervals containing
$s=0$ and $s=r$. If $(\tau_{j})_{j\in\mathbb{N}}$ is an infinite
sequence, then $F_{\phi}(s)$ has essential singularities inside $D$,
but despite this it can be seen that the function $F_{\phi}\left(\frac{r}{2}+iy\right)$
is well defined for $y\in\mathbb{R}$ if $\Gamma(s)\,\phi(s)$ is
analytic on the line $\text{Re}(s)=\frac{r}{2}$. Moreover, if the
parameters of $\phi$ and $\psi$ satisfy $\overline{a}(n)=b(n)$
and $\lambda_{n}=\mu_{n}$, then $F_{\phi}\left(\frac{r}{2}+iy\right)$
is real valued for $y\in\mathbb{R}$.

In the next corollary we give sufficient conditions for $F_{\phi}(s)$
to have infinitely many zeros on the critical line $\text{Re}(s)=\frac{r}{2}$.
For the case where $\phi(s)=\pi^{-s}\zeta(2s)$, this result has been
proved by A. Dixit, N. Robles, A. Roy and A. Zaharescu \cite{combinations_dixit}.
Their proof employed a variation of Hardy's method and involved the
computation of the shifted moments of the Riemann $\xi-$function. Due to
the limitations of generalizing the method of the moments to other classes
of Dirichlet series, their proof cannot be easily extended to the class of Dirichlet series satisfying definition \ref{definition 1.1.}. However,
the method employed by us in the proof of Theorem \ref{deuring to hold} and Corollary \ref{corollary on theta} can. Although the following result is stated for the class $\mathcal{A}$
only, by Remarks \ref{remark with general residual} and \ref{theorem 4.1 to general residual}, it can be also expected to hold for a more general class of Dirichlet series. 

\begin{corollary} \label{combination}
Let $\left(c_{j}\right)_{j\in\mathbb{N}}$ be a sequence of non-zero
real numbers such that $\sum|c_{j}|<\infty$ and $\tau_{j}$ be a
bounded sequence of distinct real numbers. Assume also
that $\phi$ satisfies the conditions given in the statement
of Theorem \ref{deuring to hold}. If one of the following condition holds: 
\begin{enumerate}
\item The generalized $\theta-$function, $\Theta(z;\,a;\,\lambda)$, satisfies
$\Theta\left(e^{i\left(\frac{\pi}{2}-\epsilon\right)};\,a;\,\lambda\right)=o\left(\epsilon^{-\frac{r+1}{2}}\right)$
as $\epsilon\rightarrow0^{+}$. 
\item The abcissa of absolute convergence of $\phi$, $\sigma_{a}$, satisfies
the condition $\sigma_{a}<\frac{r+1}{2}$. 
\end{enumerate}
Then the function $F_{\phi}(s):=\sum_{j=1}^{\infty}\,c_{j}\,\Gamma(s+i\tau_{j})\,\phi(s+i\tau_{j})$
has infinitely many zeros on the critical line $\text{Re}(s)=\frac{r}{2}$.
\end{corollary}

\begin{proof}
Throughout this proof we shall denote $F_{\psi}(s):=\sum_{j=1}^{\infty}\,c_{j}\,\Gamma(s+i\tau_{j})\,\psi(s+i\tau_{j})$.
We just show this corollary for the first item, as for the second
a simple adaptation of the proofs of Claim \ref{inductive claim}  and Corollary \ref{narrow abcissa corollary} will suffice. The proof
starts with the sequence of integral representations (\ref{Essssential to prove Deuring})
provided by Lemma \ref{lemma integral representation}. Recalling (\ref{Selberg Chowla in an integral form of writing}),
it is simple to see that, for any elements $\tau_{j}$ and $\tau_{k}$
belonging to the bounded sequence $(\tau_{j})_{j\in\mathbb{N}}$,
the following representation holds 
\begin{align*}
\Gamma(r+2it+i(\tau_{k}-\tau_{j}))\,\mathcal{Z}_{2}\left(r+2it+i(\tau_{k}-\tau_{j});\,a;\,\lambda\right) & =-2\phi(0)\,\Gamma(r+2it+i(\tau_{k}-\tau_{j}))\,\phi\left(r+2it+i(\tau_{k}-\tau_{j})\right)+\\
+2\,\Gamma(r)\,\rho\,\Gamma(2it+i(\tau_{k}-\tau_{j}))\,\phi\left(2it+i(\tau_{k}-\tau_{j})\right)
+\frac{1}{2\pi}\,\int_{-\infty}^{\infty} & R_{\psi}\left(\frac{r}{2}+i(y+\tau_{j}-t)\right)\,R_{\phi}\left(\frac{r}{2}+i(y+\tau_{k}+t)\right)\,dy.
\end{align*}
Multiplying the previous equality by $c_{j}\,c_{k}$ and summing over
$j$ and $k$ and using the hypothesis $(c_{j})_{j\in\mathbb{N}}\in\ell^{1}$, we arrive at the equality 
\begin{align}
\sum_{j,k=1}^{\infty}c_{j}\,c_{k}\Gamma(r+2it+i(\tau_{k}-\tau_{j}))\,\mathcal{Z}_{2}\left(r+2it+i(\tau_{k}-\tau_{j});\,a;\,\lambda\right)=-2\phi(0)\sum_{j,k=1}^{\infty}c_{j}\,c_{k}\,R_{\phi}\left(r+2it+i(\tau_{k}-\tau_{j})\right) +\nonumber \\ 
+2\,\Gamma(r)\,\rho\,\sum_{j,k=1}^{\infty}\,c_{j}\,c_{k}\,R_{\phi}\left(2it+i(\tau_{k}-\tau_{j})\right)+\frac{1}{2\pi}\,\int_{-\infty}^{\infty}F_{\psi}\left(\frac{r}{2}+i(y-t)\right)\,F_{\phi}\left(\frac{r}{2}+i(y+t)\right)\,dy,\label{series of series}\ 
\end{align}
which is now useful to employ the contradiction hypothesis. Assume
that $F_{\phi}(s)$ does not have infinitely many zeros on the line
$\text{Re}(s)=\frac{r}{2}$. 
Following the proof of Theorem \ref{deuring to hold} and using the fact that
$(\tau_{j})_{j\in\mathbb{N}}$ is bounded and that all infinite series converge
absolutely (so that we can apply all the estimates given in (\ref{Estimates for A31 A33})
(\ref{preserving the constant intouchable}) and (\ref{estimate residual})
uniformly), we arrive at a lower bound similar to (\ref{preserving the constant intouchable})
of the form 
\begin{equation}
\left|\sum_{j,k=1}^{\infty}c_{j}\,c_{k}\,\Gamma(r+2it+i(\tau_{k}-\tau_{j}))\,\mathcal{Z}_{2}\left(r+2it+i(\tau_{k}-\tau_{j});\,a;\,\lambda\right)\right|>C\,|t|^{r}\,e^{-\pi|t|},\,\,\,\,|t|>T_{0},\label{clearly contradicts hypo}
\end{equation}
where $T_{0}$ is such that $F_{\phi}\left(\frac{r}{2}+iy\right)$
has no zeros for $|y|>T_{0}$ and $C$ is some positive constant.
Note also that, by the functional equation for $\mathcal{Z}_{2}(s;\,a;\,\lambda)$,
the double series on the left-hand side of (\ref{clearly contradicts hypo})
is a real function of $t$. 

The inequality (\ref{clearly contradicts hypo}) clearly contradicts
the first hypothesis of our corollary: indeed, from (\ref{integral expressing theta once more}), 
\[
e^{-({\tau_{k}-\tau_{j}})\left(\frac{\pi}{2}-\epsilon\right)}\,\Theta_{2}\left(e^{i\left(\frac{\pi}{2}-\epsilon\right)};\,a;\,\lambda\right)=\frac{1}{2\pi i}\,\int_{C_{1}}\Gamma(s+i(\tau_{k}-\tau_{j}))\,\mathcal{Z}_{2}(s+i(\tau_{k}-\tau_{j});\,a;\,\lambda)\,e^{-i\left(\frac{\pi}{2}-\epsilon\right)s}ds+O(1),
\]
where $C_{1}$ denotes the path $\left(r-i\infty,\,r-2iT_{0}\right)\cup\left(r+2iT_{0},\,r+i\infty\right)$. Following the proof of Corollary \ref{corollary on theta}, (\ref{clearly contradicts hypo}) yields
\begin{equation}
\sum_{j,k=1}^{\infty}|c_{j}\,c_{k}|\,e^{-(\tau_{k}-\tau_{j})\left(\frac{\pi}{2}-\epsilon\right)}\,\left|\Theta_{2}\left(e^{i\left(\frac{\pi}{2}-\epsilon\right)};\,a;\,\lambda\right)\right|>D\,\epsilon^{-(r+1)}+O(1),\label{immediately implies contradiction on combination}
\end{equation}
for every positive $\epsilon$. Since the double sequence $A_{j,k}:=(\tau_{k}-\tau_{j})_{j,k\in\mathbb{N}}$
is bounded and $\sum_{j,k}|c_{j}\,c_{k}|<\infty$ by hypothesis, it
is clear that the double series in the left-hand side of the previous
inequality is convergent. Henceforth, (\ref{immediately implies contradiction on combination})
immediately contradicts the first assumption, by virtue of (\ref{the identity of square}). 
\end{proof}



\bigskip{}


As an application of the method described in Theorem \ref{deuring to hold} and Corollary
\ref{narrow abcissa corollary}, we now give a quantitative result regarding the distribution
of the zeros of $\phi(s)$ on the critical line, i.e., for a sufficiently large
$T$ we find $H=H(T)$ such that the interval $[T,\,T+H]$ contains
the ordinate of a zero of $\phi(s)$ located at the critical line
$\text{Re}(s)=\frac{r}{2}$. This method may be useful in future approaches to the problem of finding $H(T)$ for which a given Dirichlet series possesses a zero on its critical line with ordinate between $T$ and $T+H$. 

\begin{corollary} \label{quantitative Classic}
Let $\phi(s)$ be a Dirichlet series satisfying the conditions of
Corollary \ref{narrow abcissa corollary}. Then, for any fixed $\epsilon>0$, there exists a positive
number $T_{0}(\epsilon)$ such that, for all $T\geq T_{0}(\epsilon)$, 
there is a zero $s=\frac{r}{2}+i\tau$ of $\phi(s)$ with 
\begin{equation}
|\tau-T|\leq T^{\sigma_{a}+\frac{1-r}{2}+\epsilon}.\label{quantitative result at last}
\end{equation}
\end{corollary}
\begin{proof}
We just indicate the main steps: instead of taking a fixed parameter
$T_{0}$ as in the proof of Theorem \ref{deuring to hold}, we will take a sufficiently
large $T$. If one assumes that $\phi(\frac{r}{2}+iy)$ does not possess
a zero whenever $y\in[T,\,T+H]$, then the integrand on (\ref{Selberg Chowla in an integral form of writing})
does not possess a zero when $y\in[t-T-H,\,t-T]$. If we let $t$
vary on the interval $(T+\frac{H}{4},\,T+\frac{H}{2}]$ and take a
partition of the integral exactly as (\ref{decomposition of H into three integrals}),
we will now obtain a partition of the form 
\[
\left\{ \int_{\alpha H}^{\infty}+\int_{-\infty}^{-\alpha H}+\int_{-\alpha H}^{\alpha H}\right\} \,R_{\psi}\left(\frac{r}{2}+i\left(y-t\right)\right)\,R_{\phi}\left(\frac{r}{2}+i\left(y+t\right)\right)\,dy,
\]
where $\frac{1}{4}<\alpha\leq\frac{1}{2}$. Estimates of the first
two integrals will be similar to the already given (\ref{estimate auxiliar integrals})
and by the contradiction hypothesis, we are allowed to take the modulus
inside the third integral as in (\ref{contradiction hypothesissisis})
and then we can invoke Cauchy's Theorem to arrive at a lower bound. At this point, since
the amplitude of the interval considered for $\mathcal{A}_{3}(t)$
is $2\alpha H$, the quantity $H(T)$ needs to be chosen so that it
must be greater than the estimates for the integrals concerning the horizontal
segments in (\ref{another decomposition of integrals}). An immediate adaptation of the steps leading to (\ref{writing as Dirichlet series on the first proof})
will give a lower bound for $\mathcal{Z}_{2}(r+2it;\,a;\,\lambda)$
of the form 
\[
|\mathcal{Z}_{2}\left(r+2it;\,a;\,\lambda\right)|>B^{\prime}\,T^{-\frac{1}{2}}\,H,\,\,\,\,\,T+\frac{H}{4}<t\leq T+\frac{H}{2},
\]
for some positive $B^{\prime}$. An inductive reasoning similar to
the one given in Claim \ref{inductive claim} will also yield
\begin{equation}
|\mathcal{Z}_{2^{k+1}}\left(2^{k}r+2it;\,a;\,\lambda\right)|>B^{\prime}\,T^{-\frac{2^{k+1}-1}{2}}\,H^{2^{k+1}-1},\,\,\,\,\,T+\frac{H}{2^{k+2}}<t\leq T+\frac{H}{2^{k+1}},\label{bound contradictory for H(T)}
\end{equation}
valid for all $k\in\mathbb{N}_{0}$. If our Dirichlet series is such
that (\ref{Condition Narrow critical strip}) holds, the choice $H(T)=T^{\sigma_{a}+\frac{1-r}{2}+\epsilon}$
satisfies all of our requirements. In fact, under the hypothesis (\ref{Condition Narrow critical strip}), (\ref{bound contradictory for H(T)})
will contradict the Phragm\'en-Lindel\"of bound (\ref{Phragmen contradictory here}) for $\mathcal{Z}_{2^{k+1}}$. 
\end{proof}


\begin{remark}
Although (\ref{quantitative result at last}) is valid
in the general conditions of Corollary \ref{narrow abcissa corollary}, in most of the particular cases direct adaptations
of the method described in Theorem \ref{deuring to hold} and Corollary \ref{quantitative Classic} yield far better estimates than (\ref{quantitative result at last}) (see Examples \ref{hardy and jacobi} and \ref{epstein example} below). This happens when we have enough information regarding the Epstein zeta functions associated with $\phi(s)$. Once we have such information, the condition (\ref{Condition Narrow critical strip}) used in Corollary \ref{narrow abcissa corollary} may be discarded, since we can find a lower bound for $\mathcal{A}_{3}(T)$ by invoking Ramachandra's Theorem \cite{balasubramanian}. 
\end{remark}

\begin{center}\section{A class of examples and identities arising from the Selberg-Chowla formula} \label{section 5} \end{center}

\begin{example} \label{Watsonformula as Example} (A Generalized Watson formula and its corollaries)
On Taylor's paper regarding the functional equation for the Epstein zeta function
and its Selberg-Chowla representation, formula (\ref{Selberg Chowla Formula})
was attributed to Kober {[}\cite{taylor_epstein}, p. 182{]}. Although Kober \cite{kober_epstein} hadn't the explicit purpose of studying the analytic
continuation of the Epstein zeta function, he did study a generalization of a previous formula due to Watson {[}\cite{watson_reciprocal}, eq. (4){]}. By using the Poisson summation formula, G. N. Watson proved that, for $\text{Re}(s)>0$ and $x>0$, the following
identity holds
\begin{equation}
\sum_{n=1}^{\infty}n^{s}K_{s}\left(2\pi nx\right)+\frac{1}{4}(\pi x)^{-s}\Gamma(s)-\frac{\sqrt{\pi}}{4}(\pi x)^{-s-1}\Gamma\left(s+\frac{1}{2}\right)=\frac{\sqrt{\pi}}{2x}\left(\frac{x}{\pi}\right)^{s+1}\Gamma\left(s+\frac{1}{2}\right)\,\sum_{n=1}^{\infty}\frac{1}{(n^{2}+x^{2})^{s+\frac{1}{2}}}.\label{WATSON FORMULA FORMULA FORMULA}
\end{equation}

In this first example, we look at generalizations of (\ref{WATSON FORMULA FORMULA FORMULA}). These generalizations also include Kober's formulas given in \cite{kober_epstein}. Looking at our derivation of formulas (\ref{First Selberg Chowla even non diagonal})
and (\ref{First Selberg Chowla odd}) on Theorem
\ref{selberg-chowla non diagonal theorem}, it is not hard
to write generalizations of Watson's formula for Bochner Dirichlet
series. Let $\phi(s)=\sum a(n)\lambda_{n}^{-s}$, $\text{Re}(s)>\sigma_{a}$,
be a Bochner Dirichlet series belonging to the class $\mathcal{B}$.
If $\nu$ is a complex number satisfying $\text{Re}(\nu)>\frac{\sigma_{a}}{2}$,
consider the infinite series
\[
\mathcal{S}(\nu):=\sum_{n\in\mathbb{Z}}\frac{a(n)}{\left(a\lambda_{n}^{2}+b\lambda_{n}+c\right)^{\nu}},
\]
where the denominator is positive definite, i.e., $-d=4ac-b^{2}>0$,
$a>0$. By looking at the proof of Theorem \ref{selberg-chowla non diagonal theorem}, if we perform a summation
with respect to only one variable, we obtain the following formulas
\begin{align}
-2\phi(0)\,c^{-\nu}+\sum_{n\neq0}\frac{a(n)}{\left(a\lambda_{n}^{2}+b\lambda_{n}+c\right)^{\nu}} & =\rho\sqrt{\pi}\,\frac{\Gamma\left(\nu-\frac{1}{2}\right)}{\Gamma(\nu)}a^{-\nu}k^{1-2\nu}\nonumber \\
+\frac{4\,k^{\frac{1}{2}-\nu}a^{-\nu}}{\Gamma(\nu)}\,\sum_{n=1}^{\infty}b(n) & \,\cos\left(\frac{b}{a}\mu_{n}\right)\,\mu_{n}^{\nu-\frac{1}{2}}\,\,K_{\nu-\frac{1}{2}}\left(2k\,\mu_{n}\right),\,\,\,\,\,\text{\ensuremath{a}}(n)\,\,\text{even},\label{Watson formula for even}
\end{align}

\begin{equation}
\sum_{n\neq0}\frac{a(n)}{\left(a\lambda_{n}^{2}+b\lambda_{n}+c\right)^{\nu}}=-\frac{4\,k^{\frac{1}{2}-\nu}a^{-\nu}}{\Gamma(\nu)}\,\sum_{n=1}^{\infty}b(n)\,\sin\left(\frac{b}{a}\mu_{n}\right)\,\mu_{n}^{\nu-\frac{1}{2}}\,\,K_{\nu-\frac{1}{2}}\left(2k\,\mu_{n}\right),\,\,\,\,\,a(n)\,\,\text{odd}.\label{Watson formula for oddddd}
\end{equation}
The first of these identities generalize Watson's formula and its
extension by Kober [\cite{kober_epstein}, p. 614, eq. (2b)]. Furthermore, the Selberg-Chowla formula
(\ref{First Selberg Chowla Berndt Case for analogues Redirus}) given
in Remark \ref{remark odd odd Epstein} gives another formula of Watson-type
\begin{align}
\sum_{n\neq0}\frac{a(n)\,\lambda_{n}}{\left(a\lambda_{n}^{2}+b\lambda_{n}+c\right)^{\nu}}=\,\frac{4k^{\frac{3}{2}-\nu}a^{-\nu}}{\Gamma(\nu)}\,\sum_{n=1}^{\infty}\,b(n)\mu_{n}^{\nu-\frac{1}{2}}\,\nonumber \\
\times\left\{ \cos\left(\frac{b}{a}\,\mu_{n}\right)\,K_{\frac{3}{2}-\nu}(2k\,\mu_{n})+\frac{2b}{\sqrt{|d|}}\,\sin\left(\frac{b}{a}\,\mu_{n}\right)\,K_{\frac{1}{2}-\nu}(2k\,\mu_{n})\right\}, & \,\,\,\,\,a(n)\,\,\,\text{odd}.\label{another version watson fooooormula for odddd}
\end{align}

Let $\chi$ be a non-principal and primitive Dirichlet character modulo
$\ell$: since $\zeta(s)$ and $L(s,\,\chi)$ both satisfy functional
equations of Bochner-type (\ref{This is the first Bochner ever}), we can replace the sequences appearing
in (\ref{Watson formula for even}) and (\ref{another version watson fooooormula for odddd})
by those of these Dirichlet series. For example, if $\phi(s)=\zeta(s)$
and $b=0$, (\ref{Watson formula for even}) reduces to an equivalent
form of Watson's formula (\ref{WATSON FORMULA FORMULA FORMULA}). Under the same hypothesis of diagonal $Q(x,y)$, if we replace, respectively, $a(n)$ by $\chi(n)$
in (\ref{Watson formula for even}) and (\ref{another version watson fooooormula for odddd})
we obtain, respectively, formulas (2.9) and (2.10) given in [\cite{koshliakov_ramanujan_character},
p. 3, Thm 2.1]. 

\medskip{}

There are several known applications of the classical Watson formula. In his Lost Notebook \cite{Guinand_Ramanujan}, Ramanujan
recorded the following formula. Let $s\in\mathbb{C}$ and assume that
$\alpha,\beta>0$ are such that $\alpha\beta=\pi^{2}$. Then the following
identity holds
\begin{align}
\sqrt{\alpha}\sum_{n=1}^{\infty}\sigma_{-s}(n)\,n^{s/2}K_{\frac{s}{2}}(2n\alpha) & -\sqrt{\beta}\sum_{n=1}^{\infty}\sigma_{-s}(n)\,n^{s/2}K_{\frac{s}{2}}(2n\beta)=\nonumber \\
\frac{1}{4}\Gamma\left(-\frac{s}{2}\right)\zeta(-s)\left\{ \beta^{(1+s)/2}-\alpha^{(1+s)/2}\right\}  & +\frac{1}{4}\Gamma\left(\frac{s}{2}\right)\zeta(s)\left\{ \beta^{(1-s)/2}-\alpha^{(1-s)/2}\right\} .\label{Guinand Summation Equivalent}
\end{align}

This formula was rediscovered by Guinand in 1955 \cite{guinand_rapidly_convergent, Guinand_Ramanujan}, who
employed Watson's formula (\ref{WATSON FORMULA FORMULA FORMULA}) in order to derive it. 

By using a similar idea, we can also provide a generalization of (\ref{Guinand Summation Equivalent}) with the generalized divisor function introduced in this paper. To do so, for $\xi>0$, recall the analogue
of the Epstein zeta function $\mathcal{Z}_{2}(s;\,a_{1},\,a_{2};\,\lambda,\,\xi\lambda^{\prime})$
given in (\ref{Diagonal Epstein with parameter xssssi}). By using
the first Selberg-Chowla representation (\ref{Again First-1}), we
have deduced (\ref{Diagonal Parameter 1st}). By using the same substitutions
there mentioned and using the second Selberg-Chowla formula (\ref{Again Second-1}),
we can analogously derive
\begin{align}
\Gamma(s)\,\sum_{m,n\neq0}^{\infty}\frac{a_{1}(m)\,a_{2}(n)}{\left(\lambda_{m}+\lambda_{n}^{\prime}\,\xi\right)^{s}} & =-\xi^{-s}\,\phi_{1}(0)\,\Gamma(s)\,\phi_{2}(s)+\xi^{-r_{2}}\rho_{2}\,\Gamma(r_{2})\,\Gamma(s-r_{2})\,\phi_{1}\left(s-r_{2}\right)\nonumber \\
 & +\,2\,\xi^{-\frac{r_{2}+s}{2}}\,\sum_{m,n=1}^{\infty}b_{2}(m)\,a_{1}(n)\,\left(\frac{\mu_{m}^{\prime}}{\lambda_{n}}\right)^{\frac{s-r_{2}}{2}}\,K_{r_{2}-s}\left(2\,\sqrt{\frac{\mu_{m}^{\prime}\lambda_{n}}{\xi}}\right).\label{Diagonal parameter 2nd}
\end{align}

Viewed in this way, the Selberg-Chowla formula not only provides the
analytic continuation of a Dirichlet series $\mathcal{Z}_{2}(s,;\,a_{1},\,a_{2};\,\lambda,\,\xi \lambda^{\prime})$
depending on $\xi>0$ but it also gives a reformulation of this continuation in terms of a formula of modular type. A comparison between (\ref{Diagonal Parameter 1st}) and 
(\ref{Diagonal parameter 2nd}) yields a generalization of the Ramanujan-Guinand
formula (\ref{Guinand Summation Equivalent}), valid for all $s\in\mathbb{C}$ by Proposition \ref{proposition 2.1}
and the assumptions on the class $\mathcal{A}$.  

Stated in a clear way, let $\phi_{i}(s)$, $i=1,\,2$ represent the
pair of Dirichlet series given in (\ref{Dirichlet Series in Definition}) which
also satisfy Hecke's functional equation (\ref{Hecke Dirichlet series Functional})
and belong to the class $\mathcal{A}$. Denote also by $\rho_{i}$
the residue that $\phi_{i}$ has at $s=r_{i}$. Then, for all $s\in\mathbb{C}$
and $x>0$, the following generalization of Guinand's formula holds
\begin{align}
2x^{r_{1}-s}\sum_{j=1}^{\infty}\sigma_{s-r_{1}}\left(\nu_{j};b_{1},a_{2}\right)\nu_{j}^{\frac{r_{1}-s}{2}}\,K_{r_{1}-s}\left(2x\,\sqrt{\nu_{j}}\right)-\frac{2}{x^{r_{2}+s}}\sum_{j=1}^{\infty}\sigma_{s-r_{2}}\left(\nu_{j}^{\prime};b_{2},a_{1}\right)\,\nu_{j}^{\prime\frac{r_{2}-s}{2}}K_{r_{2}-s}\left(\frac{2\sqrt{\nu_{j}^{\prime}}}{x}\right)\nonumber \\
=\Gamma(s)\left\{ \phi_{2}(0)\phi_{1}(s)-\phi_{1}(0)\phi_{2}(s)x^{-2s}\right\} +\rho_{2}\,x^{-2r_{2}}\,\Gamma(r_{2})\,\Gamma(s-r_{2})\,\phi_{1}(s-r_{2})-\rho_{1}x^{2r_{1}-2s}\Gamma(r_{1})\,\Gamma(s-r_{1})\,\phi_{2}(s-r_{1}),\label{Guinand avec Parameter}
\end{align}
where $\sigma_{z}$ denotes the generalized weighted divisor function
described by (\ref{On a generalized Divisor Function}) \footnote{Berndt [\cite{dirichletserisIII}, p. 343, eq. (9.1)] established other
analogues of Guinand's formula. But these formulas of Berndt are only valid for arithmetical functions whose Dirichlet series satisfy Hecke's functional equation, while in general the arithmetical function of divisor type $c_{z}(j):=\sigma_{z}(\nu_{j};\,b_{1},\,a_{2})$
does not give rise to a Dirichlet series satisfying (\ref{Hecke Dirichlet series Functional}). In fact, the Dirichlet series attached to this general divisor function
is actually $\varphi_{z}(s):=\psi_{1}(s-z)\,\phi_{2}(s)$
and the assumption that $\psi_{1}$ and $\phi_{2}$ satisfy Hecke's
functional equation (\ref{Hecke Dirichlet series Functional}) allows
to write a functional equation (involving a product of two $\Gamma-$functions)
for $\varphi_{z}(s)$. Thus, (\ref{Guinand avec Parameter}) is actually equivalent to this
functional equation that $\varphi_{z}(s)$ possesses (see also
[\cite{dirichletserisIII}, p. 324] and [\cite{Guinand_Ramanujan}, pp. 39-40]).}. 


\medskip{}

From (\ref{Guinand avec Parameter}) it is possible to derive generalizations
of the classical Koshliakov formula \cite{koshliakov_Voronoi}, by mimicking the steps given
in \cite{Guinand_Ramanujan}. Writing the Laurent series for $\phi_{i}(s)$
around $s=r_{i}=r$ as 
\begin{equation}
\phi_{i}(s)=\frac{\rho_{i}}{s-r}+\rho_{0,i}+O\left(s-r\right),\,\,\,\,\,i=1,\,2,\label{meromorphic phii}
\end{equation}
assuming the same condition as the ones given in Ramanujan-Guinand's
formula (\ref{Guinand avec Parameter}) and taking $r_{1}=r_{2}=r$
we have, for any $x>0$, that the following formula of Koshliakov
type holds
\begin{align}
2\,\sum_{j=1}^{\infty}d\left(\nu_{j};\,b_{1},\,a_{2}\right)\,K_{0}\left(2x\,\sqrt{\nu_{j}}\right)-2\,x^{-2r}\,\sum_{j=1}^{\infty}d\left(\nu_{j}^{\prime};\,b_{2},\,a_{1}\right)\,K_{0}\left(\frac{2}{x}\,\sqrt{\nu_{j}^{\prime}}\right)\nonumber \\
=\Gamma(r)\,\left\{ \phi_{2}(0)\rho_{0,1}+\phi_{2}(0)\,\rho_{1}\,\frac{\Gamma^{\prime}(r)}{\Gamma(r)}-\rho_{1}\,\left\{ \rho_{2}^{\star}\Gamma^{\prime}(r)+\rho_{0,2}^{\star}\Gamma(r)\right\} +2\,\log(x)\,\phi_{2}(0)\,\rho_{1}\right\} +\label{a simpler some how Koshhhh}\\
+\,x^{-2r}\Gamma(r)\left\{ \left\{ \rho_{1}^{\star}\Gamma^{\prime}(r)+\rho_{0,1}^{\star}\Gamma(r)\right\} \rho_{2}-\phi_{1}(0)\,\rho_{0,2}-\phi_{1}(0)\,\rho_{2}\,\frac{\Gamma^{\prime}(r)}{\Gamma(r)}+2\log(x)\,\phi_{1}(0)\,\rho_{2}\right\},\nonumber 
\end{align}
where $\rho_{i}^{\star}$ and $\rho_{0,i}^{\star}$ denote the coefficients
of the meromorphic expansion (\ref{meromorphic phii}) for $\psi_{i}(s)$ and $d(\nu_{j};\,b_{1}, a_{2})$ is defined by (\ref{generalized Divisor function}). 
We obtain (\ref{a simpler some how Koshhhh}) by letting $s\rightarrow r$ in (\ref{Guinand avec Parameter}) and
using the meromorphic expansions for $\Gamma(s)$ and $\phi_{i}(s)$
around $s=0$, (\ref{Expansion Gamma 0}) and (\ref{Taylor phi1}), as well as (\ref{meromorphic phii}). 


\bigskip{}

Special cases of (\ref{Guinand avec Parameter}) and (\ref{a simpler some how Koshhhh})
are also possible to obtain for Dirichlet series in the Bochner class. Another important corollary of Koshliakov's formula (or, in fact,
a reformulation of it for $r=\frac{1}{2}$ \cite{Koshliakov_Soni, relations_equivalent})
is due to K. Soni. According to an account in \cite{Guinand_Ramanujan}, Soni's formula
already appeared in Ramanujan's lost notebook, written more than fifty years earlier (see also a multidimensional analogue of Soni's formula in [\cite{Yakubovich_Voronoi}, p. 813, eq. (3.12)]). 

\medskip{}


Indeed, taking $r=\frac{1}{2},\,\frac{3}{2}$ in (\ref{a simpler some how Koshhhh})
and mimicking the steps in \cite{Koshliakov_Soni, Guinand_Ramanujan}, it is possible to generalize
Soni's formula {[}\cite{Koshliakov_Soni}, p. 543, eq. (4){]} and some character analogues appearing in \cite{koshliakov_ramanujan_character}. For the case where
$r=\frac{1}{2}$ and assuming convergence of all the infinite series involved, we have the following formula 
\begin{align}
2\,\sum_{j=1}^{\infty}d\left(\nu_{j};\,b_{1},\,a_{2}\right)\,\frac{\log\left(\alpha/2\sqrt{\nu_{j}}\right)}{\alpha^{2}-4\nu_{j}}-\frac{2\pi}{\alpha}\,\sum_{j=1}^{\infty}d\left(\nu_{j}^{\prime};\,b_{2},\,a_{1}\right)\,K_{0}\left(2\sqrt{2}\nu_{j}^{\prime\frac{1}{4}}\alpha^{\frac{1}{2}}\right)\nonumber \\
=\frac{\sqrt{\pi}}{\alpha^{2}}\left\{ \phi_{2}(0)\,\rho_{0,1}-\sqrt{\pi}\rho_{1}\rho_{0,2}^{\star}-2\rho_{1}\phi_{2}(0)\left(2\gamma+\log(2\alpha)\right)\right\} \label{Soni Statement final version}\\
+\frac{\pi^{\frac{3}{2}}}{2\alpha}\,\left\{ -\phi_{1}(0)\rho_{0,2}+\sqrt{\pi}\rho_{2}\rho_{0,1}^{\star}+2\rho_{2}\phi_{1}(0)\,\log\left(\frac{2}{\alpha}\right)\right\},\nonumber 
\end{align}
where $\alpha\in\mathbb{R}_{+}\setminus\{2\sqrt{\nu_{j}}\}_{j\in\mathbb{N}}$.
\end{example}

\begin{example}[A character analogue of the Epstein zeta function] \label{example 5.3} In this example we study two particular cases of the non-diagonal
Epstein zeta function considered in section \ref{section 3}. Let $Q$ denote an
integral, positive definite and binary Quadratic form and let $\chi_{1}$
and $\chi_{2}$ be two non-principal, primitive Dirichlet characters
(with the same parity) having moduli $\ell_{1}$ and $\ell_{2}$.
We introduce the character analogue of the Epstein zeta function in
the following form
\begin{equation}
Z_{2}(s,\,Q,\,\chi_{1},\,\chi_{2})=\sum_{(m,n)\neq(0,0)}\frac{\chi_{1}(m)\,\chi_{2}(n)}{Q(m,n)^{s}},\,\,\,\,\text{Re}(s)>1.\label{V-53}
\end{equation}

Let us assume first that $\chi_{1}$ and $\chi_{2}$ are even Dirichlet
characters and let us take in Theorem \ref{selberg-chowla non diagonal theorem} $\phi_{1}(s)=\pi^{-\frac{s}{2}}L(s,\,\chi_{1})$
and $\phi_{2}(s)=\pi^{-\frac{s}{2}}L(s,\,\chi_{2})$. The functional
equation for $L(s,\,\chi)$ is given by \cite{Davenport}
\begin{equation}
\left(\frac{\pi}{\ell}\right)^{-\frac{s}{2}}\Gamma\left(\frac{s}{2}\right)\,L(s,\chi)=\frac{G(\chi)}{\sqrt{\ell}}\,\left(\frac{\pi}{\ell}\right)^{-\frac{1-s}{2}}\Gamma\left(\frac{1-s}{2}\right)\,L(1-s,\,\overline{\chi}),\label{Functional equation L function even}
\end{equation}
where $G(\chi)$ denotes the Gauss sum associated to the character
$\chi$,
\begin{equation}
G(\chi)=\sum_{r=1}^{\ell-1}\chi(r)\,e^{2\pi ir/\ell}.\label{Gaussssss suuuuum once more}
\end{equation}

Thus $\phi_{1}(s),\,\phi_{2}(s)\in\mathcal{B}$ with $\delta=0$ and
in order to apply our Theorem \ref{selberg-chowla non diagonal theorem} we need to perform the following
substitutions in (\ref{First Selberg Chowla even non diagonal}): 
\begin{equation}
a_{j}(n)=\chi_{j}(n),\,\,\,\,b_{j}(n)=\frac{G(\chi_{j})}{\ell_{j}}\,\overline{\chi}_{j}(n),\,\,\,\,\lambda_{n,j}=\sqrt{\pi}n,\,\,\,\mu_{n,j}=\frac{\sqrt{\pi}n}{\ell_{j}}.\label{first set of substitutions}
\end{equation}

\medskip{}

Analogously, if $\chi_{1}$ and $\chi_{2}$ are odd and primitive
Dirichlet characters having moduli $\ell_{1}$ and $\ell_{2}$, let
us take $\phi_{1}(s)=\pi^{-\frac{s}{2}}L(s,\,\chi_{1})$ and $\phi_{2}(s)=\pi^{-\frac{s}{2}}L(s,\,\chi_{2})$.
The functional equation for a Dirichlet $L-$function attached to
an odd and primitive character $\chi$ modulo $\ell$ reads \cite{Davenport}
\begin{equation}
\left(\frac{\pi}{\ell}\right)^{-\frac{s+1}{2}}\Gamma\left(\frac{s+1}{2}\right)\,L(s,\,\chi)=-\frac{iG(\chi)}{\sqrt{\ell}}\left(\frac{\pi}{\ell}\right)^{\frac{s-2}{2}}\,\Gamma\left(\frac{2-s}{2}\right)\,L(1-s,\,\overline{\chi}),\label{Functional equation Dirichlet L function oddddd}
\end{equation}
which implies that $\phi_{1},\,\phi_{2}\in\mathcal{B}$ with $\delta=1$. In order to use Theorem \ref{selberg-chowla non diagonal theorem}, we also take the substitutions 
\begin{equation}
a_{j}(n)=\chi_{j}(n),\,\,\,\,b_{j}(n)=-\frac{i\,G(\chi_{j})}{\ell_{j}}\,\overline{\chi}_{j}(n),\,\,\,\,\lambda_{n,j}=\sqrt{\pi}n,\,\,\,\mu_{n,j}=\frac{\sqrt{\pi}n}{\ell_{j}}.\label{substitutions oddddddd second seeeet}
\end{equation}

\bigskip{}

After a simple application of Theorem \ref{selberg-chowla non diagonal theorem} and Corollary \ref{corollary 3.1}
and using the substitutions outlined in (\ref{first set of substitutions})
and (\ref{substitutions oddddddd second seeeet}), we see that $Z_{2}(s,\,Q,\,\chi_{1},\,\chi_{2})$
has a continuation to the complex plane as an entire function. 

Indeed, it possesses the following Selberg-Chowla formulas: 
\begin{enumerate}
\item If $\chi_{1}$ and $\chi_{2}$ are even, then an application of the
Selberg-Chowla formulas (\ref{First Selberg Chowla even non diagonal})
and (\ref{Second Selberg Chowla even non-diagonal}) yields 
\begin{equation}
Z_{2}(s,Q,\chi_{1},\chi_{2})= \frac{8\pi^{s}a^{-s}}{\Gamma(s)}G(\chi_{1})k^{\frac{1}{2}-s}\ell_{1}^{-\left(s+\frac{1}{2}\right)}
\sum_{n=1}^{\infty}\sigma_{1-2s}(n,\,\overline{\chi}_{1},\chi_{2}) n^{s-\frac{1}{2}}\cos\left(\pi\frac{b}{a}\,\frac{n}{\ell_{1}}\right)K_{\frac{1}{2}-s}\left(\frac{2\pi kn}{\ell_{1}}\right),\label{even characters even even}
\end{equation}
as well as 
\begin{equation}
Z_{2}(s,Q,\chi_{1},\chi_{2})= \frac{8\pi^{s}c^{-s}}{\Gamma(s)}G(\chi_{2})k^{\prime\frac{1}{2}-s}\ell_{2}^{-\left(s+\frac{1}{2}\right)}
\sum_{n=1}^{\infty}\sigma_{1-2s}(n,\overline{\chi}_{2},\chi_{1}) n^{s-\frac{1}{2}}\cos\left(\pi\frac{b}{c}\,\frac{n}{\ell_{2}}\right)K_{\frac{1}{2}-s}\left(\frac{2\pi k^{\prime}n}{\ell_{2}}\right).\label{second characters odd odd example}
\end{equation}
\item If $\chi_{1}$ and $\chi_{2}$ are odd and primitive, then we equivalently
have the Selberg-Chowla formula (from (\ref{First Selberg Chowla odd}))
\begin{equation}
Z_{2}(s,Q,\chi_{1},\chi_{2}) =-\frac{8i\pi^{s}a^{-s}}{\Gamma(s)}G(\chi_{1})k^{\frac{1}{2}-s}\ell_{1}^{-\left(s+\frac{1}{2}\right)}
\sum_{n=1}^{\infty}\sigma_{1-2s}(n,\overline{\chi}_{1},\chi_{2}) n^{s-\frac{1}{2}}\sin\left(\pi\frac{b}{a}\frac{n}{\ell_{1}}\right)K_{\frac{1}{2}-s}\left(\frac{2\pi kn}{\ell_{1}}\right),
\end{equation}
with an analogous second Selberg-Chowla formula taking place, being of the
form
\begin{equation}
Z_{2}(s,\,Q,\,\chi_{1},\,\chi_{2}) =-\frac{8i\,\pi^{s}c^{-s}}{\Gamma(s)}\,G(\chi_{2})\,k^{\prime\frac{1}{2}-s}\ell_{2}^{-\left(s+\frac{1}{2}\right)}
\sum_{n=1}^{\infty}\sigma_{1-2s}(n,\,\overline{\chi}_{2},\,\chi_{1}) n^{s-\frac{1}{2}}\,\sin\left(\pi\,\frac{b}{c}\,\frac{n}{\ell_{2}}\right)\,K_{\frac{1}{2}-s}\left(\frac{2\pi k^{\prime}\,n}{\ell_{2}}\right).\label{second characters odddddddd odddddd}
\end{equation}
\end{enumerate}
\medskip{}

In all of the previous representations, $\sigma_{z}(n,\,\chi_{a},\,\chi_{b})$
is a character analogue of the divisor function, defined by
\begin{equation}
\sigma_{z}(n,\,\chi_{a},\,\chi_{b})=\sum_{d|n}\chi_{a}(d)\,\chi_{b}\left(\frac{n}{d}\right)d^{z}.\label{character analogue double to double characters Epstein once more}
\end{equation}

Moreover, it follows from Corollary \ref{corollary 3.1} that $Z_{2}(s;\,Q;\,\chi_{1},\,\chi_{2})$ satisfies the functional equation 
\begin{equation}
G(\overline{\chi}_{1})\,G(\overline{\chi}_{2})\,\left(\frac{2\pi}{\sqrt{|d|}}\right)^{-s}\Gamma(s)\,Z_{2}\left(s,\,Q,\,\chi_{1},\,\chi_{2}\right)=\left(\frac{2\pi}{\sqrt{|d|}}\right)^{s-1}\Gamma(1-s)\,Z_{2}\left(1-s,\,Q_{\ell_{1},\ell_{2}}^{-1},\,\overline{\chi}_{1},\,\overline{\chi}_{2}\right),
\end{equation}
where $Q_{\ell_{1},\ell_{2}}^{-1}$ denotes the quadratic form $Q^{-1}\left(\frac{x}{\ell_{1}},\,\frac{y}{\ell_{2}}\right)$. 

\bigskip{}

By using the previous example, we can establish analogues of the Ramanujan-Guinand formula for the character analogue of the divisor function (\ref{character analogue double to double characters Epstein once more}). We only do this for the first case, i.e.,
assuming that the characters $\chi_{1}$ and $\chi_{2}$ are even.
It is effortless to see that the comparison of (\ref{even characters even even})
and (\ref{second characters odd odd example}) and the substitution
$a=c=1$, $b=0$, yield the double-weighted Ramanujan-Guinand's formula
\begin{equation}
\sqrt{\alpha}\,G(\chi_{1})\,\ell_{1}^{-\frac{s}{2}-1}\,\sum_{n=1}^{\infty}\sigma_{-s}(n,\,\overline{\chi}_{1},\,\chi_{2})\,n^{s/2}K_{\frac{s}{2}}\left(\frac{2\alpha n}{\ell_{1}}\right)=\sqrt{\beta}\,G(\chi_{2})\,\ell_{2}^{-\frac{s}{2}-1}\,\sum_{n=1}^{\infty}\sigma_{-s}(n,\,\overline{\chi}_{2},\,\chi_{1})\,n^{s/2}\,K_{\frac{s}{2}}\left(\frac{2\beta n}{\ell_{2}}\right),\label{Guinand two characters}
\end{equation}
where $\alpha,\,\beta>0$ are such that $\alpha\beta=\pi^{2}$. Note
that (\ref{Guinand two characters}) can also be obtained from (\ref{Guinand avec Parameter})
and by using the fact that $\pi^{-s}L(2s,\,\chi)\in\mathcal{A}$ and
satisfies Hecke's functional equation with $r=\frac{1}{2}$. 
\medskip{}

Taking the limit $s\rightarrow0$ in (\ref{Guinand two characters})
yields Koshliakov's formula for double characters. To write it in
the form (\ref{a simpler some how Koshhhh}), take the substitution
$\alpha=\pi\sqrt{\frac{\ell_{1}}{\ell_{2}}}\,x$ and replace $\overline{\chi_{1}}$
by $\chi_{1}$. Use also the relation for Gauss sums \cite{Davenport}, 
\begin{equation}
G(\chi)\,G(\overline{\chi})=\chi(-1)\,\ell.\label{formula ayoub for Gausssss summmms}
\end{equation}

Then a particular case of (\ref{a simpler some how Koshhhh}) is formulated as 
\begin{equation}
\sum_{n=1}^{\infty}d_{\chi_{1},\chi_{2}}(n)\,K_{0}\left(\frac{2\pi n\,x}{\sqrt{\ell_{1}\ell_{2}}}\right)=\frac{G(\chi_{1})\,G(\chi_{2})}{x\,\sqrt{\ell_{1}\ell_{2}}}\,\sum_{n=1}^{\infty}d_{\overline{\chi}_{1},\overline{\chi}_{2}}(n)\,K_{0}\left(\frac{2\pi n}{x\,\sqrt{\ell_{1}\ell_{2}}}\right),\,\,\,\,\chi_{1},\,\chi_{2}\,\,\text{even},\label{Koshliakov double appearing in berndt kosh}
\end{equation}
where $d_{\chi_{1},\chi_{2}}(n)$ represents (\ref{character analogue double to double characters Epstein once more})
at $z=0$, this is, the double-weighted divisor function, $\sum_{d|n}\chi_{1}(d)\,\chi_{2}(n/d)$.

\medskip{}

Identity (\ref{Koshliakov double appearing in berndt kosh}) is given
in {[}\cite{koshliakov_berndt}, p. 45, eq. (3.7.){]}. From (\ref{character analogue double to double characters Epstein once more}), one can also
derive character analogues of Soni's formula: applying (\ref{Soni Statement final version})
gives (after replacing also $\alpha$ by $2\pi\sqrt{\ell_{1}\ell_{2}}\,\alpha$)
\begin{equation}
\sum_{n=1}^{\infty}\,\frac{d_{\chi_{1},\chi_{2}}(n)\,\log\left(\frac{\ell_{1}\ell_{2}\alpha}{n}\right)}{(\ell_{1}\ell_{2}\alpha)^{2}-n^{2}}=\frac{2\pi^{2}G(\chi_{1})\,G(\chi_{2})}{\alpha\,\ell_{1}^{2}\,\ell_{2}^{2}}\,\sum_{n=1}^{\infty}d_{\overline{\chi}_{1},\overline{\chi}_{2}}(n)\,K_{0}\left(4\pi\sqrt{\alpha n}\right),\label{Soni double character reformulated}
\end{equation}
which is valid for every $\alpha\neq n/\ell_{1}\ell_{2}$, $n\in\mathbb{N}$. A particular case of (\ref{Soni double character reformulated}) where $\chi_{1}(n)=\chi_{2}(n)=\chi(n)$, $\ell_{1}=\ell_{2}=\ell$, 
has been observed in {[}\cite{koshliakov_ramanujan_character}, p. 5, eq. (3.8){]}
and is equivalent to 
\[
\sum_{n=1}^{\infty}\frac{\chi(n)\,d(n)\,\log\left(\frac{\alpha\ell^{2}}{n}\right)}{\alpha^{2}\ell^{4}-n^{2}}=\frac{2\pi^{2}G(\chi)^{2}}{\alpha\ell^{4}}\,\sum_{n=1}^{\infty}\overline{\chi}(n)\,d(n)\,K_{0}(4\pi\sqrt{\alpha n}).
\]

\medskip{}

Extensions of other formulas given in \cite{koshliakov_ramanujan_character}
are also possible. For instance, it is not hard to derive generalizations
of {[}\cite{koshliakov_ramanujan_character}, eq. (3.15), (3.16) and (3.19){]}
for the character analogue of the divisor function $d_{\chi_{1},\chi_{2}}(n)$. 

Some particular cases of (\ref{Guinand two characters})
include identities derived by Dixit \cite{dixit_RiemannXI} by an entirely
different method {[}\cite{dixit_RiemannXI}, p. 322, Theorem 1.5.,
eq. (1-15){]}, employing the symmetric properties of an integral involving
Riemann's $\Xi-$function and a character analogue of it. Our formula (\ref{Guinand two characters})
applied to the case $\chi_{1}=\chi_{2}=\chi$ yields $F(z,\,\alpha,\,\chi)=F(-z,\,\beta,\,\overline{\chi})$
in formula (1-15) of \cite{dixit_RiemannXI} \footnote{Although Dixit's results may seem to be of a different nature, the integral representation involving the Riemann
$\Xi-$function appearing in his paper {[}\cite{dixit_RiemannXI}, eq. (1-14){]}
can be generalized to our class of Dirichlet series. In fact, it
is a generalization of Dixit's integral representation that we employ
to study the zeros of the Dirichlet series in Theorem \ref{deuring to hold}, since we
can represent the entire function $H_{r_{i}}(s;\,b_{i};\,a_{i})$
in a similar form (c.f. eq. (\ref{Essssential to prove Deuring}) and (\ref{First Integral Representation})
in the previous section).}. 
\end{example}

\begin{example}[Hardy's Theorem and the 4-square Theorem] \label{hardy and jacobi}
Based upon Theorem \ref{deuring to hold}, this example gives a new proof of Hardy's
Theorem for $\zeta(s)$, which is curious as its conclusion is derived
from presumably independent properties of the arithmetical function
$r_{4}(n)$. As it is well known, $r_{4}(n)$ counts the number of
representations of $n$ as a sum of 4 squares with different signs
and different orders of the summands giving distinct representations.
Let us invoke Theorem \ref{deuring to hold} to $\phi(s)=\pi^{-s}\zeta(2s)$, which
satisfies Hecke's functional equation with $r=\frac{1}{2}$. If $\phi(s)$
does not have infinitely many zeros on the critical line $\text{Re}(s)=\frac{1}{2}$,
then all of its dyadic Epstein zeta functions (\ref{dyadic Epstein 2^k}) will only have finitely
many zeros at their critical lines. Consider $\mathcal{Z}_{4}(s)$:
it comes immediately from the definition of multidimensional Epstein
zeta function that, for $\text{Re}(s)>2$, 
\[
\mathcal{Z}_{4}(s)=\pi^{-s}\sum_{m_{1},...,m_{4}\neq0}^{\infty}\frac{1}{\left(m_{1}^{2}+m_{2}^{2}+m_{3}^{2}+m_{4}^{2}\right)^{s}}=\frac{\pi^{-s}}{4}\sum_{n=1}^{\infty}\frac{r_{4}(n)}{n^{s}}=\frac{\pi^{-s}}{4}\zeta_{4}(s).
\]

\medskip{}

Hence, according with our assumption, $\zeta_{4}(s)$ cannot have infinitely
many zeros at its critical line $\text{Re}(s)=1$. However, by the celebrated Jacobi's
four square theorem, $\zeta_{4}(s)$ is described
by 
\begin{equation}
\zeta_{4}(s)=8\left(1-2^{2-2s}\right)\zeta(s)\,\zeta(s-1).\label{Jacobi 4 square theorem}
\end{equation}

Note that the right-hand side of (\ref{Jacobi 4 square theorem})
has infinitely many zeros in the line $\text{Re}(s)=1$, all of them being of the form $1\pm\frac{\pi i}{\log(2)}\,k$, with $k\in\mathbb{N}$.
This means that $\mathcal{Z}_{4}(s)$ has infinitely many zeros on
the line $\text{Re}(s)=1$ and a contradiction is derived, proving Hardy's Theorem in its classical form \footnote{The novelty in this proof is the observation that a purely arithmetical
property such as (\ref{Jacobi 4 square theorem}) implies a deep analytic
theorem regarding the distribution of the zeros of the Riemann zeta
function. One may argue that since (\ref{Jacobi 4 square theorem})
admits proofs using modular forms  and a particular modular form (the $\theta$-function) is used in Hardy's proof of Hardy's Theorem \cite{hardy_note}, the reason should be this, as a look at a generalization of this argument in Example \ref{epstein example} shows. Despite this, the fact that there are elementary proofs of (\ref{Jacobi 4 square theorem}) available, such as in \cite{jacobi_4_square}, adds a certain curiosity to this connection.}. By using a formula due to Siegel \cite{Siegel_lectures_quadratic}, the argument presented here can be generalized to Epstein zeta functions attached to quadratic forms whose matrices have determinant 1.

\medskip{}

This observation can also be used to improve the quantitative result
given in Corollary \ref{quantitative Classic}. For the case where $\phi(s)=\pi^{-s}\zeta(2s)$
or $\pi^{-\frac{s-\delta}{2}}L(2s-\delta,\,\chi)$, for $\chi$ being
a primitive Dirichlet character modulo $\ell$, the result on Corollary \ref{quantitative Classic} gives $H(T)=T^{\frac{3}{4}+\epsilon}$, which is a result attributed
to Landau. The argument invoking the 4-square theorem
can actually improve this result for $\zeta(s)$ to $H(T)=T^{\frac{1}{2}+\epsilon}$, a result which was recorded for the first time in a paper by de La Vall\'ee Poussin \cite{fekete_zeros}. As it is clear from the brief sketch of the proof, the
condition $H(T)=T^{\sigma_{a}+\frac{1-r}{2}+\epsilon}$
is only necessary in order to (\ref{bound contradictory for H(T)}) contradict
the Phragm\'en-Lindel\"of estimates (\ref{Phragmen contradictory here}) combined with the assumption (\ref{Condition Narrow critical strip}).
In fact, for the case where $\phi(s)=\pi^{-s}\zeta(2s)$ it suffices
to assume that $H(T)=T^{\frac{1}{2}+\epsilon}$ until the point we
arrive at the condition (\ref{bound contradictory for H(T)}). In
this particular case, (\ref{bound contradictory for H(T)}) gives the inequality 
\begin{equation}
|\zeta_{4}(1+2it)|>B^{\prime}\,\left(\frac{H}{T^{\frac{1}{2}}}\right)^{3},\,\,\,\,\,\,B^{\prime}>0
\end{equation}
valid for every $T+\frac{H}{8}<t\leq T+\frac{H}{4}$. This obviously
contradicts (\ref{Jacobi 4 square theorem}), which implies that there
is always a zero of $\zeta_{4}(1+it)$ in any interval of length $\frac{\pi}{\log(2)}$. 

Other proofs of Hardy's Theorem for $\zeta(s)$ can be derived from
corollaries \ref{corollary the one with residue} and \ref{narrow abcissa corollary}, yielding even easier arguments. As it was proved by Dixit et al \cite{combinations_dixit}, $\phi(s)=\pi^{-2s}\,\zeta(2s)$ satisfies Corollary \ref{combination}. 

If $\chi$ is a primitive Dirichlet character modulo $\ell$, then
an application of Corollary \ref{combination} (or Corollary \ref{corollary the one with residue} together with remark \ref{fekete remark} for the case where
$\chi$ is even), shows that the function
\[
F(s,\,\chi):=\sum_{j=1}^{\infty}c_{j}\left(\frac{\pi}{\ell}\right)^{-(s+i\tau_{j})}\Gamma\left(\frac{s+i\tau_{j}+\delta}{2}\right)\,L(s+i\tau_{j};\,\chi),\,\,\,\,\delta\in\{0,\,1\},
\]
where $c_{j}\in\ell^{1}(\mathbb{R})$ and $\tau_{j}\in\ell^{\infty}(\mathbb{R})$,
has infinitely many zeros on the critical line $\text{Re}(s)=\frac{1}{2}$. 
\end{example}

\begin{example}[Character analogue for the divisor function and Dedekind zeta function]
Let $\chi$ be an odd and primitive Dirichlet character modulo $\ell$.
By the functional equations for $\zeta(s)$ and $L(s,\chi)$ (\ref{Functional equation Dirichlet L function oddddd}),
one can check that $\zeta(s)\,L(s,\,\chi)$ satisfies the functional
equation
\begin{equation}
\left(\frac{2\pi}{\sqrt{\ell}}\right)^{-s}\Gamma(s)\,\zeta(s)\,L(s,\chi)=-\frac{iG(\chi)}{\sqrt{\ell}}\,\left(\frac{2\pi}{\sqrt{\ell}}\right)^{-(1-s)}\Gamma(1-s)\,\zeta(1-s)\,L(1-s,\overline{\chi}).\label{Functional equation now Hecke for d chi}
\end{equation}

Let us write the associated Dirichlet series satisfying (\ref{Functional equation now Hecke for d chi}) as 
\[
\phi(s)=\zeta(s)\,L(s,\chi)=\sum_{n=1}^{\infty}\frac{d_{\chi}(n)}{n^{s}},\,\,\,\,\,\,\,\text{Re}(s)>1,
\]
where $d_{\chi}(n)$ denotes the character analogue of the divisor
function {[}\cite{koshliakov_berndt}, p. 42, eq. (3.1){]}, 
\[
d_{\chi}(n)=\sum_{d|n}\chi(d).
\]

If $\chi_{1}$ and $\chi_{2}$ denote two odd and primitive Dirichlet
characters modulo $\ell_{1}$ and $\ell_{2}$, we know by (\ref{Functional equation now Hecke for d chi})
that $\phi_{j}(s)=\zeta(s)\,L(s,\,\chi_{j})$ satisfies Hecke's functional
equation with $r_{j}=1$, $j=1,\,2$. Moreover, each $\phi_{i}(s)$ belongs to
the class $\mathcal{A}$ and their residues at $s=1$ are simply $\rho_{j}=\,L(1,\chi_{j})$.
We define the diagonal Epstein zeta function for $d_{\chi_{1}}(n)$
and $d_{\chi_{2}}(n)$ as the double Dirichlet series \begin{equation}
\mathcal{Z}_{2}\left(s;\,d_{\chi_{1}},\,d_{\chi_{2}}\right):=\sum_{m,n\neq0}^{\infty}\frac{d_{\chi_{1}}(m)\,d_{\chi_{2}}(n)}{(m+n)^{s}},\,\,\,\,\,\,\text{Re}(s)>2.\label{double dirichlet series with character}
\end{equation}

As in Example \ref{example 5.3}, in order to match the functional equation (\ref{Functional equation now Hecke for d chi}) with the formulations of Theorem \ref{theorem 2.1} and Corollary \ref{The Analytic Continuation},
we need to perform the following substitutions for $\phi_{j}(s)$, 
\[
a_{j}(n)=d_{\chi_{j}}(n),\,\,\,b_{j}(n)=-\frac{2\pi iG(\chi_{j})}{\ell_{j}}\,d_{\overline{\chi_{j}}}(n),\,\,\,\lambda_{n,j}=n,\,\,\,\mu_{n,j}=\frac{4\pi^{2}}{\ell_{j}}n.
\]

Using the first Selberg-Chowla formula (\ref{Again First-1}), we
obtain the following continuation for (\ref{double dirichlet series with character})
\begin{align*}
\Gamma(s)\,\mathcal{Z}_{2}\left(s;\,d_{\chi_{1}},\,d_{\chi_{2}}\right) & =\frac{1}{2}\,L(0,\,\chi_{2})\Gamma(s)\,\zeta(s)\,L(s,\,\chi_{1})+L(1,\chi_{1})\,\Gamma(s-1)\,L(s-1,\,\chi_{2})\\
-\frac{4\pi iG(\chi_{1})}{\ell_{1}^{\frac{s+1}{2}}} & \,(2\pi)^{s-1}\,\sum_{m,n=1}^{\infty}d_{\overline{\chi_{1}}}\,(m)\,d_{\chi_{2}}(n)\,\left(\frac{m}{n}\right)^{\frac{s-1}{2}}\,K_{s-1}\left(\frac{4\pi}{\ell_{1}}\sqrt{m\,n}\right).
\end{align*}
Analogously, (\ref{Again Second-1}) gives a second
Selberg-Chowla formula of the form
\begin{align*}
\Gamma(s)\,\mathcal{Z}_{2}\left(s;\,d_{\chi_{1}},\,d_{\chi_{2}}\right) & =\frac{1}{2}\,L(0,\,\chi_{1})\Gamma(s)\,\zeta(s)\,L(s,\,\chi_{2})+L(1,\chi_{2})\,\Gamma(s-1)\,L(s-1,\,\chi_{1})\\
-\frac{4\pi iG(\chi_{2})}{\ell_{2}^{\frac{s+1}{2}}} & \,(2\pi)^{s-1}\,\sum_{m,n=1}^{\infty}d_{\chi_{1}}(m)\,d_{\overline{\chi_{2}}}\,(n)\,\left(\frac{m}{n}\right)^{\frac{1-s}{2}}\,K_{s-1}\left(\frac{4\pi}{\ell_{2}}\sqrt{m\,n}\right).
\end{align*}

\medskip{}

Both formulas given above provide the analytic continuation of (\ref{double dirichlet series with character})
as a meromorphic function with a simple pole located at $s=2$ with residue
$L(1,\,\chi_{1})\cdot L(1,\,\chi_{2})$. The functional equation for
(\ref{double dirichlet series with character}) can be written as 
\begin{equation}
\left(\frac{2\pi}{\sqrt{\ell_{1}\ell_{2}}}\right)^{-s}\Gamma(s)\,\mathcal{Z}_{2}\left(s;\,d_{\chi_{1}},\,d_{\chi_{2}}\right)=-G(\chi_{1})\,G(\chi_{2})\,\left(\frac{2\pi}{\sqrt{\ell_{1}\ell_{2}}}\right)^{-(2-s)}\,\Gamma(2-s)\,\mathcal{Z}_{2}\left(2-s;\,\mathbb{I}_{\ell_{1},\ell_{2}};\,d_{\overline{\chi}_{1}},\,d_{\overline{\chi}_{2}}\right),\label{Functional equation for epstein divisor character}
\end{equation}
where $\mathcal{Z}_{2}\left(s;\,\mathbb{I}_{\ell_{1},\ell_{2}};\,d_{\overline{\chi}_{1}},\,d_{\overline{\chi}_{2}}\right)$
is described as the double Dirichlet series 
\[
\mathcal{Z}_{2}\left(s;\,\mathbb{I}_{\ell_{1},\ell_{2}};\,d_{\overline{\chi}_{1}},\,d_{\overline{\chi}_{2}}\right):=\sum_{m,n\neq0}^{\infty}\frac{d_{\overline{\chi}_{1}}(m)\,d_{\overline{\chi}_{2}}(n)}{\left(\ell_{2}m+\ell_{1}n\right)^{s}},\,\,\,\,\,\text{Re}(s)>2.
\]

It is also possible to construct analogues of Guinand's formula via
the above Selberg-Chowla formulas for $\mathcal{Z}_{2}(s,\,d_{\chi_{1}},\,d_{\chi_{2}})$.
Analogues of Koshliakov's formula can be also obtained, although we
need some additional computations of the values of $L^{\prime}(1,\,\chi)$ \cite{denninger}.

\medskip{}
In the same spirit, for $i=1,2$, let $F_{i}(n)$ denote, as usual, the number of
integral ideals of norm $n$ in an imaginary quadratic number field
with discriminant $D_{i}$, $K_{i}=\mathbb{Q}\left(\sqrt{-D_{i}}\right)$. Then it is well-known that the associated Dirichlet
series is the classical Dedekind zeta function
\begin{equation}
\zeta_{K_{i}}(s)=\sum_{n=1}^{\infty}\frac{F_{i}(n)}{n^{s}},\,\,\,\,\,\text{Re}(s)>1,\label{Dedekind imaginary quadratic field}
\end{equation}
which satisfies Hecke's functional equation
\begin{equation}
\left(\frac{2\pi}{\sqrt{D_{i}}}\right)^{-s}\Gamma(s)\,\zeta_{K_{i}}(s)=\left(\frac{2\pi}{\sqrt{D_{i}}}\right)^{s-1}\Gamma(1-s)\,\zeta_{K_{i}}(1-s).\label{functional equation Dedekind}
\end{equation}

Since $\zeta_{K_{i}}(s)$ has a continuation to the entire
complex plane as an analytic function except at a simple pole located at
$s=1$, we see that, for $i=1,\,2$, $\zeta_{K_{i}}(s)\in\mathcal{A}$
with $r_{i}=1$. Moreover, the residue of $\zeta_{K_{i}}(s)$ at $s=1$
can be explicitly computed as
\[
\rho_{i}=\frac{2\pi\,h(K_{i})R(K_{i})}{\sqrt{D_{i}}\,w(K_{i})},
\]
where $h(K_{i}),$ $R(K_{i})$ and $w(K_{i})$ denote, respectively,
the class number of $K_{i}$, the regulator of $K_{i}$ and the number
of roots of unity in $K_{i}$. We now take $\phi_{i}(s):=\zeta_{K_{i}}(s)$:
in order to apply Corollary \ref{The Analytic Continuation}, we need to take the following substitutions
\[
a_{i}(n)=F_{i}(n),\,\,\,b_{i}(n)=\frac{2\pi}{\sqrt{D_{i}}}F_{i}(n),\,\,\,\lambda_{n,i}=n,\,\,\,\mu_{n}=\frac{4\pi^{2}}{D_{i}}n.
\]

By the conditions on the class $\mathcal{A}$, we have $\phi_{i}(0)=-{h(K_{i})\,R(K_{i})}/{w(K_{i})}$,
so that we take $F_{i}(0):=h(K_{i})\,R(K_{i})/w(K_{i})$. We now construct
the diagonal Epstein zeta function attached to $K_{1}$ and $K_{2}$ as follows
\begin{equation}
\mathcal{Z}_{2}\left(s;\,K_{1},\,K_{2}\right):=\sum_{m,n\neq0}^{\infty}\frac{F_{1}(m)\,F_{2}(n)}{\left(m+n\right)^{s}},\,\,\,\,\,\,\,\text{Re}(s)>2.\label{Epstein zeta function associated to Dedekind}
\end{equation}

Under the substitutions above mentioned, we have from the Selberg-Chowla
formula (\ref{Again First-1}) that the following identity holds for (\ref{Epstein zeta function associated to Dedekind}) 
\begin{align*}
\Gamma(s)\,\mathcal{Z}_{2}\left(s;\,K_{1},\,K_{2}\right) & =\frac{h(K_{2})\,R(K_{2})}{w(K_{2})}\,\Gamma(s)\,\zeta_{K_{1}}(s)+\frac{2\pi\,h(K_{1})R(K_{1})}{\sqrt{D_{1}}\,w(K_{1})}\,\Gamma(s-1)\,\zeta_{K_{2}}\left(s-1\right)\\
 & +2\,\left(\frac{2\pi}{\sqrt{D_{1}}}\right)^{s}\,\sum_{m,n=1}^{\infty}F_{1}(m)\,F_{2}(n)\,\left(\frac{m}{n}\right)^{\frac{s-1}{2}}\,K_{s-1}\left(\frac{4\pi\sqrt{m\,n}}{D_{1}}\right),
\end{align*}
giving the analytic continuation of $\mathcal{Z}_{2}(s;\,K_{1},\,K_{2})$
as a meromorphic function with one simple pole located at $s=2$ whose
residue is precisely 
\[
\text{Res}_{s=2}\mathcal{Z}_{2}(s;\,K_{1},\,K_{2})=\frac{4\pi^{2}}{\sqrt{D_{1}D_{2}}}\,\frac{h(K_{1})h(K_{2})\,R(K_{1})R(K_{2})}{w(K_{1})w(K_{2})}.
\]

Moreover, if we write a second Selberg-Chowla for (\ref{Epstein zeta function associated to Dedekind}), we can deduce that it satisfies the following functional equation: 
\[
\left(\frac{2\pi}{\sqrt{D_{1}D_{2}}}\right)^{-s}\Gamma(s)\,\mathcal{Z}_{2}\left(s;\,K_{1},\,K_{2}\right)=\sqrt{D_{1}D_{2}}\,\left(\frac{2\pi}{\sqrt{D_{1}D_{2}}}\right)^{-(2-s)}\Gamma(2-s)\,\mathcal{Z}_{2}\left(s;\,\mathbb{I}_{D_{1},D_{2}},\,K_{1},\,K_{2}\right),
\]
where, just as in (\ref{Functional equation for epstein divisor character}), $\mathcal{Z}_{2}\left(s;\,\mathbb{I}_{D_{1},D_{2}},\,K_{1},\,K_{2}\right)$
denotes the Dirichlet series 
\begin{equation}
\mathcal{Z}_{2}\left(s;\,\mathbb{I}_{D_{1},D_{2}},\,K_{1},\,K_{2}\right):=\sum_{m,n\neq0}^{\infty}\frac{F_{1}(m)\,F_{2}(n)}{\left(D_{2}m+D_{1}n\right)^{s}},\,\,\,\,\,\text{Re}(s)>2.
\end{equation}

If, for the quadratic fields $K_{1}$ and $K_{2}$, we define an analogue of the divisor function, $\sigma_{\nu}(n;\,K_{1},\,K_{2})$, 
in the following way
\begin{equation}
\sigma_{\nu}(n;\,K_{1},\,K_{2})=\sum_{d|n}F_{1}(d)\,F_{2}\left(\frac{n}{d}\right)\,d^{\nu},\label{divisor function quadratic field}
\end{equation}
we have that, for any $x>0$, the identity of Ramanujan-Guinand type (\ref{Guinand avec Parameter}) holds 
\begin{equation}
\left(\frac{2\pi}{\sqrt{D_{1}}}\right)^{s}\,x^{1-s}\,\sum_{n=1}^{\infty}\,\sigma_{s-1}\left(n;\,K_{1},\,K_{2}\right)\,n^{\frac{1-s}{2}}\,K_{s-1}\left(\frac{4\pi\sqrt{n}}{\sqrt{D_{1}}}\,x\right)\nonumber
\end{equation}
\begin{equation}
 -\left(\frac{2\pi}{\sqrt{D_{2}}}\right)^{s}\,x^{-1-s}\sum_{m,n=1}^{\infty}\sigma_{s-1}\left(n;\,K_{2},\,K_{1}\right)\,n^{\frac{1-s}{2}}\, K_{s-1}\left(\frac{4\pi\sqrt{n}}{\sqrt{D_{2}}\,x}\right)\nonumber 
 \end{equation}
 \begin{equation}
=\frac{h(K_{1})\,R(K_{1})\,x^{-2s}}{2\,w(K_{1})}\left\{ \Gamma(s)\,\zeta_{K_{2}}(s)-\frac{2\pi\,x^{2}}{\sqrt{D_{1}}}\Gamma(s-1)\,\zeta_{K_{2}}(s-1)\right\} \, \nonumber
\end{equation}
 \begin{equation}
+\frac{h(K_{2})\,R(K_{2})}{2\,w(K_{2})}\left\{ \frac{2\pi}{\sqrt{D_{2}}\,x^{2}}\,\Gamma(s-1)\,\zeta_{K_{1}}(s-1)-\Gamma(s)\,\zeta_{K_{1}}(s)\right\}.\label{Guinand Ramanujan Dedeekiniiind}
\end{equation}



From the Kronecker limit formula for imaginary quadratic fields, it is also possible to take the limit
$s\rightarrow1$ on (\ref{Guinand Ramanujan Dedeekiniiind}) in order to
derive analogues of Koshliakov's formula for the divisor function
(\ref{divisor function quadratic field}). One dimensional analogues
of (\ref{Guinand Ramanujan Dedeekiniiind}) are also given in \cite{berndt_preprint}. 

\end{example}

\begin{example}[Dirichlet series attached to Eisenstein Series and Cusp Forms]\label{example cusp forms}
Let $k\geq3$ be an odd positive integer and consider the divisor
function $\sigma_{k}(n)$. We know that the associated Dirichlet series
is given by the product of zeta functions 
\begin{equation}
\zeta(s)\,\zeta(s-k)=\sum_{n=1}^{\infty}\frac{\sigma_{k}(n)}{n^{s}},\,\,\,\,\,\text{Re}(s)>k+1,\label{divisor k definition}
\end{equation}
and that $\phi(s)=(2\pi)^{-s}\,\zeta(s)\,\zeta(s-k)$ obeys to the Hecke's functional equation {[}\cite{arithmetical identities}, p. 17, eq. (57){]}
\begin{equation}
(2\pi)^{-s}\,\Gamma(s)\,\zeta(s)\,\zeta(s-k)=(-1)^{\frac{k+1}{2}}(2\pi)^{-(k+1-s)}\Gamma(k+1-s)\,\zeta(k+1-s)\,\zeta(1-s).\label{functional equation Divisor}
\end{equation}
Moreover, since $\zeta(1-k)=0$, we know that $\phi\in\mathcal{A}$, once we take the convention $\sigma_{k}(0):=-B_{k+1}/2k+2$. 

Let $k_{1},\,k_{2}\geq3$ be two odd positive integers and consider
two Dirichlet series $\phi_{i}(s)$, $i=1,\,2$, given by $\phi_{i}(s)=(2\pi)^{-s}\zeta(s)\,\zeta(s-k_{i})$.
We construct the diagonal Epstein zeta function (\ref{Epstein as double series})
as 
\begin{equation}
\mathcal{Z}_{2}(s;\,k_{1},\,k_{2})=(2\pi)^{-s}\sum_{m,n\neq0}^{\infty}\frac{\sigma_{k_{1}}(m)\,\sigma_{k_{2}}(n)}{(m+n)^{s}}:=(2\pi)^{-s}\,Z_{2}(s;\,k_{1},\,k_{2}),\,\,\,\,\,\,\text{Re}(s)>2\,\max\left\{ k_{1},\,k_{2}\right\} +2. \label{Epstein for divisor once more}
\end{equation}

In order to apply directly Theorem \ref{theorem 2.1} and its corollary \ref{The Analytic Continuation}, we take the simple
substitutions
\begin{equation}
{a}_{i}(n)=\sigma_{k_{i}}(n),\,\,\,\ b_{i}(n)=(-1)^{\frac{k_{i}+1}{2}}\sigma_{k_{i}}(n),\,\,\,\lambda_{n,i}=\mu_{n,i}=2\pi n.\label{substitutions divisor k}
\end{equation}

Using Euler's formula for $\zeta(2n),$ it is very simple to find
that
\begin{equation}
\rho_{i}=\frac{(-1)^{\frac{k_{i}-1}{2}}}{2(k_{i}+1)!}\,B_{k_{i}+1},\,\,\,\,\phi_{i}(0):=\zeta(0)\,\zeta(-k_{i})=\frac{B_{k_{i}+1}}{2k_{i}+2}\label{computation residue and value at zero once more sigma k}
\end{equation}
where $\left(B_{n}\right)_{n\in\mathbb{N}}$ denotes the sequence
of Bernoulli numbers. Applying these computations to the first Selberg-Chowla
formula on the class $\mathcal{A}$ (\ref{Again First-1}), we arrive
at the representation 
\begin{align*}
(2\pi)^{-s}\,\Gamma(s)\,Z_{2}(s;\,k_{1},\,k_{2}) & =-\frac{B_{k_{2}+1}\,(2\pi)^{-s}}{2k_{2}+2}\,\Gamma(s)\,\zeta(s)\,\zeta(s-k_{1})+\frac{(-1)^{\frac{k_{1}-1}{2}}\,B_{k_{1}+1}\,(2\pi)^{-s+k_{1}+1}}{2k_{1}+2}\,\Gamma(s-k_{1}-1)\times\\
\times\zeta(s-k_{1}-1) & \,\zeta(s-k_{1}-k_{2}-1)+2\,\sum_{m,n=1}^{\infty}(-1)^{\frac{k_{1}+1}{2}}\sigma_{k_{1}}(m)\,\sigma_{k_{2}}(n)\,\left(\frac{m}{n}\right)^{\frac{s-k_{1}-1}{2}}\,K_{s-k_{1}-1}\left(4\pi\sqrt{m\,n}\right).
\end{align*}
which gives the analytic continuation of $Z_{2}(s;\,k_{1},\,k_{2})$
as a meromorphic function with a simple pole located at $s=k_{1}+k_{2}+2$
whose residue is given by 
\begin{equation}
\text{Res}_{s=k_{1}+k_{2}+2}\,\,Z_{2}(s;\,k_{1},\,k_{2})=\frac{(-1)^{\frac{k_{1}+k_{2}}{2}-1}(2\pi)^{k_{1}+k_{2}+2}}{4\,(k_{1}+1)\,(k_{2}+1)}\,\frac{B_{k_{1}+1}\times B_{k_{2}+1}}{(k_{1}+k_{2}+1)!}.\label{Residue sigma_k}
\end{equation}

The second Selberg-Chowla formula can be analogously written and an
application of Corollary \ref{The Analytic Continuation} shows that $Z_{2}(s;\,k_{1},\,k_{2})$
must satisfy the functional equation 
\begin{equation}
(2\pi)^{-s}\,\Gamma(s)\,Z_{2}(s;\,k_{1},\,k_{2})=(-1)^{\frac{k_{1}+k_{2}+2}{2}}(2\pi)^{-(k_{1}+k_{2}+2-s)}\,\Gamma(k_{1}+k_{2}+2-s)\,Z_{2}(k_{1}+k_{2}+2-s;\,k_{1},\,k_{2}).\label{as it should be expected odd k}
\end{equation}


An application of Theorem \ref{deuring to hold} together with the absence of zeros (at the critical line) of the Dirichlet series
given in (\ref{divisor k definition}), allows to prove some lower
bounds (at the critical line) for some Epstein zeta functions. In {[}\cite{Hecke_middle_line}, Beispiel 11, p. 95{]} it
is constructed an octonary quadratic form whose Epstein zeta function
does not possess a single zero on the critical line $\text{Re}(s)=2$.
Indeed, if $\mathbb{I}_{n}$ denotes the $n-$dimensional identity
matrix and $\mathcal{Q}(x_{1},...,x_{8})$ is the octonary quadratic form
represented by the following matrix
\[
\mathcal{Q}=\left(\begin{array}{cc}
2\cdot\mathbb{I}_{4} & \mathcal{S}\\
-\mathcal{S} & 2\cdot\mathbb{I}_{4}
\end{array}\right)\,\,\,\,\,\,\text{with }\,\,\,\,\,\,\mathcal{S}:=\left(\begin{array}{cccc}
0 & 1 & 1 & 1\\
-1 & 0 & -1 & 1\\
-1 & 1 & 0 & -1\\
-1 & -1 & 1 & 0
\end{array}\right),
\]
then its associated Epstein zeta function, $Z_{8}(s,\,\mathcal{Q})$,
can be written explicitly as 
\begin{equation}
Z_{8}(s,\,\mathcal{Q})=240\cdot2^{-s}\,\zeta(s)\,\zeta(s-3),\label{Hecke's construction}
\end{equation}
which obviously does not possess any zero at the critical line $\text{Re}(s)=2$
\footnote{note that this violation of the Riemann hypothesis is not implied
by the Selberg-Chowla formula in the sense of Corollary 4.1 and that
$Z_{8}(s,\,\mathcal{Q})$ possesses an Euler product, not falling
under the class of Epstein zeta functions studied by Davenport and
Heilbronn \cite{davenport heilbronn}.}. Thus, if we construct the Quadratic form in sixteen dimensions,
with a matrix representation of the form
\[
\mathcal{P}=\left(\begin{array}{cc}
\mathcal{Q} & \mathbf{0}_{8}\\
\mathbf{0}_{8} & \mathcal{Q}
\end{array}\right), 
\]
we see by a simple application of item 2. of Theorem \ref{deuring to hold} that the
order of the Epstein zeta function satisfies a lower bound of the
form $|Z_{16}\left(4+it;\,\mathcal{P}\right)|>C\,|t|^{\frac{7}{2}},$
$C>0$, whenever $|t|$ is taken sufficiently large \footnote{By the functional equation for the Riemann zeta-function and Stirling's
formula, we have that $|Z_{8}(2+it,\,\mathcal{Q})|\gg|t|^{\frac{3}{2}}$,
$|t|\gg1$. Applying this to the Selberg-Chowla formula (\ref{Essssential to prove Deuring}) and proceeding
just as in the proof of Claim \ref{inductive claim}, we can derive the lower bound
$|Z_{16}(4+it,\,\mathcal{P})|\gg|t|^{\frac{7}{2}}$, $|t|\gg1$.}. This proves that, at most, $Z_{16}(s,\,\mathcal{P})$ has a finite number of zeros in the line $\text{Re}(s)=4$. Hence, besides violating
the Riemann hypothesis quite badly, as the construction (\ref{Hecke's construction})
shows, Epstein zeta functions with sufficiently high dimension do
not satisfy the Lindel\"of hypothesis. 

\bigskip{}

As a complementary case of the previous Dirichlet series, let $f$ be a cusp form with weight $k\geq12$ (with $k$ always being
an even integer) for the full modular group and $L(s,\,a)$ the associated
$L-$function, 
\begin{equation}
L(s,\,a)=\sum_{n=1}^{\infty}\frac{a(n)}{n^{s}},\label{Modular group Dirichlet series}
\end{equation}
where the arithmetical functions $a(n)$ represent the Fourier coefficients
of $f$. Then (\ref{Modular group Dirichlet series}) is known to
be absolutely convergent in the half-plane $\text{Re}(s)>\frac{k+1}{2}$
and can be analitically continued to an entire function {[}\cite{apostol_book}, Theorem 6.20{]}. Furthermore, it obeys to Hecke's functional
equation 
\begin{equation}
(2\pi)^{-s}\Gamma(s)\,L(s,\,a)=(-1)^{k/2}(2\pi)^{-(k-s)}\Gamma(k-s)\,L(k-s,\,a).\label{Hekce's functional equation L functions cusp forms}
\end{equation}

Consider two holomorphic cusp forms $f_{1}$ and $f_{2}$ having
weights $k_{1}$ and $k_{2}$ respectively and the pairs of Dirichlet
series $\phi_{i}(s)=(2\pi)^{-s}L(s,\,a_{i})$ and $\psi_{i}(s)=(-1)^{k_{i}/2}(2\pi)^{-s}L(s,\,a_{i})$, $i=1,\,2$, associated with them. Here, the substitutions analogous
to (\ref{first set of substitutions}) are obvious to consider. Since each $\phi_{i}(s)$ is entire,
it follows from Corollary \ref{The Analytic Continuation} that its (diagonal) Epstein zeta function given by 
\begin{equation}
\mathcal{Z}_{2}(s;\,f_{1},\,f_{2})=\,\left(2\pi\right)^{-s}\sum_{m,n\neq0}^{\infty}\frac{a_{1}(m)\,a_{2}(n)}{(m+n)^{s}}=(2\pi)^{-s}Z_{2}(s;\,f_{1},\,f_{2}),\,\,\,\,\,\,\,\,\text{Re}(s)>\max\left\{ k_{1},\,k_{2}\right\} +1,\label{Diagonal Epstein cusp forms}
\end{equation}
is also entire and its continuation satisfies the Selberg-Chowla formulas
\[
(2\pi)^{-s}\,\Gamma(s)\,Z_{2}\left(s;\,f_{1},\,f_{2}\right)=2\,(-1)^{\frac{k_{1}}{2}}\,\sum_{m,n=1}^{\infty}a_{1}(m)\,a_{2}(n)\,\left(\frac{m}{n}\right)^{\frac{s-k_{1}}{2}}\,K_{k_{1}-s}\left(4\pi\sqrt{m\,n}\right)
\]
and 
\[
(2\pi)^{-s}\,\Gamma(s)\,Z_{2}\left(s;\,f_{1},\,f_{2}\right)=2\,(-1)^{\frac{k_{2}}{2}}\,\sum_{m,n=1}^{\infty}a_{1}(m)\,a_{2}(n)\,\left(\frac{m}{n}\right)^{\frac{k_{2}-s}{2}}\,K_{k_{2}-s}\left(4\pi\sqrt{m\,n}\right).
\]

It follows also from Corollary \ref{The Analytic Continuation} that $Z_{2}\left(s;\,f_{1},\,f_{2}\right)$
satisfies the functional equation 
\begin{equation}
(2\pi)^{-s}\,\Gamma(s)\,Z_{2}\left(s;\,f_{1},\,f_{2}\right)=(-1)^{\frac{k_{1}+k_{2}}{2}}\,(2\pi)^{-(k_{1}+k_{2}-s)}\,\Gamma\left(k_{1}+k_{2}-s\right)\,Z_{2}\left(k_{1}+k_{2}-s;\,f_{1},\,f_{2}\right).\label{functional equation cusp cusp double}
\end{equation}

Due to the fact that $\phi_{1}$ and $\phi_{2}$ are entire, it is
now very easy to establish analogues of formulas of Ramanujan-Guinand and Koshliakov type.
In particular, for the case $k_{1}=k_{2}=12$, $\phi_{1}(s)=\phi_{2}(s)=\phi(s)$, 
we could consider the Dirichlet series associated to the Ramanujan
$\tau-$function, 
\[
\phi(s):=(2\pi)^{-s}\,L(s,\,\tau)=(2\pi)^{-s}\,\sum_{n=1}^{\infty}\frac{\tau(n)}{n^{s}},\,\,\,\,\,\,\text{Re}(s)>\frac{13}{2}
\]
and the resulting diagonal Epstein zeta function (\ref{Diagonal Epstein cusp forms})
would satisfy Hecke's functional equation with $r:=k_{1}+k_{2}=24$.
For this case, a curious reflection formula\footnote{A different proof of (\ref{Koshliakov Ramanujan Tau}) is given in
{[}\cite{dirichletserisI}, p. 357, Example 2{]}. This result is obtained from
a generalization of Bochner's modular relation (\ref{Bochner Modular relation at intro}) for Dirichlet series
satisfying the functional equation (\ref{functional equation multi})
with $\Delta(s)=\Gamma^{N}(s)$.} of Koshliakov type takes
place in the form 
\begin{equation}
\sum_{j=1}^{\infty}d_{\tau}\left(n\right)\,K_{0}\left(4\pi\,\sqrt{n}\,x\right)=x^{-24}\,\sum_{j=1}^{\infty}d_{\tau}(n)\,K_{0}\left(\frac{4\pi\sqrt{n}}{x}\right),\,\,\,\,\,\,x>0,\label{Koshliakov Ramanujan Tau}
\end{equation}
where $d_{\tau}(n)$ denotes the divisor function associated with
$\tau(n)$, 
\[
d_{\tau}(n)=\sum_{d|n}\tau(d)\,\tau\left(\frac{n}{d}\right).
\]
\medskip{}
Now assume that $a_{i}(n)$, $i=1,\,2$, is real. By (\ref{functional equation cusp cusp double}),
we know that $Z_{2}\left(s;\,f_{1},\,f_{2}\right)$ is a real and
entire Hecke Dirichlet series with signature $\left(1,\,k_{1}+k_{2},\,(-1)^{\frac{k_{1}+k_{2}}{2}}\right)$. From a classical application of Hecke's theory {[}\cite{Hecke_middle_line}, p.
79, Satz 3{]}, the fact that the signature parameter $\lambda=1$ of $Z_{2}\left(s;\,f_{1},\,f_{2}\right)$ satisfies $0<\lambda<2$,  then one has that $\Theta_{2}(e^{i(\frac{\pi}{2}-\epsilon)};\,f_{1},\,f_{2})=O(\epsilon^{-\beta})$,
with $\beta<\frac{k_{1}+k_{2}+1}{2}$. Thus, according with Remark
\ref{remark Hecke condition}, $Z_{2}\left(s;\,f_{1},\,f_{2}\right)$ must have infinitely
many zeros on the critical line $\text{Re}(s)=\frac{k_{1}+k_{2}}{2}$.
In particular, this shows that $Z_{2}(s;\,f,\,f)$ has infinitely
many zeros at the critical line $\text{Re}(s)=k$ and so, by Theorem
\ref{deuring to hold}, $L(s,\,a)$ has infinitely many zeros at the
critical line $\text{Re}(s)=\frac{k}{2}$. Since the $L$-functions (\ref{Modular group Dirichlet series}) admit Selberg-Chowla expansions (see [\cite{suzuki}, p. 485, eq. (1.24)]), it would be also interesting to study more analogues of Hardy's Theorem for the Dirichlet series composing these .   
\end{example}

\begin{example}[The Epstein zeta function] \label{epstein example} 
Kober [\cite{kober_zeros}, p. 5] established the condition given in Corollary \ref{corollary on theta} for a class
of Epstein zeta functions attached to binary and positive definite
quadratic forms whenever  $a/\sqrt{|d|}$ is irrational. 
For the case where $a/\sqrt{|d|}$ is rational, then by using a set of
conditions similar to that given in Corollary \ref{corollary the one with residue}, we are only left
with a finite set of Epstein zeta functions for which we need to
verify Hardy's Theorem (see [\cite{kober_zeros}, p. 8]). In all of them, a
similar argument mimicking the one using Jacobi's 4-square Theorem
would suffice, but we would have to know apriori a formula for the
representation of a given integer as the particular quaternary forms involved. 

Since the theta function attached to $Z_{2}(s,\,Q)$ satisfies (\ref{contradiction hypothesis onceee more})
whenever $a/\sqrt{|d|}\notin\mathbb{Q}$, it is now possible
to derive from Corollary \ref{combination} that, for any $c_{j}\in\ell^{1}(\mathbb{R})$
and $\tau_{j}\in\ell^{\infty}(\mathbb{R})$, the infinite combination 
\[
F_{Q}(s)=\sum_{j=1}^{\infty}\,c_{j}\,\left(\frac{2\pi}{\sqrt{|d|}}\right)^{-s-i\tau_{j}}\Gamma(s+i\tau_{j})\,Z_{2}(s+i\tau_{j},\,Q)
\]
has infinitely many zeros on the critical line $\text{Re}(s)=\frac{1}{2}$. 

A still larger subclass of Epstein zeta functions satisfy the conditions of Corollary \ref{corollary the one with residue} {[}\cite{kober_zeros}, p. 5{]}, as well as Dedekind zeta functions attached to imaginary
quadratic fields [\cite{berndt_zeros_(ii)}, p. 689]. 


\medskip{}

\bigskip{}
By virtue of their functional equations (see [\cite{terras_epstein}, p. 481, Thm 2] and eq. (\ref{functional equation of Epstein multi multi multi}) below) and remark \ref{Ramachandra Remark}, it is also possible to derive Hardy's Theorem for certain multidimensional Epstein zeta functions satisfying the symmetric conditions given in Theorem \ref{deuring to hold}. Assume that $Q$ is a $n\times n$ matrix attached to a real positive-definite
Quadratic form and let $s$ be a complex number such that $\text{Re}(s)>\frac{n}{2}$.
Then the $n-$dimensional Epstein's $\zeta-$function is defined as the series $Z_{n}\left(s,\,Q\right)=\sum_{\mathbf{m}\in\mathbb{Z}^{n}\setminus0}\,Q(\mathbf{m})^{-s}$,
where $Q(\mathbf{x})=\mathbf{x}^{T}Q\,\mathbf{x}$ and $\mathbf{m}\in\mathbb{Z}^{n}\setminus\{0\}$.\footnote{Note now that this notation slightly differs from the conventions
given in the section \ref{section 3}: here, the Dirichlet series $Z_{n}(s,\,Q)$
is attached to the matrix representation of the quadratic form $Q$
so that $Z_{n}(s,\,Q^{-1})$ represents the Epstein zeta function
attached to Quadratic form having as matrix $Q^{-1}$.}. It is well-known that $Z_{n}(s,Q)$ can be
extended to the complex plane as a meromorphic function with a simple
pole at $s=n/2$ and with residue given by $|Q|^{-\frac{1}{2}}\,\pi^{n/2}/\Gamma(n/2)$,
where $|Q|=\text{det}(Q)$. Moreover, it satisfies Hecke's functional equation
\begin{equation}
\pi^{-s}\Gamma(s)\,Z_{n}\left(s,\,Q\right)=|Q|^{-\frac{1}{2}}\pi^{-\left(\frac{n}{2}-s\right)}\,\Gamma\left(\frac{n}{2}-s\right)\,Z_{n}\left(\frac{n}{2}-s,\,Q^{-1}\right),\label{functional equation of Epstein multi multi multi}
\end{equation}
where $|Q|$ is the determinant of $Q$. Note that, from (\ref{functional equation of Epstein multi multi multi}), $\pi^{-s}\Gamma(s)\,Z_{n}(s,\,Q)$ is real valued on the critical line $r=\frac{n}{2}$ whenever $Q$ is unimodular.

Let $Q_{1}$ and $Q_{2}$ denote, respectively, two matrices
representing two positive-definite quadratic forms in $\mathbb{R}^{m_{1}}$
and $\mathbb{R}^{m_{2}}$ with $m_{1}\leq m_{2}$. We shall apply Theorem \ref{theorem 2.1} with $\phi_{i}(s)=\pi^{-s}\,Z_{m_{i}}(s,\,Q_{i})$ and $\psi_{i}(s)=|Q|^{-\frac{1}{2}}\,\pi^{-s}Z_{n}\left(s,\,Q^{-1}\right)$. By definition of (\ref{Epstein as double series}), the diagonal Epstein zeta function attached
to $\phi_{1}(s)$ and $\phi_{2}(s)$ is 
\[
\mathcal{Z}_{2}\left(s;\,Q_{1},\,Q_{2}\right)=\pi^{-s}\,\sum_{\mathbf{m}\in\mathbb{Z}^{m_{1}+m_{2}}\setminus0}\frac{1}{Q_{1,2}^{(2)}(\mathbf{m})^{s}},\,\,\,\,\text{Re}(s)>\frac{m_{1}+m_{2}}{2}, 
\]
where $Q_{1,2}^{(2)}$ is the $m_{1}+m_{2}-$dimensional quadratic
form represented by the matrix
\begin{equation}
Q_{1,2}^{(2)}=\left(\begin{array}{cc}
Q_{1} & \mathbf{0}\\
\mathbf{0} & Q_{2}
\end{array}\right).\label{diagonal quadratic form}
\end{equation}

A straightforward application of Theorem \ref{theorem 2.1} gives the following
Selberg-Chowla formula (c.f. [\cite{terras_epstein}, p. 480, eq. (2.3)])
\begin{align}
\pi^{-s}\Gamma(s)\,Z_{m_{1}+m_{2}}\left(s,\,Q_{1,2}\right) & =\pi^{-s}\,\Gamma(s)\,Z_{m_{1}}\left(s,\,Q_{1}\right)+|Q_{1}|^{-\frac{1}{2}}\pi^{-s+\frac{m_{1}}{2}}\,\Gamma\left(s-\frac{m_{1}}{2}\right)\,Z_{m_{2}}\left(s-\frac{m_{1}}{2},\,Q_{2}\right)\nonumber \\
+2\,|Q_{1}|^{-\frac{1}{2}} & \,\sum_{\mathbf{m}_{1},\mathbf{m}_{2}}\,\left(\frac{Q_{1}^{-1}(\mathbf{m}_{1})}{Q_{2}\left(\mathbf{m}_{2}\right)}\right)^{\frac{s}{2}-\frac{m_{1}}{4}}\,K_{\frac{m_{1}}{2}-s}\left(2\pi\sqrt{Q_{1}^{-1}(\mathbf{m}_{1})\,Q_{2}\left(\mathbf{m}_{2}\right)}\right),\label{Selberg Chowla diagonal Multidimensional}
\end{align}
where $\mathbf{m}_{1}\in\mathbb{Z}^{m_{1}}\setminus\{0\}$ and $\mathbf{m}_{2}\in\mathbb{Z}^{m_{2}}\setminus\{0\}$.

Using (\ref{Selberg Chowla diagonal Multidimensional}) and the formalism
of Theorem \ref{deuring to hold}, we are ready to prove Hardy's Theorem for the multidimensional
Epstein zeta functions $Z_{k}(s,\,Q$). Assume that $k\geq3$ (the
case $k=2$ is already mentioned above) and let $Q$ be a quadratic
form attached to a matrix with $|Q|=\text{det}(Q)=1$ and having the genus 
of the identity matrix $\mathbb{I}_{k}$, $\gamma_{k}$. Let $\phi(s)=\pi^{-s}Z_{k}(s,\,Q)$:
then the four-dimensional Epstein zeta function (\ref{dyadic Epstein 2^k}) constructed from $\phi$, $\mathcal{Z}_{4}(s;\,\cdot)$, is actually $\pi^{-s}Z_{4k}(s;\,Q^{(4)})$, where $Q^{(4)}$ denotes the $4k-$dimensional matrix 
\[
Q^{(4)}=\left(\begin{array}{cc}
Q^{(2)} & \mathbf{0}\\
\mathbf{0} & Q^{(2)}
\end{array}\right),
\]
with $Q^{(2)}$ being essentially (\ref{diagonal quadratic form}) with $Q_{1}=Q_{2}=Q$. 

It is clear that $Q^{(4)}$ has the same genus, $\gamma_{4k}$, as 
$\mathbb{I}_{4k}$. If $Q_{1},$...,$Q_{h}$ denote the representatives
of all different classes of $\gamma_{4k}$ and $\text{Aut}(Q)$ the automorphism group of $Q$, we can consider the Eisenstein series
\begin{equation}
E(z;\,\gamma_{4k})=\frac{\sum_{\ell=1}^{h}\theta_{Q_{\ell}}(z)/|\text{Aut}(Q_{\ell})|}{\sum_{\ell=1}^{h}1/|\text{Aut}(Q_{\ell})|}\label{Eisenstein final}
\end{equation}
and attach to it a Dirichlet series $Z(s;\,\gamma_{4k})$. It is well-known \cite{Siegel_lectures_quadratic, Siegel_Contributions} that this Dirichlet series is of the form
\[
Z(s;\,\gamma_{4k})=c_{k}\,\left\{ 1-2^{-s}+(-1)^{k}\left(2^{2k-2s}-2^{-s}\right)\right\} \,\zeta(s)\,\zeta(s+1-2k),
\]
where $c_{k}>0$.
It is clear from the expression on the braces that $Z(s;\,\gamma_{4k})$ has infinitely many zeros on the line $\text{Re}(s)=k$, all of them of the form
\[
\begin{cases}
s=k\pm\frac{i\pi n}{\log(2)},\,\,\,n\in\mathbb{N} & k\,\,\,\text{odd}\\
s=k+\frac{i}{\log(2)}\left\{ 2n\pi\pm\arctan\left(\sqrt{2^{2k}-1}\right)\right\} ,\,\,\,n\in\mathbb{Z} & k\,\,\,\text{even.}
\end{cases}
\]

Since $\theta_{Q^{(4)}}(z)-E(z;\,\gamma_{4k})$ is a cusp form \cite{Siegel_Contributions, Siegel_lectures_quadratic},
the Dirichlet series attached to it, $L\left(s;\,S_{4k}\right)$,
satisfies the convex estimate  $L\left(k+it;\,S_{4k}\right)=O(|t|^{\frac{1}{2}+\epsilon}),\,\,\,\,|t|\rightarrow \infty$.
If we assume that $Z_{k}(s,\,Q)$ has only a finite number of zeros
on the line $\text{Re}(s)=\frac{k}{4}$, then from (4.37) we must
have that $|Z_{4k}\left(k+it,\,Q^{(4)}\right)|>C\,|t|^{\frac{3}{2}}$
for $|t|$ sufficiently large. Since 
\[
Z_{4k}\left(k+it;\,Q^{(4)}\right)=Z\left(k+it;\,\gamma_{4k}\right)+O\left(|t|^{\frac{1}{2}+\epsilon}\right),
\]
we have that $|Z\left(k+it;\,\gamma_{4k}\right)|>C^{\prime}\,|t|^{\frac{3}{2}}$
for some positive $C^{\prime}$ and $|t|$ large enough. This is absurd as $Z(s;\,\gamma_{4k})$
possesses infinitely many zeros on the line $\text{Re}(s)=k$. 

Adapting the argument outlined in Example \ref{hardy and jacobi}, we can also show the existence of a zero of $Z_{k}(s,\,Q)$ in any interval of the form $[T,\,T+T^{\frac{1}{2}+\epsilon}]$ for a sufficiently large $T$. Although far from being as sharp as Siegel's results in \cite{Siegel_Contributions} (valid for $k\geq 12$), it seems that this method gives a quantitative estimate not covered by the methods employed in \cite{Siegel_Contributions}. Siegel writes {[}\cite{Siegel_Contributions},
p. 363{]} `` ...\textit{ it is possible to discuss the remaining cases for
$3<k<12$, but this requires some numerical computations and we omit
it. For the function $\zeta_{3}(s)$, however, our method does not
lead to any result}''. Since our method, when applied to the situation described above, gives $N_{0}(T,\zeta_{3}(s))\gg T^{\frac{1}{2}-\epsilon}$, $\epsilon>0$, this result seems, as far as we know, to be the first quantitative estimate regarding the number of zeros of $\zeta_{3}(s)$ on the line $\text{Re}(s)=\frac{3}{4}$. 

\end{example}
\medskip{}

\medskip{}
\textit{Disclosure Statement.} The authors declare that they have no conflict of interest. 

\medskip{}

\textit{Acknowledgements.} The first and second authors were partially supported by CMUP (UIBD/MAT/00144/2020), which is financed by national funds through FCT – Fundação para a Ciência e a
Tecnologia (Portugal). The first author was also supported by the grant 2020.07359.BD, under the FCT PhD Program UC|UP MATH PhD Program.

The authors are also grateful to Professor Bruce C. Berndt for sending them the paper \cite{berndt_preprint}. 

A great amount of the work presented in this paper was revised and discussed with our friend and Professor Jos\'e Carlos Petronilho. To him our eternal gratitude. 
\newpage{}

\end{document}